\documentclass[11pt]{amsart}

\newlength{\myhmargin}  
\usepackage[textheight=574pt, textwidth=445pt, marginparwidth=50pt, centering]{geometry}

\usepackage{amsmath,amssymb,amsthm,amsfonts,amscd,flafter,
graphicx,verbatim,pinlabel,mathrsfs}
\usepackage[all]{xy}
\usepackage{epstopdf}
\usepackage[colorlinks=false]{hyperref}
\usepackage[all]{hypcap}
\AtBeginDocument{\addtocontents{toc}{\protect\setlength{\parskip}{0pt}}}

\title{A contact invariant in sutured monopole homology}

\author[John A. Baldwin]{John A. Baldwin}
\address{Department of Mathematics \\ Boston College}
\email{john.baldwin@bc.edu}

\author[Steven Sivek]{Steven Sivek}
\address{Department of Mathematics \\ Princeton University}
\email{ssivek@math.princeton.edu}

\thanks{JAB was partially supported by NSF grant DMS-1104688.  SS was partially supported by NSF postdoctoral fellowship DMS-1204387.}

\def\R{{\mathbb{R}}}

\newcommand\hf{\widehat{HF}}
\newcommand\hfp{HF^+}

\newcommand\zt{\mathbb{F}_2}
\newcommand\zz{\mathbb{Z}}

\newcommand\Sc{\text{Spin}^c}
\newcommand\spc{\mathfrak{s}}

\newcommand\ssm{\smallsetminus}

\newcommand\Psit{\underline{\Psi}}
\newcommand\Z{\mathbb{Z}}

\newcommand\inr{{\rm int}}
\newcommand\RR{\mathcal{R}}

\newcommand\data{\mathscr{D}}
\newcommand\SFH{SFH}
\newcommand\SHM{SHM}
\newcommand\SHMt{\underline{\SHM}}
\newcommand\KHM{KHM}
\newcommand\KHMt{\underline{\KHM}}

\newcommand{\definefunctor}[1]{\textbf{\textup{#1}}}

\newcommand\SHMtfun{\definefunctor{\SHMt}}

\newcommand\Img{{\rm Im}}
\newcommand\invt{\psi} 
\newcommand\HMtoc{\HMto_{\bullet}}

\newcommand\PSys{\textbf{\textup{PSys}}}
\newcommand{\RPSys}[1][\RR]{{#1}\mbox{-}\PSys}
\newcommand\DiffSut{\textup{\textbf{DiffSut}}}
\newcommand\CobSut{\textup{\textbf{CobSut}}}

\newcommand{\longcomment}[2]{#2}

\DeclareFontFamily{U}{mathx}{\hyphenchar\font45}
\DeclareFontShape{U}{mathx}{m}{n}{
      <5> <6> <7> <8> <9> <10>
      <10.95> <12> <14.4> <17.28> <20.74> <24.88>
      mathx10
      }{}
\DeclareSymbolFont{mathx}{U}{mathx}{m}{n}
\DeclareFontSubstitution{U}{mathx}{m}{n}
\DeclareMathAccent{\widecheck}{0}{mathx}{"71}
\newcommand{\HMto}{\widecheck{\mathit{HM}}}

\longcomment{
    \RequirePackage{rotating}                   
    \def\HMto{%
       \setbox0=\hbox{$\widehat{\mathit{HM}}$}
       \setbox1=\hbox{$\mathit{HM}$}
       \dimen0=1.1\ht0
       \advance\dimen0 by 1.17\ht1
       \smash{\mskip2mu\raise\dimen0\rlap{%
          \begin{turn}{180}
              {$\widehat{\phantom{\mathit{HM}}}$}
           \end{turn}} \mskip-2mu    
                \mathit{HM}
    }{\vphantom{\widehat{\mathit{HM}}}}{}}
}
    
    \newcommand*\oline[1]{%
  \vbox{%
    \hrule height 0.35pt
    \kern0.1ex
    \hbox{%
      \kern-0.0em
      \ifmmode#1\else\ensuremath{#1}\fi
      \kern-0.1em
    }
  }
}

\newtheorem{theorem}{Theorem}[section]
\newtheorem{lemma}[theorem]{Lemma}

\newtheorem{conjecture}[theorem]{Conjecture}
\newtheorem{corollary}[theorem]{Corollary}
\newtheorem{proposition}[theorem]{Proposition}

\theoremstyle{definition}
\newtheorem{definition}[theorem]{Definition}
\newtheorem{notation}[theorem]{Notation}

\newtheorem{remark}[theorem]{Remark}

\makeatletter
\newtheorem*{rep@thm}{\rep@title}
\newcommand{\newreptheorem}[2]{%
\newenvironment{rep#1}[1][0,0]{%
\def\rep@title{#2##1}%
\begin{rep@thm}}%
{\end{rep@thm}}}
\makeatother
\newreptheorem{theorem}{}

\begin{document}
\begin{abstract} 
We define an invariant of contact 3-manifolds with convex boundary using Kronheimer and Mrowka's sutured monopole Floer  homology  theory ($\SHM$). Our invariant can be viewed as a generalization of Kronheimer and Mrowka's contact invariant for closed contact 3-manifolds and as the monopole Floer analogue of Honda, Kazez, and Mati{\'c}'s contact invariant in sutured Heegaard Floer homology ($\SFH$). In the process of defining our invariant, we construct  maps on $\SHM$ associated to contact handle attachments, analogous to those defined by Honda, Kazez, and Mati{\'c} in $\SFH$. We  use these maps to establish a bypass exact triangle in   $\SHM$  analogous to Honda's in $\SFH$. This paper also provides the topological basis for the construction of similar gluing maps in sutured instanton Floer homology, which are used in \cite{bsSHI} to define a contact invariant in the instanton Floer setting.
 \end{abstract}

\maketitle


\section{Introduction}
\label{sec:intro}
Floer-theoretic invariants of contact structures---in particular, those defined by Kronheimer and Mrowka in \cite{km} and by  Ozsv{\'a}th and Szab{\'o} in  \cite{osz1}---have led to a number of spectacular results  in low-dimensional    topology over the last decade or so. Kronheimer and Mrowka's invariant, defined using Taubes's work  on the Seiberg-Witten equations for symplectic 4-manifolds,  assigns to a closed contact 3-manifold $(Y,\xi)$ a class
$\psi(Y,\xi)$ in the monopole Floer homology   $\HMtoc(-Y,\spc_{\xi})$, where $\spc_{\xi}$ is the $\Sc$ structure  associated to $\xi$.\footnote{This formulation of Kronheimer and Mrowka's invariant first appears in \cite{kmosz}, actually.}  Ozsv{\'a}th and Szab{\'o}'s invariant similarly  takes the form of a class $c^+(Y,\xi)$ in the Heegaard Floer homology   $\hfp(-Y,\spc_{\xi})$, but is defined using Giroux's correspondence between contact structures and open books.

 Honda, Kazez, and Mati{\'c}  introduced an important generalization of Ozsv{\'a}th and Szab{\'o}'s construction in \cite{hkm4}, using a \emph{relative} version of Giroux's correspondence  to define an invariant of \emph{sutured contact manifolds}, which are triples of the form $(M,\Gamma,\xi)$ where $(M,\xi)$ is a contact 3-manifold with convex boundary and  $\Gamma\subset \partial M$ is a multicurve dividing the characteristic foliation of $\xi$ on $\partial M$.\footnote{Technically, the invariant in \cite{hkm4} generalizes the ``hat" version of Ozsv{\'a}th and Szab{\'o}'s invariant. Also, it is worth mentioning that the term \emph{sutured contact manifold} is used slightly differently in \cite{cghh}.} Their invariant  assigns to $(M,\Gamma,\xi)$ a class $EH(M,\Gamma,\xi)$ in  the sutured Heegaard Floer homology   $\SFH(-M,-\Gamma)$. The work in \cite{hkm4} and its sequel \cite{hkm5} has enhanced our understanding of   Legendrian knots \cite{sv}, knot Floer homology \cite{evz},  functoriality in $SFH$ \cite{juhasz3}, and  bordered Heegaard Floer homology \cite{zarev}, and has   revealed interesting categorical structure in contact topology \cite{honda4}. This categorical structure has, in turn,  had  important applications to the categorification of quantum groups \cite{yintian1,yintian2}.

In this paper, we define an invariant of sutured contact manifolds in Kronheimer and Mrowka's sutured monopole Floer homology   ($SHM$), generalizing their  invariant  for closed contact manifolds in the same way that Honda, Kazez, and Mati{\'c}'s contact invariant generalizes that of Ozsv{\'a}th and Szab{\'o} on the Heegaard Floer side.  In other words, \[\text{our invariant}: \psi \,\,::\,\, EH:c^+.\] Although our contact invariant can be thought of as the monopole Floer analogue of Honda, Kazez, and Mati{\'c}'s $EH$ invariant, our construction is quite different from theirs (not surprising, considering the different constructions of $\psi$ and $c^+$). One advantage of our construction is that it does not rely on the full strength of the relative Giroux correspondence, a complete proof of which is currently lacking. Moreover, we show that our contact invariant is ``natural" in the sense that it is preserved by the canonical isomorphisms relating the different sutured monopole homology groups associated to a given  sutured contact manifold, something which has not been completely established on the Heegaard Floer side.

As a byproduct of our construction, we define  ``gluing" maps in $\SHM$ associated to contact handle attachments, analogous to those  in $\SFH$ defined by Honda, Kazez, and Mati{\'c} in \cite{hkm5}.  As we shall see, Kronheimer and Mrowka's approach to sutured Floer theory via  the \emph{closure} operation    allows for a conceptually simpler construction of  these  maps than  in $\SFH$. We use these gluing maps  to establish  a monopole Floer analogue of Honda's bypass exact triangle in $\SFH$---the  centerpiece   of his \emph{contact category} \cite{honda4}.  Moreover, our approach shows that these bypass  triangles are instances of the usual surgery exact triangle in Floer homology, suggesting that Honda's contact category fits naturally into a larger category of closed manifolds. 

Our work on defining gluing maps in $\SHM$ also provides the topological foundation for  a  similar gluing map construction in Kronheimer and Mrowka's sutured instanton Floer homology. We use this construction  in \cite{bsSHI} to define the first invariant of contact manifolds in the instanton Floer setting.  

Beyond providing new insights into   developments that have sprung from   Honda, Kazez, and Mati{\'c}'s work,  intrinsic advantages of the monopole Floer perspective have enabled us to prove  results with no counterparts on the Heegaard Floer side. This point is illustrated in \cite{bsLEG}, where we use the contact invariant defined in this paper to construct new invariants of Legendrian and transverse knots in monopole knot homology. The functoriality of Kronheimer and Mrowka's invariant $\psi$ under exact symplectic cobordism leads to a proof that our Legendrian invariant is functorial with respect to Lagrangian cobordism (cf. \cite{sivek} for a similar result), something which is not known to be true of the analogous ``LOSS" invariant in knot Floer homology \cite{lossz}.

Below, we outline the constructions of our contact invariant and our gluing maps, elaborating on several points in the discussion above. We discuss  future work at the end.

\subsection{A contact invariant in $\SHM$} 
\label{ssec:introcontact}

Suppose $(M,\Gamma)$ is a balanced sutured manifold.  Consider the manifold obtained by gluing  a  thickened surface $F\times I$ to $M$ by   a map which identifies $\partial F\times I$ with a tubular neighborhood of $\Gamma\subset\partial M$. Under  mild  assumptions, this manifold has two diffeomorphic boundary components. Gluing these  together by some diffeomorphism, one obtains a closed 3-manifold $Y$ with a distinguished surface $R\subset Y$. Kronheimer and Mrowka call a pair $(Y,R)$ obtained in this way  a \emph{closure} of $(M,\Gamma)$. Its \emph{genus}  refers to the genus of $R$.

Suppose now that $(Y,R)$ is a closure of $(M,\Gamma)$ with genus at least two and   $\eta$ is an embedded, nonseparating 1-cycle in $R$. Let $\RR$ be the Novikov ring over $\Z$. As defined in \cite{km4}, the  \emph{sutured monopole homology} of $(M,\Gamma)$ is  the $\RR$-module given by the portion of the  ``twisted" monopole Floer homology of $Y$ in the ``topmost" $\Sc$ structures relative to $R$, \begin{equation}\label{eqn:closuredef}\SHMt(M,\Gamma) := \HMtoc(Y|R;\Gamma_{\eta}):=\bigoplus_{ \langle c_1(\spc), [R]\rangle = 2g(R)-2} \HMtoc(Y,\spc; \Gamma_{\eta}).\footnote{$\Gamma_{\eta}$ refers to a local system on the Seiberg-Witten configuration space $\mathcal{B}(Y,\spc)$ with fiber  $\RR$ specified by  $\eta$. } \end{equation}
 Kronheimer and Mrowka showed that $\SHMt(M,\Gamma)$ is well-defined up to isomorphism. Moreover, the combined results of Kronheimer and Mrowka \cite[Lemma 4.9]{km4}, Taubes \cite{taubes1,taubes2,taubes3,taubes4,taubes5}, Colin, Ghiggini, and Honda \cite{cgh3, cgh4, cgh5}, and Lekili \cite{lekili2} show that \begin{equation}\label{eqn:SFHMiso}\SHMt(M,\Gamma) \cong \SFH(M,\Gamma)\otimes\RR.\footnote{See also Kutluhan, Lee, and Taubes \cite{klt1,klt2,klt3,klt4,klt5}.}\end{equation}

In \cite{bs3}, we introduced a refinement of this construction which assigns to $(M,\Gamma)$ an $\RR$-module that is well-defined up to \emph{canonical} isomorphism, modulo multiplication by units in $\RR$. Our refinement begins with a modification of Kronheimer and Mrowka's notion of closure. For us, a (marked) closure is a tuple $\data$ which records the data $(Y,R,\eta)$, as well as things like  an explicit smooth structure on $Y$ and  smooth embeddings of $M$ and $R$ into $Y$. The sutured monopole homology $\SHMt(\data)$ of a closure $\data$ of $(M,\Gamma)$ is then defined in terms of $(Y,R,\eta)$ as in \eqref{eqn:closuredef}. For any two closures $\data,\data'$ of $(M,\Gamma)$, we constructed an isomorphism \[\Psit_{\data,\data'}:\SHMt(\data)\to\SHMt(\data'),\] which is well-defined up to multiplication by a unit in $\RR$ and satisfies the transitivity  \[\Psit_{\data,\data''}\doteq\Psit_{\data',\data''}\circ\Psit_{\data,\data'}.\footnote{Here, ``$\doteq$'' means ``equal up to multiplication by a unit".}\]  We view these maps as canonical isomorphisms relating the $\RR$-modules assigned to the different closures of $(M,\Gamma)$. These maps and modules are organized into what we call a \emph{projectively transitive system}, denoted by $\SHMtfun(M,\Gamma)$. It is  this system we are referring to in this paper when we talk about the \emph{sutured monopole homology} of $(M,\Gamma)$.

Now, suppose $(M,\Gamma,\xi)$ is a sutured contact manifold. To define the contact invariant, we introduce the notion of    a \emph{contact closure} of $(M,\Gamma,\xi)$, which is a closure $\data$ of $(M,\Gamma)$  together with a contact structure $\bar\xi$ on $Y$ extending $\xi$ and satisfying certain conditions. One of these conditions is that the surface $R$ is convex with negative region an annulus, which guarantees that \[\langle c_1(\spc_{\bar\xi}),[R]\rangle = 2-2g(R)\] by basic convex surface theory. But this implies that $\HMtoc(-Y,\spc_{\bar\xi};\Gamma_{-\eta})$ is a direct summand of $\SHMt(-\data)$, where $-\data$ is the closure of $(-M,-\Gamma)$ induced by reversing the orientations on $Y,$ $R,$ and $\eta$. It therefore makes sense to define \[\psi(\data,\bar\xi):=\psi(Y,\bar\xi)\in \HMtoc(-Y,\spc_{\bar\xi};\Gamma_{-\eta})\subset\SHMt(-\data),\] where, here,  $\psi(Y,\bar\xi)$ is the ``twisted" version of Kronheimer and Mrowka's contact invariant. 

Our main result is, roughly speaking, that the classes $\invt(\data,\bar\xi)$ define an invariant of $(M,\Gamma,\xi)$ up to canonical isomorphism. For the sake of exposition, we have broken  this result into the following two theorems (see Theorems \ref{thm:well-defined} and \ref{thm:well-defined2} for more precise statements).

\begin{theorem}
\label{thm:mainintrosame}
If $(\data,\bar\xi)$ and $(\data',\bar\xi')$ are  contact closures of $(M,\Gamma,\xi)$, then \[\Psit_{-\data,-\data'}(\invt(\data,\bar\xi))\doteq \invt(\data',\bar\xi')\] for $g(\data)=g(\data')$.
\end{theorem}

\begin{theorem}
\label{thm:mainintrodiff}
If $(\data,\bar\xi)$ and $(\data',\bar\xi')$ are  contact closures of $(M,\Gamma,\xi)$, then \[\Psit_{-\data,-\data'}(\invt(\data,\bar\xi))\doteq \invt(\data',\bar\xi')\] whenever $g(\data)$ and $g(\data')$ are sufficiently large.
\end{theorem}

In the lexicon of projectively transitive systems, Theorem \ref{thm:mainintrosame} is equivalent to the statement that, for each $g$, the collection of classes $\{\invt(\data,\bar\xi)\mid g(\data)=g\}$  defines an \emph{element} \[\psi^g(M,\Gamma,\xi)\in\SHMtfun(-M,-\Gamma).\] Meanwhile, Theorem \ref{thm:mainintrodiff} is equivalent to the statement that these elements stabilize in the sense that $\psi^g(M,\Gamma,\xi) = \psi^h(M,\Gamma,\xi)$ for $g,h$  sufficiently large. Our contact invariant  is defined to be this stable element, which we denote  by \[\psi(M,\Gamma,\xi)\in \SHMtfun(-M,-\Gamma).\]





 The key facts in   the proof of Theorem \ref{thm:mainintrosame} are   that (1) contact closures of the same genus are related by Legendrian surgery and (2)  $\Psit_{-\data,-\data'}$ is the map induced by the associated Stein cobordism.\footnote{This is a considerable simplification of the real story.} 
 Theorem \ref{thm:mainintrosame}  therefore follows from the functoriality of Kronheimer and Mrowka's contact invariant with respect to exact symplectic cobordism (see Theorem \ref{thm:psi-weakly-fillable}). For  closures of different genera, $\Psit_{-\data,-\data'}$ is defined in terms of a  \emph{splicing} cobordism which does not admit, in any obvious way, the structure of an exact symplectic cobordism. So, the previous argument cannot be used to  prove Theorem \ref{thm:mainintrodiff}. Our proof  relies instead on our gluing map construction together with what we call the ``existence" part of the relative Giroux correspondence. We     outline this proof in detail in the next subsection after describing these gluing maps.
 
 Our contact invariant shares several features with Honda, Kazez, and Mati{\'c}'s  invariant. For example, it is preserved by contact isotopy and \emph{flexibility}, and vanishes for overtwisted contact structures (see Corollaries \ref{cor:isotopyindependence} and \ref{cor:independenceofdividingset} and Theorem \ref{thm:ot} for more precise statements). 
 
 We also prove  the following theorem relating  our invariant to Kronheimer and Mrowka's contact invariant for closed manifolds (stated more precisely in Proposition \ref{prop:darboux-complement}). Below,  $(Y,\xi)$ is a closed contact 3-manifold and   $(Y(1), \xi|_{Y(1)})$ is the sutured contact manifold obtained from it by removing a Darboux ball.

 \begin{theorem}
 \label{thm:introhatplus}
There exists a  map 
\[\SHMtfun(-Y(1)) \to \HMtoc(-Y) \otimes\RR\]
which sends $\invt(Y(1), \xi|_{Y(1)})$ to $\psi(Y,\xi) \otimes 1$.  \end{theorem}

As explained in Remark \ref{rmk:analoguehatplus}, this map can be thought of as a monopole Floer analogue of the natural map in Heegaard Floer homology relating the ``hat" and ``plus" versions of Ozsv{\'a}th and Szab{\'o}'s contact invariant. 

The following is an immediate corollary (see Corollary \ref{cor:nonzerostronglyfillable}).

\begin{corollary}
If $(Y,\xi)$ is strongly symplectically fillable, then $\invt(Y(1), \xi|_{Y(1)})\neq 0$. \qed
\end{corollary}

Before moving on, it is worth mentioning that Kronheimer and Mrowka also define   a version of $\SHM$  without local coefficients. However, local coefficients are necessary in this paper, both for naturality purposes (in defining the canonical isomorphisms for closures of different genera) and because the contact class $\psi(\data,\bar\xi)$ always vanishes without them (see Remark \ref{rem:why-twisted}).

\subsection{A gluing map in $\SHM$} Below, we describe a  gluing map on $\SHM$ for contact handles. Our main results in this direction are the following (combining  Propositions \ref{prop:H0}, \ref{prop:H1}, \ref{prop:H2}, \ref{prop:H3}, and Corollary \ref{cor:handles}).

\begin{theorem}
\label{thm:contacthandle} Suppose  $(M',\Gamma',\xi')$ is obtained from $(M,\Gamma,\xi)$ by attaching a contact $i$-handle, for some $i\in\{0,1,2,3\}$. Then there exists a map
\[\mathscr{H}_i:\SHMtfun(-M,-\Gamma)\to\SHMtfun(-M',-\Gamma')\] which sends $\invt^g(M,\Gamma,\xi)$ to $\invt^g(M',\Gamma',\xi')$ for $g$ sufficiently large. 
\end{theorem}

\begin{corollary}
\label{cor:introhandles}
The map $\mathscr{H}_i$ sends $\invt(M,\Gamma,\xi)$ to $\invt(M',\Gamma',\xi')$.
\end{corollary}

It is worth pointing out that these maps depend \emph{only} on the smooth data involved in the  handle attachments; in particular, they do not depend on the contact structures $\xi$ or $\xi'$. 

These maps are defined in terms of natural  cobordisms between   closures: for $i=0,1$, we show that a contact closure of $(M',\Gamma',\xi')$ can also  be viewed naturally as a contact closure of $(M,\Gamma,\xi)$, and  we define $\mathscr{H}_i$ to be the isomorphism induced by the  identity map on monopole Floer homology. For $i=2$, the curve of attachment in $\partial M$   gives rise to  a Legendrian knot in any contact closure of $(M,\Gamma,\xi)$. We  prove that  the result of contact $(+1)$-surgery  along  such a  knot can be viewed naturally as a contact closure of $(M',\Gamma',\xi')$, and  we define  $\mathscr{H}_2$ in terms of the map on  Floer homology induced by the corresponding 2-handle cobordism. Finally, for $i=3$, we prove that one obtains a contact closure of $(M,\Gamma,\xi)$ by taking a connected sum of a contact closure of $(M',\Gamma',\xi')$ with the tight $S^1\times S^2$, and we define $\mathscr{H}_3$ in terms of the map on Floer homology induced by the corresponding 1-handle cobordism.

Theorem \ref{thm:contacthandle} is reminiscent of the main result of \cite{hkm5}. Suppose  $(M,\Gamma)$ is a sutured submanifold of $(M',\Gamma')$ and  $\xi$ is a contact structure on $M'\ssm \inr(M)$ with  dividing set $\Gamma\cup \Gamma'$. Suppose further that $\xi_M$ is a contact structure on $M$ which agrees with $\xi$ near $\partial M$. In \cite{hkm5}, Honda, Kazez, and Mati{\'c}  construct  a   map\footnote{Modulo incorporating the naturality results of \cite{juhaszthurston}.} \[\Phi_{\xi}:\SFH(-M,-\Gamma)\to\SFH(-M',-\Gamma'),\] depending only on $\xi$,  which sends $EH(M,\Gamma,\xi_M)$ to  $EH(M',\Gamma',\xi_M\cup \xi).$
We can use Theorem \ref{thm:contacthandle} to define an analogous map in $\SHM$, starting from the observation that  $(M'\ssm \inr(M),\Gamma \cup \Gamma',\xi)$ can be obtained from a vertically invariant contact structure on $\partial M\times I$ by attaching contact handles. Given such a contact handle decomposition $H$,   we define  \begin{equation}\label{eqn:PhiH}\Phi_{\xi,H}:\SHMtfun(-M,-\Gamma)\to\SHMtfun(-M',-\Gamma')\end{equation}  to be the  corresponding composition of  contact  handle attachment maps. Corollary \ref{cor:introhandles} implies that this map sends contact invariant to contact invariant. We conjecture the following.

\begin{conjecture}
\label{conj:PhiHwd}
The map $\Phi_{\xi,H}$ is independent of the handle decomposition $H$.
\end{conjecture} 

We next outline the proof of Theorem \ref{thm:mainintrodiff}, as promised. As mentioned above, our proof relies upon the  ``existence" part of the relative Giroux correspondence, which states that every sutured contact manifold admits a compatible partial open book decomposition.\footnote{The full statement of the relative Giroux correspondence comprises  this ``existence" part, whose proof is well-established, together with the ``uniqueness" part, which states that any two such partial open book decompositions are related by positive stabilization. A complete proof of the ``uniqueness" part is lacking.} This implies, in particular, that for every $(M,\Gamma,\xi)$, there is a compact surface $S$ with nonempty boundary such that $(M,\Gamma,\xi)$ can be  obtained from the tight sutured contact manifold \[H(S)=(S\times [-1,1], \partial S\times\{0\},\xi_S)\] by attaching contact 2-handles. The corresponding composition of contact 2-handle  maps, \[\mathscr{H}:\SHMtfun(-H(S))\to\SHMtfun(-M,-\Gamma),\] therefore  sends $\invt^g(H(S))$ to $\invt^g(M,\Gamma,\xi)$ for $g$ sufficiently large. Now, suppose that \[(\data_S,\bar\xi_S),\,(\data'_S,\bar\xi'_S) \,\, \text{ and }\,\,  (\data,\bar\xi),\,(\data',\bar\xi')\] are contact closures of  $H(S)$ and $(M,\Gamma,\xi)$, respectively, with \[
g=g(\data_S)=g(\data) \,\, \text{ and } \,\,
h=g(\data'_S) = g(\data').\]
The morphism $\mathscr{H}$ encodes maps $\mathscr{H}_{-\data_S,-\data}$ and $\mathscr{H}_{-\data'_S,-\data'}$ which make the  diagram 
\[ \xymatrix@C=50pt@R=30pt{
\SHMt(-\data_S) \ar[r]^-{\mathscr{H}_{-\data_S,-\data}} \ar[d]_{\Psit_{-\data_S,-\data'_S}} & \SHMt(-\data) \ar@<-1.5ex>[d]^{\Psit_{-\data,-\data'}} \\
\SHMt(-\data'_S) \ar[r]_-{\mathscr{H}_{-\data'_S,-\data'}} & \SHMt(-\data')
} \] 
commute, up to multiplication by a unit. Theorem \ref{thm:contacthandle} can be translated as saying that 
\[
\mathscr{H}_{-\data_S,-\data}(\invt(\data_S,\bar\xi_S))\doteq\invt(\data,\bar\xi)\,\, \text{ and } \,\,
\mathscr{H}_{-\data_S',-\data'}(\invt(\data_S',\bar\xi_S'))\doteq\invt(\data',\bar\xi')\]
for sufficiently large $g$ and $h$. To prove Theorem \ref{thm:mainintrodiff}, it therefore suffices to show that \[\Psit_{-\data_S,-\data'_S}(\invt(\data_S,\bar\xi_S))\doteq\invt(\data_S',\bar\xi_S').\] But this is true since $\Psit_{-\data_S,-\data'_S}$  is an isomorphism and $\invt(\data_S,\bar\xi_S)$ and $\invt(\data_S',\bar\xi_S')$ generate the modules $\SHMt(-\data_S)$ and $\SHMt(-\data'_S)$, which are both isomorphic to $\RR$ (see Subsection \ref{ssec:closureexamples}).

From this proof sketch, one also finds that the contact invariant $\invt(M,\Gamma,\xi)\in \SHMtfun(-M,-\Gamma)$ is characterized by \begin{equation}\label{eqn:characterization}\invt(M,\Gamma,\xi) = \mathscr{H}(\mathbf{1}),\end{equation} where $\mathbf{1}=\invt(H(S))$ is the \emph{generator} of $\SHMtfun(-H(S))\cong \RR$. It is  therefore natural to ask whether one can \emph{define} a contact invariant in this way, forgetting about contact closures and Kronheimer and Mrowka's contact invariant entirely. In other words, a partial open book decomposition $\mathfrak{ob}$ compatible with $(M,\Gamma,\xi)$ determines a surface $S$, a map $\mathscr{H}$, and a class \[\invt(\mathfrak{ob}):=\mathscr{H}(\mathbf{1})\in\SHMtfun(-M,-\Gamma),\] and the question is whether one can show, without appealing to the existing monopole Floer apparatus for closed contact manifolds, that $\invt(\mathfrak{ob}) = \invt(\mathfrak{ob'})$ for any two such partial open book decompositions. Although we do not show it in this paper, it turns out this can be done using the full relative Giroux correspondence (both the ``existence" and the ``uniqueness" parts). In fact, this idea is the basis for our construction in \cite{bsSHI} of a contact invariant in sutured instanton Floer homology. 

We end with a   synopsis of our bypass exact triangle in $\SHM$. A \emph{bypass move} is a certain local modification of the dividing set of  a sutured  manifold (see Figure \ref{fig:bypass-move}). Every such move can be achieved by attaching a \emph{bypass} (roughly, half  of a thickened overtwisted disk) to the manifold along an  arc in its boundary. In turn, this bypass attachment can  be achieved by  attaching a contact 1-handle followed by a contact 2-handle in a manner determined by the arc (see Figure \ref{fig:bypass-handles}). So, the contact handle attachment maps of Theorem \ref{thm:contacthandle} allow us to define similar maps for bypass attachments.

In \cite{honda4}, Honda studies a certain 3-periodic sequence of bypass moves  which he calls a \emph{bypass triangle} (see Figure \ref{fig:bypass-triangle}), and he shows that the $\SFH$ groups  of sutured manifolds related by a bypass triangle fit into an exact triangle. This \emph{bypass exact triangle} is the main feature of his \emph{contact category,} which is envisioned as an algebraic approach to  contact geometry. This contact category, though still a work-in-progress, has been studied for a variety of purposes, including as an approach to categorifying certain quantum groups \cite{yintian1,yintian2,yintian3}.

In this paper, we establish a monopole Floer analogue of Honda's bypass exact triangle (see Theorem \ref{thm:bypass} for a more precise statement):

\begin{theorem}
\label{thm:introbypass}
Suppose $\Gamma_1,\Gamma_2,\Gamma_3\subset \partial M$ is a 3-periodic sequence of sutures related by the moves in a bypass triangle. Then there is an exact triangle
\[ \xymatrix@C=-37pt@R=28pt{
\SHMtfun(-M,-\Gamma_1) \ar[rr] & & \SHMtfun(-M,-\Gamma_2) \ar[dl] \\
& \SHMtfun(-M,-\Gamma_3) \ar[ul] & \\
}, \]
in which each arrow is the corresponding bypass attachment map.\footnote{We only prove this over $\RR\otimes\Z/2\Z$.}

\end{theorem}

Recall that each bypass attachment map is the composition of a 1-handle map with a 2-handle map. But, on the level of closures, a  1-handle map is essentially the identity and a 2-handle map is the cobordism map associated to integral surgery on a knot. It is perhaps not surprising then  that our bypass exact triangle is  really just the usual surgery exact triangle in monopole Floer homology in disguise. This suggests that  Honda's contact category may  be   a natural subcategory of  some   larger category of closed 3-manifolds defined without reference to contact structures (perhaps one with an interesting $A_\infty$ structure).

\subsection{Future directions}
\label{ssec:introfuture} Below, we describe several goals for future work.
One of our  immediate  goals is to prove that our contact invariant ``agrees" with Honda, Kazez, and Mati{\'c}'s  $EH$ invariant.  Forgetting about naturality, we can view  $\invt(M,\Gamma,\xi)$  as given by \[\psi(M,\Gamma,\xi):=\invt(Y,\bar\xi)\in \HMtoc({-}Y|{-}R;\Gamma_{-\eta}) =: \SHMt(-M,-\Gamma)\] for any contact closure $(Y,R,\eta,\bar\xi)$ of $(M,\Gamma,\xi)$. We aim   to prove the following

\begin{conjecture} 
\label{conj:isosfhshm}There exists an isomorphism  \[\SHMt(-M,-\Gamma)\to\SFH(-M,-\Gamma)\otimes\RR\] which sends $\psi(M,\Gamma,\xi)$ to $EH(M,\Gamma,\xi)\otimes 1$.
\end{conjecture}

If true, this conjecture would lead to  a new proof of the invariance of $EH$ up to isomorphism, independent of the more controversial ``uniqueness" part of the relative Giroux correspondence. Combined with our work in \cite{bsLEG}, it would also show that the ``LOSS" invariant in knot Floer homology satisfies a certain functoriality with respect to Lagrangian concordance.

Unsurprisingly, our strategy for proving Conjecture \ref{conj:isosfhshm} relies on the combined work of Taubes \cite{taubes1,taubes2,taubes3,taubes4,taubes5} and Colin, Ghiggini, and Honda \cite{cgh3, cgh4, cgh5}, which  shows that there is an isomorphism \[\HMtoc({-}Y|{-}R;\Gamma_{-\eta}) \to \hfp({-}Y|{-}R;\Gamma_{-\eta})\] sending $\psi(Y,\bar\xi)$ to $c^+(Y,\bar\xi)$, together with   the work of Lekili \cite{lekili2}, which establishes an isomorphism \[\hfp({-}Y|{-}R;\Gamma_{-\eta})\to \SFH(-M,-\Gamma)\otimes\RR.\] To prove Conjecture \ref{conj:isosfhshm}, it thus suffices to prove that Lekili's isomorphism sends $c^+(Y,\bar\xi)$ to $EH(M,\Gamma,\xi)\otimes 1$. We plan to prove this for a slight modification of Lekili's isomorphism, using a characterization of $c^+(Y,\bar\xi)$ similar to that of $\invt(M,\Gamma,\xi) = \psi(Y,\bar\xi)$ in \eqref{eqn:characterization}.

Another immediate goal involves Kronheimer and Mrowka's \emph{monopole knot homology}, defined in \cite{km4}. Their theory  assigns to a knot $K$ in a closed 3-manifold $Y$ the isomorphism class of $\RR$-modules \[\KHMt(Y,K):=\SHMt(Y\ssm \nu(K),m \cup -m),\] where $\nu(K)$ is a tubular neighborhood of the knot and $m$ is an oriented meridian on the boundary of this knot complement (see \cite{bs3} for our ``natural" refinement of this construction.) It follows from the work of Taubes et al. that monopole knot homology is isomorphic to the ``hat" version of knot Floer homology, \[\KHMt(Y,K)\cong \widehat{HFK}(Y,K)\otimes\RR.\] Our aim is to use the bypass attachment maps defined here to construct a version of monopole knot homology analogous to the more powerful ``minus" version of knot Floer homology. Our approach is based on the  work of Etnyre, Vela-Vick, and Zarev \cite{evz} described below. 

Suppose $K$ is an oriented Legendrian knot in $(Y,\xi)$ and let $M_n = (Y \smallsetminus \nu(K),\Gamma_n)$ be the complement of the $n$-fold negative stabilization of $K$. Then $M_n$ can be thought of as obtained from $M_{n-1}$ by gluing on a layer $(T^2\times I,\xi_-)$ called a \emph{negative basic slice}. Etnyre, Vela-Vick, and Zarev define $\underrightarrow{SFH}(-Y,K)$ to be the direct limit of the sequence
\[ \SFH(-M_0) \xrightarrow{\Phi_{\xi_-}} \SFH(-M_1) \xrightarrow{\Phi_{\xi_-}} \SFH(-M_2) \to \dots, \]
where the $\Phi_{\xi_-}$ are the  gluing maps associated to these basic slice attachments. This limit $\underrightarrow{SFH}(-Y,K)$ is a module over $(\Z/2\Z)[U]$, where the $U$-action is induced  by  maps \[\Phi_{\xi_+}:\SFH(-M_n)\to\SFH(-M_{n-1})\] associated to gluing on layers  $(T^2\times I, \xi_+)$ called \emph{positive} basic slices. The authors of \cite{evz} show that $\underrightarrow{SFH}(K)$ is isomorphic to $HFK^-(-Y,K)$ as a $(\Z/2\Z)[U]$-module.

Since these basic slice attachments are equivalent to bypass attachments, we can use our bypass attachment maps to define a similar limit module in $\SHM$. To define an $\RR[U]$-module structure on this limit, one must show that the bypass attachment maps corresponding to the positive and negative basic slices  commute. This is work in progress. Once complete, it will be interesting to compare this limit theory with $HFK^-(-Y,K)$ and with a similar invariant defined by Kutluhan \cite{kutluhan}  in terms of filtrations on the monopole Floer complex of $-Y$. 

Another major goal is to prove Conjecture \ref{conj:PhiHwd}, which posits that the  map $\Phi_{\xi,H}$ in \eqref{eqn:PhiH} is independent of  $H$. This would imply, in particular, that the maps associated to the  positive and negative basic slice attachments  above commute as desired. More importantly, it would  allow us to  to assign well-defined maps  to cobordisms between sutured manifolds in the monopole Floer setting---in the language of \cite{bs3}, to extend $\SHMtfun$ to a functor from $\CobSut$  to $\RPSys$. Our approach is based on the work of Juh{\'a}sz  \cite{juhasz3} outlined below.

As defined in \cite{juhasz3},  a cobordism  from $(M_1,\Gamma_1)$ to $(M_2,\Gamma_2)$ consists of a smooth 4-manifold $W$ with boundary $\partial W=-M_1\cup Z \cup M_2,$ together with a contact structure $\xi$ on $Z$ with dividing set $\Gamma_1\cup \Gamma_2$ on $\partial Z = -\partial M_1\cup \partial M_2$. Juh{\'a}sz assigns to such a cobordism a map \[F_W:\SFH(-M_1,-\Gamma_1)\to\SFH(-M_2,-\Gamma_2),\] defined as the composition  $F_W = F_{W}'\circ\Phi_{\xi},$ where \begin{equation}\label{eqn:jhkm}\Phi_{\xi}:\SFH(-M_1,-\Gamma_1)\to \SFH(-M_1\cup H, -\Gamma_2)\end{equation} is the contact gluing map defined by Honda, Kazez, and Mati{\'c}, and \begin{equation}\label{eqn:j2}F_W':\SFH(-M_1\cup H, -\Gamma_2)\to\SFH(-M_2,-\Gamma_2)\end{equation} is a map defined via more standard Heegaard Floer techniques.

Once we prove Conjecture \ref{conj:PhiHwd}, we will have an $SHM$ analogue  of the map in \eqref{eqn:jhkm}. Moreover, we  can already define a monopole Floer analogue of the map in \eqref{eqn:j2}. Indeed, since the sutured manifolds $(-M_1\cup H, -\Gamma_2)$ and $(-M_2,-\Gamma_2)$ have the same (sutured) boundaries, there is a natural way of turning $W$ into a cobordism $\overline W$ between their closures. We define the analogue of $F'_W$ to be the map in monopole Floer homology induced by $\overline W$. With analogues of the maps in \eqref{eqn:jhkm} and  \eqref{eqn:j2}, we can then define an analogue of Juh{\'a}sz's map $F_W$ via composition as above.


The last project we  mention here concerns defining a monopole Floer version of bordered Heegaard Floer homology.
In \cite{zarev}, Zarev showed that the homologies of Lipshitz, Ozsv{\'a}th, and Thurston's bordered Heegaard Floer invariants \cite{lot3} can be expressed as direct sums of certain sutured Floer homology groups. Furthermore, he showed that the multiplication in the homology of the DGA associated to a parametrized surface in bordered Floer homology can be expressed in terms of the sutured cobordism maps defined by Juh{\'a}sz in \cite{juhasz3}. Once we define analogous sutured cobordism maps in $\SHM$, as described above, we will be able to define corresponding homology-level bordered invariants in the monopole Floer setting. 
Of course, simply knowing the homology-level multiplications for the surface DGA and the bordered modules  is not enough for a pairing theorem (a central feature of any good bordered theory), but it would be an important  start.

As mentioned previously, the ideas  in this paper are used to define similar contact handle attachment maps in the instanton Floer setting in \cite{bsSHI}. These maps give rise to an analogous bypass exact triangle in that setting. We plan to use them in future work to construct  sutured cobordism maps and bordered invariants in instanton Floer homology as well, following the strategy outlined above.

\subsection{Organization} In Section \ref{sec:prelims}, we  review  projectively transitive systems, the construction of sutured monopole homology, Kronheimer and Mrowka's  invariant of  closed contact 3-manifolds, and some convex surface theory. In Section \ref{sec:monopoleinvt}, we define the  classes $\invt^g(M,\Gamma,\xi)$ and $\invt(M,\Gamma,\xi)$  and establish some of their basic properties. Much of Section \ref{sec:monopoleinvt} is devoted to preliminary work on \emph{contact preclosures} that is used in Section \ref{sec:well-definedness} to prove Theorem \ref{thm:mainintrosame}---that $\invt^g(M,\Gamma,\xi)$ is well-defined for each $g$. In Section \ref{sec:well-definedness}, we prove Theorem \ref{thm:mainintrosame} and define the contact handle attachment maps in $SHM$. We then use these maps to prove Theorem \ref{thm:mainintrodiff}---that $\invt(M,\Gamma,\xi)$ is well-defined. In Section \ref{sec:bypasstri}, we prove Theorem \ref{thm:introbypass}---that $\SHM$ satisfies a bypass exact triangle.
\subsection{Acknowledgements} We thank John Etnyre, Ko Honda, Peter Kronheimer, {\c C}a{\u g}atay Kutluhan, Tom Mrowka, Jeremy Van Horn-Morris, Shea Vela-Vick, and Vera V{\'e}rtesi for helpful conversations.

\section{Preliminaries}
\label{sec:prelims}
In this section, we review the notion of a projectively transitive system, the construction of sutured monopole homology, and basic properties of Kronheimer and Mrowka's contact invariant, and we collect some facts from convex surface theory.

\subsection{Projectively transitive systems}
\label{ssec:transitive-systems}
In \cite{bs3} we introduced \emph{projectively transitive systems} to make precise the idea of a collection of modules being canonically isomorphic up to multiplication by a unit. We recount their definition and related notions below.

\begin{definition}
Suppose $M_\alpha$ and $M_\beta$ are modules over a unital commutative ring $\RR$. We say that  elements $x,y\in M_\alpha$ are \emph{equivalent} if $x=u\cdot y$ for some  $u\in \RR^\times$. Likewise,  homomorphisms \[f,g: M_\alpha \to M_\beta\]  are  \emph{equivalent} if $f=u\cdot g$ for some  $u\in \RR^\times$.  
\end{definition}

\begin{remark}
We will write $x\doteq y$ or $f\doteq g$ to indicate that two  elements or homomorphisms are equivalent, and  will denote their equivalence classes by $[x]$ or $[f]$.
\end{remark}

Note that \emph{composition} of  equivalence classes of homomorphisms is well-defined, as is the \emph{image} of the equivalence class of an element under an equivalence class of homomorphisms.

\begin{definition}
Let $\RR$ be a unital commutative ring.  A \emph{projectively transitive system of $\RR$-modules} consists of a set $A$ together with:
\begin{enumerate}
\item a collection of $\RR$-modules $\{M_\alpha\}_{\alpha \in A}$ and
\item a collection of equivalence classes of homomorphisms $\{g^\alpha_\beta\}_{\alpha,\beta \in A}$ such that:
\begin{enumerate}
\item elements of the equivalence class $g^\alpha_\beta$ are isomorphisms from $M_\alpha$ to $M_\beta$, 
\item  $g^\alpha_\alpha=[id_{M_\alpha}]$,
\item $g^\alpha_\gamma = g^\beta_\gamma \circ g^\alpha_\beta$.
\end{enumerate}
\end{enumerate}
\end{definition}


\begin{remark}The equivalence classes  of homomorphisms in a projectively transitive system of $\RR$-modules can  be thought of as specifying canonical isomorphisms between the modules in the system that are well-defined up to multiplication by units in $\RR$. 
\end{remark}


The class of  projectively transitive systems of $\RR$-modules forms a category $\RPSys$ with the following notion of morphism.

\begin{definition}
\label{def:projtransysmor} A \emph{morphism} of  projectively transitive systems of $\RR$-modules \[F:(A,\{M_{\alpha}\},\{g^{\alpha}_{\beta}\})\to(B,\{N_{\gamma}\},\{h^{\gamma}_{\delta}\})\]  is  a collection of equivalence classes  of homomorphisms $F=\{F^\alpha_\gamma\}_{\alpha\in A,\,\gamma\in B}$  such that:
\begin{enumerate}
\item elements of the equivalence class $F^\alpha_\gamma$ are homomorphisms from $M_\alpha$ to $N_\gamma$,
\item \label{eqn:projtransysmor} $F^\beta_\delta\circ g^\alpha_\beta = h^\gamma_\delta\circ F^\alpha_\gamma$.
\end{enumerate}  
Note that $F$ is an \emph{isomorphism} iff the elements in each equivalence class $F^\alpha_\gamma$ are isomorphisms.
\end{definition}

\begin{remark}\label{rmk:completesubset} A collection of equivalence classes of homomorphisms $\{F^\alpha_\gamma\}$  with indices  ranging over any nonempty subset of $A\times B$ can be uniquely completed to a morphism as long as this collection satisfies the compatibility in \eqref{eqn:projtransysmor} where it makes sense.
\end{remark}

\begin{definition}
\label{def:element}
 An \emph{element} of a projectively transitive system of $\RR$-modules \[x\in\mathcal{M}=(A,\{M_{\alpha}\},\{g^{\alpha}_{\beta}\})\]
 is a collection of equivalence classes of elements $x = \{x_{\alpha}\}_{\alpha\in A}$ such that:
\begin{enumerate}
\item elements of the equivalence class $x_{\alpha}$ are elements of $M_\alpha$,
\item \label{eqn:projtransysmorelt} $x_\beta = g^\alpha_\beta(x_\alpha)$.
\end{enumerate}  
\end{definition}

\begin{remark}
\label{rmk:completesubset2}
 As in Remark \ref{rmk:completesubset}, a collection of equivalence classes of elements $\{x_{\alpha}\}$ with indices ranging over any nonempty subset of $A$ can be uniquely completed to an element of $\mathcal{M}$ as long as this collection satisfies the compatibility in \eqref{eqn:projtransysmorelt} where it makes sense.
 \end{remark}
 
We  say that $x$ is a \emph{unit} in $\mathcal{M}$  if each $M_\alpha$ is isomorphic to $\RR$ and each $x_{\alpha}$ is the equivalence class of a generator. As there is a well-defined notion of scalar multiplication for projectively transitive systems,  we may also talk about \emph{primitive} elements of $\mathcal{M}$. The \emph{zero} element $0\in \mathcal{M}$ is the collection of equivalence classes of the elements $0\in M_\alpha$. Finally, it is clear how to define the image $F(x)$ of an element $x\in \mathcal{M}$ under a morphism $F:\mathcal{M}\to\mathcal{N}$ of projectively transitive systems of $\RR$-modules.

\begin{remark} In an abuse of notation, we will also use $\RR$ to denote the distinguished system in $\RPSys$ given by  \[ \RR=(\{\star\},\{\RR\}, \{[id_\RR]\})\] consisting of the single $\RR$-module $\RR$ together with the equivalence class of the identity map. There is an obvious  correspondence between elements of a projectively transitive system of $\RR$-modules $\mathcal{M}$ and morphisms $\RR\to\mathcal{M}$. 
\end{remark}

As the category $\RPSys$ contains kernels and images,  there is a straightforward notion of an exact sequence of projectively transitive systems of $\RR$-modules. Concretely, suppose 
\[
\mathcal{M} = (A, \{M_\alpha\}, \{g^\alpha_\beta\}), \,\,\,\,\,\,
\mathcal{N} = (B, \{N_\gamma\}, \{h^\gamma_\delta\}), \,\,\,\,\,\,
\mathcal{P} = (C, \{P_\epsilon\}, \{i^\epsilon_\zeta\})\] are projectively transitive systems of $\RR$-modules. It is easy to see that a sequence \[\mathcal{M} \xrightarrow{F} \mathcal{N} \xrightarrow{G} \mathcal{P}\] is exact at $\mathcal{N}$ iff there exist $\alpha\in A$, $\gamma\in B$, $\epsilon\in C$ and representatives $\hat F^\alpha_\gamma$, $\hat G^\gamma_\epsilon$ of the equivalence classes $F^\alpha_\gamma$, $G^\gamma_\epsilon$ such that the sequence of $\RR$-modules \[M_\alpha \xrightarrow{\hat F^\alpha_\gamma}N_\gamma \xrightarrow{\hat G^\gamma_\epsilon} P_\epsilon\] is exact at $N_\gamma$. 

\subsection{Sutured monopole homology}
\label{ssec:shm}

In this subsection, we describe  our refinement of Kronheimer and Mrowka's sutured monopole homology, as defined  in \cite{bs3}.

\begin{definition} 
\label{def:sutured}A \emph{balanced sutured manifold}  $(M,\Gamma)$ is a compact, oriented, smooth 3-manifold $M$ with a collection $\Gamma$ of disjoint, oriented,  smooth curves in $\partial M$ called \emph{sutures}. Let $R(\Gamma) = \partial M\smallsetminus\Gamma$, oriented as a subsurface of $\partial M$. We require that:
\begin{enumerate}
\item neither $M$ nor $R(\Gamma)$ has  closed components,
\item $R(\Gamma) = R_+(\Gamma)\sqcup R_-(\Gamma)$ with $\partial R_+(\Gamma) = -\partial R_-(\Gamma) = \Gamma$,
\item $\chi(R_+(\Gamma)) = \chi(R_-(\Gamma))$.
\end{enumerate}
 \end{definition}

An \emph{auxiliary surface} for $(M,\Gamma)$ is a compact, connected, oriented surface $F$ with $g(F)>0$ and $\pi_0(\partial F)\cong \pi_0(\Gamma)$.  Suppose $F$ is an auxiliary surface for $(M,\Gamma)$, $A(\Gamma)$ is a closed tubular neighborhood of $\Gamma$ in $\partial M$, and    \[h:\partial F\times[-1,1]\rightarrow A(\Gamma)\] is an orientation-reversing diffeomorphism which sends $\partial F\times \{\pm 1\}$ to $\partial (R_{\pm}(\Gamma)\smallsetminus A(\Gamma)).$ The  \emph{preclosure} of $M$ associated to $F$, $A(\Gamma)$, and $h$ is the smooth 3-manifold \begin{equation*}\label{eqn:bF}M'=M\cup_h F\times [-1,1]\end{equation*} formed by gluing $F\times[-1,1]$ to $M$ according to $h$ and rounding corners.  Condition (3) in Definition \ref{def:sutured} ensures that $M'$ has two diffeomorphic boundary components, $\partial_+ M'$ and $\partial_- M'$. In particular, an easy calculation shows that \begin{equation}\label{eqn:genusM'}g(\partial_\pm M') = \frac{|\Gamma|-\chi(R_+(\Gamma))+2g(F)}{2}.\end{equation}
We may glue $\partial_+M'$ to $\partial_-M'$ by some diffeomorphism   to form a closed manifold $Y$ containing a distinguished surface \[R:=\partial_+M' = -\partial_- M'\subset Y.\] In \cite{km4}, Kronheimer and Mrowka define a \emph{closure} of $(M,\Gamma)$ to be any pair $(Y,R)$ obtained in this way.  Our definition of closure, as needed for naturality, is  slightly more involved.

\begin{definition}[\cite{bs3}]
\label{def:smoothclosure} A \emph{marked closure} of $(M,\Gamma)$ is a tuple $\data = (Y,R,r,m,\eta)$ consisting of:
\begin{enumerate}
\item a closed, oriented,  3-manifold $Y$,
\item  a closed, oriented,  surface $R$ with $g(R)\geq 2$,
\item an oriented, nonseparating, embedded curve $\eta\subset R$,
\item a smooth, orientation-preserving embedding $r:R\times[-1,1]\hookrightarrow Y$,
\item a smooth, orientation-preserving embedding $m:M\hookrightarrow Y\smallsetminus\inr(\Img(r))$ such that: 
\begin{enumerate}
\item $m$ extends  to a diffeomorphism \[M\cup_h F\times [-1,1]\rightarrow Y\smallsetminus{\rm int}(\Img(r))\] for some $A(\Gamma)$, $F$, $h$, as above,
\item $m$ restricts to an orientation-preserving embedding \[R_+(\Gamma)\smallsetminus A(\Gamma)\hookrightarrow r(R\times\{-1\}).\]
\end{enumerate}
 \end{enumerate} 
 The \emph{genus} $g(\data)$ refers to the genus of $R$.
\end{definition}

\begin{remark}
It follows from the formula in \eqref{eqn:genusM'} that $(M,\Gamma,\xi)$ admits a genus $g$ marked closure for every \[g\geq \max\bigg(2, \frac{|\Gamma|-\chi(R_+(\Gamma))+2}{2}\bigg).\] We will denote this maximum by $g(M,\Gamma)$.
\end{remark}

\begin{remark} For a marked closure $\data$ as in Definition \ref{def:smoothclosure}, the pair $(Y,r(R\times\{t\}))$  is  a  closure of $(M,\Gamma)$ in the sense of Kronheimer and Mrowka for any $t\in[-1,1]$.
\end{remark}

 
 \begin{remark} Suppose $\data = (Y,R,r,m,\eta)$ is a marked closure of $(M,\Gamma)$. Then, the tuple \[-\data:=(-Y,-R,r,m,-\eta),\] obtained  by reversing the orientations of $Y$, $R$, and $\eta$, is a marked closure of $-(M,\Gamma):=(-M,-\Gamma),$ where, $r$ and $m$ are the induced embeddings  of $-R\times[-1,1]$ and $-M$ into $-Y$. 
  \end{remark}
 
 

\begin{notation}For the rest of this paper,  $\RR$ will be the \emph{Novikov ring over $\zz$}, given by \[\RR=\bigg\{\sum_{\alpha}c_{\alpha}t^{\alpha}\,\bigg | \,\alpha\in\mathbb{R},\,c_{\alpha}\in\zz,\,\#\{\beta<n|c_{\beta}\neq 0\}<\infty\textrm{ for all } n\in \zz\bigg\}.\] 
\end{notation}

Following Kronheimer and Mrowka \cite{km4}, we made the following definition in \cite{bs3}.

\begin{definition}\label{def:shmt}Given a marked closure $\data = (Y,R,r,m,\eta)$ of $(M,\Gamma)$, the \emph{sutured monopole homology of $\data$} is the $\RR$-module \[\SHMt(\data):=\HMtoc(Y|R;\Gamma_\eta).\]
\end{definition}

Here, $\HMtoc(Y|R;\Gamma_\eta)$ is shorthand for the monopole Floer homology of $Y$ in the ``topmost" $\Sc$ structures relative to $r(R\times\{0\})$, 
\begin{equation}\label{eqn:relativelocal}\HMtoc(Y|R;\Gamma_\eta):=\bigoplus_{\substack{\spc\in\Sc(Y)\\ \langle c_1(\spc), [r(R\times\{0\})]\rangle = 2g(R)-2}} \HMtoc(Y,\spc; \Gamma_{r(\eta\times\{0\})}), \end{equation} where, for each $\Sc$ structure $\mathfrak{s}$,  $\Gamma_{r(\eta\times\{0\})}$ is the local system on the Seiberg-Witten configuration space $\mathcal{B}(Y,\mathfrak{s})$ with fiber  $\RR$ specified by the curve $r(\eta\times\{0\})\subset Y$, as defined in \cite[Section 2.2]{km4}.

In \cite{km4}, Kronheimer and Mrowka proved that the isomorphism class of $\SHMt(\data)$ is an invariant of $(M,\Gamma)$. We strengthened this in \cite{bs3}, proving that the sutured monopole homology groups of any two marked closures of $(M,\Gamma)$ are canonically isomorphic, up to multiplication by a unit in $\RR$. Specifically, for any two marked closures $\data,\data'$ of $(M,\Gamma)$, we construct an isomorphism \[\Psit_{\data,\data'}:\SHMt(\data)\to\SHMt(\data'),\]  well-defined up to multiplication by a unit in $\RR$, such that the modules in $\{\SHMt(\data)\}_{\data}$ and the equivalence classes of maps in $\{\Psit_{\data,\data'}\}_{\data,\data'}$ form a projectively transitive system of $\RR$-modules.\footnote{The collection of marked closures is a proper class rather than a set and so cannot technically serve as the indexing object for a projectively transitive system. One can remedy this by requiring that $Y$ and $R$ be submanifolds of Euclidean space. We will not worry about such issues.} We will review the construction of these maps in Section \ref{sec:well-definedness}.

\begin{definition}
The \emph{sutured monopole homology of $(M,\Gamma)$} is the projectively transitive system of $\RR$-modules $\SHMtfun(M,\Gamma)$ given by the modules  and the equivalence classes above.
\end{definition} 

Sutured monopole homology is functorial in the following sense. Suppose \[f:(M,\Gamma)\to(M',\Gamma')\] is a diffeomorphism of sutured manifolds and $\data' = (Y',R',r',m',\eta')$ is a marked closure of $(M',\Gamma')$. Then \begin{equation}\label{eqn:dataf}\data'_f:=(Y',R',r',m'\circ f,\eta')\end{equation} is a marked closure of $(M,\Gamma)$. Let \[id_{\data'_f,\data'}: \SHMt(\data'_f)\to\SHMt(\data')\] be the identity map on $\SHMt(\data'_f) = \SHMt(\data')$. The equivalence classes of these identity maps can be completed to a morphism (as in Remark \ref{rmk:completesubset}) \[\SHMtfun(f):\SHMtfun(M,\Gamma)\to\SHMtfun(M',\Gamma'),\]  which is an invariant of the isotopy class of $f$. We proved in \cite{bs3} that these morphisms behave as expected under composition of diffeomorphisms, so that  $\SHMtfun$ defines a functor from $\DiffSut$ to $\RPSys{},$ where $\DiffSut$ is the category of balanced sutured manifolds and isotopy classes of diffeomorphisms between them. 

Recall that a \emph{product sutured manifold} is a sutured manifold $(M,\Gamma)$ obtained from a product $(S\times[-1,1],\partial S\times\{0\})$ by rounding corners, for some surface $S$ with boundary. Product sutured manifolds have simple Floer homology, as expressed  below. This fact will be important for us at several points in this paper.

\begin{proposition}
\label{prop:productsutured} If $(M,\Gamma)$ is a product sutured manifold, then $\SHMtfun(M,\Gamma)\cong\RR$.
\end{proposition}

\begin{proof}
Let $F$ be an auxiliary surface for $(M,\Gamma)$ with $g(F)\geq 2$. Thinking of $(M,\Gamma)$ as  obtained from  $(S\times[-1,1],\partial S\times\{0\})$ by rounding corners, we can  form a preclosure of $(M,\Gamma)$ by gluing $F\times[-1,1]$ to $S\times[-1,1]$ according to a map \[h:\partial F\times[-1,1]\to\partial S\times[-1,1]\] of the form $f\times{\rm id}$ for some diffeomorphism $f:\partial F\to \partial S$. This preclosure is then a product $M'=(S\cup F)\times[-1,1]$. To form a marked closure, we take $R=S\cup F$ and  glue $R\times[-1,1]$ to $M'$ by the ``identity" maps 
\[
R\times\{\pm 1\}\to S\times\{\mp 1\}.
\]
An oriented, nonseparating curve $\eta\subset S\cup F$ gives a marked closure \[\data = ((S \cup F)\times S^1, (S\cup F),r,m,\eta).\] Here, we are thinking of $S^1$ as the union of two copies of $[-1,1]$, and $r$ and $m$ as the obvious embeddings. Therefore, \[\SHMt(\data) := \HMtoc((S\cup F)\times S^1|(S\cup F);\Gamma_\eta)\cong \RR,\] by \cite[Lemma 4.7]{km4}. The proposition follows.
\end{proof}

\begin{remark}
\label{rmk:novfield}
In Section \ref{sec:bypasstri}, we will work over the \emph{Novikov field} \[\RR/2\RR:= \RR\otimes_{\zz}\Z/2\Z\] in order to use the surgery exact triangle in monopole Floer homology, which has not been established in characteristic 0. This might alarm the reader familiar with \cite{km4}, where, when working with local coefficients, Kronheimer and Mrowka require that $\RR$ (which is not necessarily the Novikov ring in their case) have no $\zz$-torsion. This condition is imposed to ensure that certain $\mathrm{Tor}$ terms arising in the K{\"u}nneth theorem vanish.  It turns out, however, that we are safe when working in characteristic 2 and using the Novikov field $\RR/2\RR$, as these $\mathrm{Tor}$ terms still vanish; see \cite[Section 2.2]{sivek} for details. 
\end{remark}


\subsection{The monopole Floer contact invariant}
\label{ssec:hm-contact}

In \cite{km}, Kronheimer and Mrowka defined an invariant  of contact structures on closed 3-manifolds which assigns to a closed contact manifold $(Y,\xi)$ a class \[\psi(Y,\xi)\in \HMtoc(-Y,\spc_\xi) \] which depends only on the isotopy class of the contact structure $\xi$. We will use the same notation for the version of this invariant  in monopole Floer homology with coefficients in a local system.  Below, we review three important properties of this  invariant. 

The first is that the invariant vanishes for overtwisted contact structures.

\begin{theorem}[\cite{km}]
\label{thm:hm-ot}
If $(Y,\xi)$ is overtwisted, then $\psi(Y,\xi) = 0$.
\end{theorem}

Next, recall that a \emph{weak symplectic filling} of a closed contact 3-manifold $(Y,\xi)$ is a symplectic manifold $(X,\omega)$ such that $Y=\partial X$ and $\omega|_\xi > 0$. 

\begin{theorem}[\cite{km,kmosz}]
\label{thm:psi-weakly-fillable}
If $(Y,\xi)$ has a weak symplectic filling $(X,\omega)$, then  \[\psi(Y,\xi)\in\HMtoc(-Y, \spc_\xi; \Gamma_{-\eta})\] is a primitive  class (in particular, nonzero) for any local system $\Gamma_{-\eta}$ with fiber $\RR$ given by a curve $\eta\subset Y$ which is Poincar{\'e} dual to  $[\omega]\in H^2(Y;\mathbb{R})$.
\end{theorem}

Finally, suppose $(Y_-,\xi_-)$ and $(Y_+,\xi_+)$ are closed contact $3$-manifolds and recall that an \emph{exact symplectic cobordism} from $(Y_-,\xi_-)$ to $(Y_+,\xi_+)$ is an exact symplectic manifold $(X,\omega=d\lambda)$ with boundary $\partial X = Y_+ \sqcup- Y_-$ for which the restrictions $\lambda|_{Y_\pm}$ are contact forms for $\xi_\pm$. 

\begin{theorem}[\cite{ht}]
\label{thm:ht-functoriality}
Suppose $(X,\omega)$ is an exact symplectic cobordism from $(Y_-,\xi_-)$ to $(Y_+,\xi_+)$.  Then, viewing $X$  as a cobordism from $-Y_+$ to $-Y_-$, the induced  map \[ \HMtoc(X): \HMtoc(-Y_+) \to \HMtoc(-Y_-) \] 
sends $\psi(Y_+,\xi_+)$ to $\pm\psi(Y_-,\xi_-)$.
\end{theorem}



We will frequently apply this theorem via the following corollary.

\begin{corollary} 
\label{cor:contactplusone}Suppose $(Y',\xi')$ is the result of contact $(+1)$-surgery on a Legendrian knot in $(Y,\xi)$ and let $W$ be the corresponding 2-handle cobordism from $Y$ to $Y'$. Then, the map
\[\HMtoc(-W): \HMto(-Y) \to \HMto(-Y') \]
 sends $\psi(Y,\xi)$ to $\pm \psi(Y,\xi')$.  
\end{corollary}

\begin{proof}
If $(Y',\xi')$ is the result of contact $(+1)$-surgery on $K\subset (Y,\xi)$, then $(Y,\xi)$ is the result of contact $(-1)$-surgery on a parallel knot $K'\subset (Y',\xi')$. Let $X$ be the Weinstein 2-handle cobordism corresponding to the latter surgery. By Theorem \ref{thm:ht-functoriality}, the map \[\HMtoc(X): \HMto(-Y) \to \HMto(-Y') \]
 sends $\psi(Y,\xi)$ to $\pm \psi(Y,\xi')$.  But, as a cobordism from  $-Y$ to $-Y'$, $X$ is isomorphic to $-W$.
 \end{proof}

\begin{remark} In \cite{ht}, Theorem \ref{thm:ht-functoriality} is  stated  for monopole Floer homology with coefficients in $\zt$. However, the theorem also holds over $\zz$ and with local coefficients, up to multiplication by a unit in both cases (the same proof extends to these settings without modification). 
\end{remark}


\subsection{Convex surfaces and contact manifolds with boundary}
\label{ssec:convex}

Here, we record some  facts about characteristic foliations and convex surfaces, largely in order to standardize vocabulary. We will assume our reader is  familiar with most of this material. For more comprehensive treatments, see \cite{giroux, giroux3, honda2}.

Suppose $S$ is a smooth surface and $\mathfrak{F}$ is a singular foliation of $S$. An embedded multicurve $\Gamma\subset S$  is said to \emph{divide} $\mathfrak{F}$ if:
\begin{enumerate}
\item $\Gamma$ is transverse to the leaves of $\mathfrak{F}$,
\item $S\ssm \Gamma$ is a disjoint union of \emph{positive} and \emph{negative} regions $S_+\sqcup S_-$,
\item there is a volume form $\omega$ and vector field $w$ on $S$ such that
\begin{enumerate}
\item $w$ directs $\mathfrak{F}$,
\item $w$ points transversely out of $S_+$ along $\Gamma$,
\item $\pm\mathcal{L}_w(\omega)>0$ on $S_{\pm}$.
\end{enumerate}
\end{enumerate}
Given an embedded surface $S\subset (M,\xi)$,  the \emph{characteristic foliation} $S_\xi$ is the singular foliation of $S$ obtained by integrating the vector field $TS\cap \xi$. Giroux showed in \cite{giroux} that $S_\xi$ determines $\xi$ in a neighborhood of $S$, up to contactomorphism fixing $S$. 


A \emph{contact vector field}  is one whose flow preserves $\xi$. An embedded surface $S\subset M$ is \emph{convex} if there exists a contact vector field transverse to $S$. Given a contact vector field $v$ transverse to $S$, the \emph{dividing set} on $S$ associated to $v$ is the multicurve  \[\Gamma=\{p\in S\,|\,v_p\in \xi_p\}.\] This multicurve divides  $S_\xi$ in the sense above. In particular,  we orient $\Gamma$ so that \[\partial \overline{S}_+ = -\partial \overline{S}_- = -\Gamma.\] 
Conversely,  any multicurve which divides $S_\xi$ is the dividing set of some contact vector field (see \cite[Theorem 4.8.5]{geiges}). The space of such multicurves  is contractible (see \cite{massot2}); in particular, the isotopy class of $\Gamma$ is independent of $v$.

A contact structure on $S\times \R$ is  called \emph{vertically invariant}  if $\partial_t$ is a contact vector field. A contact vector field $v$ transverse to $S\subset (M,\xi)$ defines (after cutting off $v$ away from $S$ using a Hamiltonian) a tubular  neighborhood $S\times\R$ of $S=S\times\{0\}$ in which $v$ is identified with $\partial_t$. We will refer to such a neighborhood  as a  \emph{vertically invariant neighborhood} of $S$.

Giroux's Flexibility Theorem below expresses the idea that the crucial information about a contact structure in the neighborhood of a convex surface $S$ is encoded by the dividing set.


\begin{theorem}[\cite{giroux}]
\label{thm:flexibility}
Suppose $S\subset (M,\xi)$ is a  convex surface with dividing set $\Gamma$ for some contact vector field $v$. Let $\mathfrak{F}$ be a singular foliation divided by $\Gamma$ and let $N$ be any neighborhood of $S$. Then there is an isotopy of embeddings $\varphi_r:S\to N$, $r\in[0,1]$, such that
\begin{enumerate}
\item $\varphi_0$ is the inclusion $S\subset M$,
\item each $\varphi_r(S)$ is transverse to $v$ (hence, convex) with dividing set $\varphi_r(\Gamma)=\Gamma$,
\item the characteristic foliation $(\varphi_1(S))_\xi$ agrees with $\varphi_1(\mathfrak{F})$. 
\end{enumerate}
\end{theorem}

We will rely heavily on Giroux's Uniqueness Lemma below.

\begin{lemma}[\cite{giroux3}]
\label{lem:uniqueness}
Suppose  $S$ is a surface and $\xi_0$ and $\xi_1$ are two contact structures on $S\times[0,1]$ which induce the same characteristic foliations on $S\times\{0,1\}$. Suppose each $S\times\{t\}$ is convex with respect to both  $\xi_0$ and $\xi_1$, and contains a multicurve $\Gamma_t$ which varies continuously with $t$ and divides the characteristic foliations $(S\times\{t\})_{\xi_0}$ and $(S\times\{t\})_{\xi_1}$. Then $\xi_0$ and $\xi_1$ are isotopic by an isotopy which is stationary on $S\times\{0,1\}$.\end{lemma}


\begin{remark}
The isotopy Giroux constructs is not necessarily stationary on $S\times\{0,1\}$. One can arrange this, however (see \cite[Lemma 4.9.2]{geiges}). 
\end{remark}


Suppose $M$ is a manifold with boundary and  $\mathfrak{F}$ is a singular foliation of $\partial M$. Let $\rm{Cont}(M,\mathfrak{F})$ be the set of contact structures on $M$ for which $\partial M$ is convex with $(\partial M)_\xi= \mathfrak{F}$. The  following  is due to Honda \cite[Proposition 4.2]{honda2}.

\begin{proposition}
\label{prop:differentfoliations} If $\mathfrak{F}_0$ and $\mathfrak{F}_1$ are two characteristic foliations of $\partial M$ divided by the same multicurve $\Gamma$, then there is a canonical bijection  $f_{01}:\pi_0({\rm Cont}(M, \mathfrak{F}_0)) \rightarrow \pi_0({\rm Cont}(M, \mathfrak{F}_1))$.
\end{proposition}

In the above proposition, the term ``canonical" means that $f_{00} = {\rm id}$ and $f_{02} = f_{12}\circ f_{01}$. Moreover, this bijection sends tight contact structures to tight contact structures. We outline  the proof of this proposition  below (our proof is slightly different in form but not in substance from Honda's) as elements of the proof will be useful later on.

\begin{proof}[Proof of Proposition \ref{prop:differentfoliations}]
Suppose $\xi_0$ is a contact structure on $M$ with characteristic foliation $(\partial M)_{\xi_0} = \mathfrak{F}_0$. We claim that there exists a  contact structure $\xi_{01}$ on $\partial M\times[0,1]$ such that: 
\begin{enumerate}
\item the restriction $\xi_{01}|_{\partial M\times\{0\}}=\xi_0|_{\partial M}$,
\item  the characteristic foliation $(\partial M \times \{1\})_{\xi_{01}}=\mathfrak{F}_1$,
\item each $\partial M\times\{t\}$ is convex and $\Gamma\times\{t\}$ divides $(\partial M\times\{t\})_{\xi_{01}}$,
\item $\partial_t$ is a contact vector field near the boundary.
\end{enumerate} 
This is a relatively easy application of Theorem \ref{thm:flexibility}. 

\begin{remark}
\label{rmk:anycontactstructure}For any contact structure $\xi_1$ defined near $\partial M$ with $(\partial M)_{\xi_1}=\mathfrak{F}_1$, we can arrange that $\xi_{01}$ restricts to $\xi_1$ on $\partial M\times\{1\}$. 
\end{remark}

\begin{remark}
\label{rmk:partialagreement}If $\mathfrak{F}_0$ agrees with $\mathfrak{F}_1$ on some open subset $A\subset \partial M$,  then we can take   $\partial_t$ to be a contact vector field  on $A\times[0,1]$. 
\end{remark}

To define $f_{01}$, we  first choose a vertically invariant collar neighborhood $\partial M\times(-\infty,0]$ of $\partial M = \partial M\times\{0\}$ such that $\Gamma$ is the dividing set associated to $\partial_t$.  Let $(M',\xi')$ be the contact manifold formed by gluing $(\partial M\times[0,1],\xi_{01})$ to $(M,\xi_0)$ along $\partial M\times\{0\}$ according to the identity map  and the obvious collars. Let  \[\varphi:M'\rightarrow M\] be the smooth map which is  the identity outside of $\partial M\times (-\infty,1]$ and sends $(x,t)$ to $(x,t-1)$ for $(x,t)\in \partial M\times(-\infty,1]$. We define $f_{01}(\xi_0)$ to be the  contact structure $\xi_1 = \varphi_*(\xi')$. Note that $(\partial M)_{\xi_1} = \mathfrak{F}_1$, as desired. 

\begin{remark} We   say that two contact structures $\xi_0$ and $\xi_1$ on $M$ are \emph{related by flexibility}  if $\xi_1$ is obtained from $\xi_0$ in this way.
\end{remark}

\begin{remark} 
\label{rmk:partialagreementcontact} Note that $\xi_1 = \xi_0$ outside of $\partial M\times[-1,0]$. If $\mathfrak{F}_0=\mathfrak{F}_1$ on some open subset $A\subset \partial M$,  then we can  arrange, per Remark \ref{rmk:partialagreement}, that $\xi_1 = \xi_0$ outside of $(\partial M \ssm A)\times[-1,0]$.  
\end{remark}

Lemma \ref{lem:uniqueness} implies that the contact structure $\xi_{01}$ is unique, up to isotopy stationary on the boundary of $\partial M \times[0,1]$. If follows that $f_{01}(\xi_0)$ is independent of $\xi_{01}$, up to isotopy stationary on $\partial M$. The fact that the space of vertically invariant collars as above  is connected (in fact, contractible) implies that $f_{01}(\xi_0)$ is  independent of  the chosen collar. Finally, it is clear that if $\xi_0$ and $\xi_0'$ are isotopic, then so are $f_{01}(\xi_0)$ and $f_{01}(\xi_0').$ Thus, $f_{01}$ is well-defined as a map from $\pi_0({\rm Cont}(M, \mathfrak{F}_0))$ to $\pi_0({\rm Cont}(M, \mathfrak{F}_1))$. It is clear that $f_{00} = {\rm id}$. The transitivity $f_{02} = f_{12}\circ f_{01}$ is an easy application of Lemma \ref{lem:uniqueness}. Note that these two relations imply that $f_{01}$ is a bijection with inverse $f_{10}$. 
\end{proof}

The convex surfaces considered to this point have been closed. A convex surface in $(M,\xi)$ with \emph{collared Legendrian boundary} is a properly embedded surface $S\subset M$ with Legendrian boundary, equipped with a collar neighborhood $\partial S\times[0,1]$ of $\partial S=\partial S\times\{0\}$ on which  $\xi$ is $[0,1]$-invariant (see \cite{honda2}). In particular, the curves $\partial S\times\{s\}$ are  Legendrian; these curves are called \emph{rulings}. Moreover, there is an even number of Legendrian arcs in $\partial S\times[0,1]$ of the form $\{p\}\times[0,1]$ for  $p\in \partial S$; these are called \emph{divides}.  Note that, for any transverse contact vector field $v$, the  dividing set on $\partial S\times[0,1]$ consists of arcs  parallel to and alternating with the  divides, as  in Figure \ref{fig:collared}.

\begin{figure}[ht]
\labellist
\hair 2pt\tiny

\pinlabel $0$ at 5.9 31
\pinlabel $1$ at 18 40
\pinlabel $-1$ at -4 3.5
\pinlabel $1$ at -2 21.5
\endlabellist
\centering
\includegraphics[width=6.7cm]{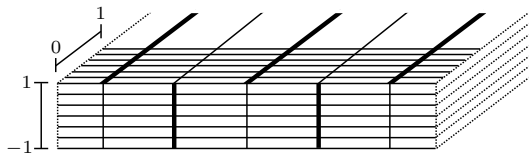}
\caption{The neighborhood $\partial S\times[0,1]_s\times[-1,1]_t\subset S\times[-1,1]_t$ for a convex surface $S$ with collared Legendrian boundary. The arcs in bold belong to the  dividing sets associated to the contact vector fields $\partial_t$ and $\partial_s$; the others are  rulings and  divides. Note that the  divides share endpoints with the dividing sets. }
\label{fig:collared}
\end{figure}

\begin{lemma}
\label{lem:existenceproduct}
Suppose $S$ is a surface with boundary, and let $\Gamma$ be a nonempty collection of oriented, disjoint, properly embedded curves and arcs on $S$ such that $S\ssm \Gamma=S_+\sqcup S_-$ with \[\partial \overline{S}_+ = -\partial \overline{S}_-=\Gamma.\]
Then there exists   a  $[-1,1]$-invariant contact structure on $S\times [-1,1]$ for  which each $S\times\{t\}$ is convex with collared Legendrian boundary $\partial S\times[0,1]\times\{t\}$ and dividing set $\Gamma\times\{t\}$. 
\end{lemma} 

Although this  result is well-known to experts, we sketch a proof below since we could not find one in the literature.

\begin{proof}
Let $(S' = S\cup -S, \Gamma' = \Gamma \cup -\Gamma)$ be the double of $(S,\Gamma)$. Then $ S'\ssm \Gamma'$ is a disjoint union $ S'_+\sqcup  S'_-$ with \[\partial \overline{S'}_+ = -\partial \overline{S'}_-=\Gamma',\]  and it is not  hard to construct a $[-1,1]$-invariant contact structure $\xi'$ on $S'\times[-1,1]$ for which each $S'\times\{t\}$ is convex with dividing set $\Gamma'\times\{t\}$. Note that  $\partial S\subset S'$ is  \emph{nonisolating}, meaning that each component of $S'\ssm \partial S$ intersects $\Gamma'$ nontrivially. It follows that there is a singular foliation $\mathfrak{F}$ of $S'$ which is divided by $\Gamma'$ and contains $\partial S$ as a union of leaves (see \cite{honda2}). In fact, we can assume that $\mathfrak{F}$ restricts to a $[0,1]$-invariant foliation on $\partial S\times[0,1]\subset S\subset S'$. By Proposition \ref{prop:differentfoliations}, there exists a contact structure $\xi'$ on $\partial S'\times[-1,1]$ such that the characteristic foliation of $\xi'$ on $S'\times\{1\}$ is equal to $\mathfrak{F}$. Changing notation, let us now denote by $(S'\times[-1,1],\xi')$  a vertically invariant neighborhood of $S'\times\{1\}$. Then the restriction of $\xi'$ to $S\times[-1,1]$ is  the desired contact structure.
 \end{proof}

\begin{remark} Note that, for any contact structure  as in Lemma \ref{lem:existenceproduct},  $\partial_t$ and $\partial_s$ are contact vector fields  on $\partial S\times[0,1]_s\times[-1,1]_t\subset S\times[-1,1]_t.$
\end{remark}

The lemma below asserts that, for a given $\Gamma$, there is essentially only one contact structure on $S\times[-1,1]$ satisfying the conditions in Lemma \ref{lem:existenceproduct}.

\begin{lemma} 
\label{lem:productuniqueness}Suppose $\xi$ and $\xi'$ are contact structures on $S\times[-1,1]$ as in Lemma \ref{lem:existenceproduct}. Then, up to flexibility, $\xi$ and $\xi'$ are isotopic.
\end{lemma}

Although this  is result also familiar to experts, we include a proof below in order to more precisely describe the isotopy alluded to in the lemma (see Remark \ref{rmk:isotopykind}).

\begin{proof}
Let $\mathfrak{F}$ and $\mathfrak{F}'$ be the characteristic foliations on $S$ induced by $\xi$ and $\xi'$, and let $C$ and $C'$ be the collars of $\partial S$ associated to $\xi$ and $\xi'$.
There exists an isotopy $\varphi_r:(S,\Gamma)\to (S,\Gamma)$, $r\in[0,1]$, such that $\varphi_0 = {\rm id}$ and $\varphi_1$ sends the restriction of $\mathfrak{F}'$ on $C'$ to the restriction of $\mathfrak{F}$ on $C$. Let $\mathfrak{F}'' = (\varphi_1)^{-1}(\mathfrak{F}).$  Let $\xi''$ be the contact structure, obtained from $\xi'$ by flexibility, whose characteristic foliation  on $S\times\{\pm 1\}$ agrees with $\mathfrak{F}''$. Note that we can apply Proposition \ref{prop:differentfoliations} here, even though $S$ has boundary, because $\mathfrak{F''}$  agrees with $\mathfrak{F}'$ on $C'$. In fact, we may assume that $\xi''$ agrees with $\xi'$ on $C'\times[-1,1]$, as in Remark \ref{rmk:partialagreementcontact}. To prove the lemma, it suffices to show that $\xi''$ is isotopic to $\xi$. We do this in three steps.

The isotopy $\varphi_r$ extends to an isotopy $\varphi_r\times id$ of $S\times[-1,1]$. Let $\xi''' = (\varphi_1\times id)_*(\xi'')$. Then the characteristic foliation of $\xi'''$ on $S\times[-1,1]$ agrees with that of $\xi$. Moreover, $\xi'''$ is $[-1,1]$-invariant on $C\times[-1,1]$. Thus, $\xi'''$ is isotopic to a contact structure $\xi''''$ which agrees with $\xi$ on $C\times[-1,1]$ by a $[-1,1]$-invariant isotopy which preserves the characteristic foliation on each $S\times\{t\}$. This is essentially Giroux's Reconstruction Lemma \cite{giroux3}. Since each $S\times\{t\}$ is convex with dividing set $\Gamma\times\{t\}$ for both $\xi''''$ and $\xi$ and the characteristic foliations of these two contact structures agree on $S\times\{\pm 1\}$, Lemma \ref{lem:uniqueness} asserts that $\xi''''$ and $\xi$ are isotopic by an isotopy which is stationary on $S\times\{\pm 1\}$. In fact, we can take this isotopy to be stationary on $C\times[-1,1]$ since $\xi''''$ and $\xi$ agree there.
\end{proof}

\begin{remark}
\label{rmk:isotopykind}
It is clear from the proof that the isotopy alluded to in Lemma \ref{lem:productuniqueness} can be taken to be of the form $\varphi_r\times id$ near $\partial S\times[-1,1]$, for some isotopy $\varphi_r$ of $(S,\Gamma).$
\end{remark}

\section{A contact invariant in sutured monopole homology}
\label{sec:monopoleinvt}

Suppose $(M,\Gamma)$ is a balanced sutured manifold and $\xi$ is a contact structure on $M$ such that $\partial M$ is convex and $\Gamma$ divides the characteristic foliation of $\partial M$ induced by $\xi$. We will refer to the triple $(M,\Gamma,\xi)$ as a \emph{sutured contact manifold}. In this section, we construct the contact invariants \[\invt^g(M,\Gamma,\xi) \in \SHMtfun(-M,-\Gamma).\]  (We will show that these elements are well-defined and agree for large $g$ in Section \ref{sec:well-definedness}.) As mentioned in the introduction, the rough  idea  is to extend $\xi$ to a contact structure $\bar{\xi}$ on a genus $g$ closure of $(M,\Gamma)$ and define $\invt^g(M,\Gamma,\xi)$ in terms of the monopole Floer contact  invariant of this closed contact manifold.  We will make this precise below. 

\subsection{Contact preclosures}
\label{ssec:contactclosure}

We begin by studying a certain class of contact structures on  preclosures of $M$.

\begin{definition}
Let $F$ be a connected, orientable surface with nonempty boundary and positive genus. An \emph{arc configuration} $\mathcal{A}$ on $F$ consists of an embedded,  nonseparating curve $c$ on $F$ together with disjoint embedded arcs $a_1,\dots, a_n$ such that:
\begin{enumerate}
\item  every $a_i$ has one endpoint on $\partial F$ and another on $c$,
\item $\inr(a_i)\cap (c\cup \partial F)=\emptyset$ for each $i$,
\item every component of $\partial F$ contains an endpoint of some $a_i$.
\end{enumerate}
\end{definition}

Suppose $\mathcal{A}$ is an arc configuration on  $F$ and let $N(\mathcal{A})$ be a regular neighborhood of $\mathcal{A}$. Let $\Gamma_{\mathcal{A}}$ be  the collection of oriented, properly embedded arcs in $F$ given by \begin{equation}\label{eqn:divset}\Gamma_{\mathcal{A}}=-\overline{(\partial N(\mathcal{A})\ssm\partial F)}.\end{equation} Let $\Xi_{\mathcal{A}}$ be a $[-1,1]$-invariant contact structure on $F\times[-1,1]$ for which each $F\times\{t\}$ is convex with collared Legendrian boundary  and dividing set $\Gamma_{\mathcal{A}}\times\{t\}$. There is a unique such $\Xi_{\mathcal{A}}$ up to flexibility and isotopy, by Lemma \ref{lem:productuniqueness}.
Note that the dividing set   on $\partial F\times[-1,1]$ is   of the form $\{p_1,\dots, p_{2n}\}\times[-1,1]$ for  points $p_i\in \partial F$. Moreover, the negative region on $F\times\{1\}$ is   $N(\mathcal{A})\times \{1\}$. See Figure \ref{fig:arcconfiguration} for an example in which  this and the negative region $\partial F\times[-1,1]$ have been shaded.

\begin{figure}[ht]
\centering
\includegraphics[width=13.6cm]{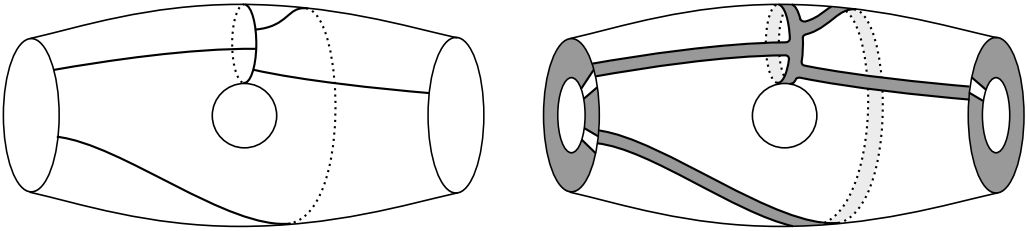}
\caption{Left, an arc configuration on a genus one surface with two boundary components. Right, the thickened surface with the negative regions on $F\times\{1\}$ and $\partial F\times[-1,1]$ shaded.}
\label{fig:arcconfiguration}
\end{figure}

Suppose   the surface $F$ above is an auxiliary surface for $(M,\Gamma)$, and consider the preclosure $M'=M\cup_h F\times[-1,1]$ associated to some  neighborhood $A(\Gamma)$ and    diffeomorphism \[h:\partial F\times[-1,1] \rightarrow A(\Gamma).\]  We would like to define a contact structure $\xi'$ on $M'$ in terms of   $\Xi_{\mathcal{A}}$ and $\xi$. To do so, we   first perturb $\xi$ in a neighborhood of $A(\Gamma)$ so that $h$ identifies  $\Xi_{\mathcal{A}}$ with this perturbed contact structure.  This enables us to glue $F\times[-1,1]$ to $M$ via $h$ contact geometrically. We will    show (Theorem \ref{thm:preclosurewelldefined}) that the resulting \emph{contact preclosure} $(M',\xi')$ is independent, up to flexibility and contactomorphism, of the arc configuration $\mathcal{A}$ and the other choices involved. 

We start by describing the perturbation of $\xi$. Label the components of $\Gamma$ by $\Gamma_1,\dots,\Gamma_m$. Each $\Gamma_i$   has an annular neighborhood $B_i\subset\partial M$ on which the leaves of the characteristic foliation  are cocores   with no singularities. Let $B(\Gamma)$ be the union of these $B_i$. Let $A_i$ be the  component of $A(\Gamma)$ containing $\Gamma_i$. We will assume that $A_i\subset \inr (B_i)$. The map $h$ identifies $A_i$ with $\partial_iF\times[-1,1]$ for some component $\partial_iF$ of $\partial F$. In this way, the ordering on the components of $\Gamma$ induces an ordering on the components of $\partial F$. Let $k_i$ be the number of arcs in $\mathcal{A}$ with an endpoint on $\partial_iF$. The dividing set  on the annulus $\partial_iF\times[-1,1]$ thus consists of $2k_i$  cocores. 

The rough idea is  to perturb $\xi$ to a contact structure $\xi_{h}$ whose dividing set $\Gamma_{h}$  restricts to $2k_i$  cocores of $A_i$, such that $h$ identifies the positive  region of $\partial F\times[-1,1]$  with the negative region of $A(\Gamma)$ determined by $\Gamma_{h}$. We do  essentially this, but work on the level of characteristic foliations rather than dividing curves for added  control on the eventual contactomorphisms between different  preclosures. 

For  $i=1,\dots, m$, let $C_i\subset \inr(A_i)$ be an annulus such that the intersection $\Gamma_i\cap C_i$ consists of  $2k_i$ cocores of $C_i$. We next choose an isotopy \[\varphi_r:\partial M\rightarrow \partial M,\,r\in[0,1]\]  which ``straightens out" these $C_i$. Precisely, we require that  $\varphi_0={ id}$, that $\varphi_1(C_i) =A_i$,   that $h$ identifies the positive region of $\partial F\times[-1,1]$ with the negative region of $A(\Gamma)$ determined by the dividing set $\varphi_1(\Gamma)$, and that each $\varphi_r$ restricts to the identity outside of $B(\Gamma)$. 


\begin{figure}[ht]
\labellist
\hair 2pt\tiny
\pinlabel $\Gamma_i$ at 44 38.5
\pinlabel $\varphi_1(\Gamma_i)$ at 184 119
\pinlabel $C_i$ at 77 122
\pinlabel $\varphi_1(C_i)=A_i$ at 210 -5
\pinlabel $B_i$ at 58 -5
\endlabellist
\centering
\includegraphics[width=8.4cm]{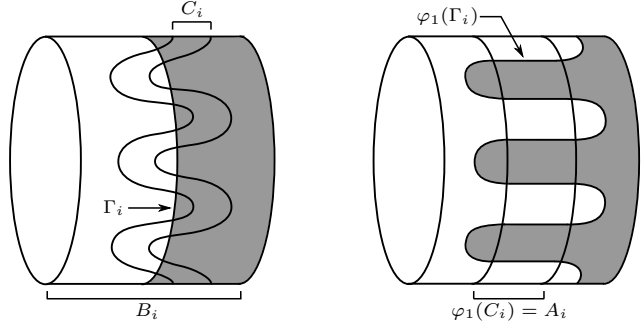}
\caption{Left, a view of  $C_i\subset B_i$, where $k_i=3$. Right,   $\varphi(\Gamma_i)=\varphi_1(\Gamma_i)$ and  $\varphi_1(C_i)=A_i$. The negative regions are shaded. }
\label{fig:straighten3}
\end{figure}


Choose a vertically invariant collar  $\partial M \times (-\infty,0]$ of $\partial M=\partial M\times\{0\}$ such that $\Gamma$ is the dividing set associated to $\partial_t$. The isotopy $\varphi_r$  induces  a diffeomorphism \[\varphi:M\rightarrow M\]  as follows. Let \[r:(-\infty,0]\rightarrow [0,1]\] be a smooth function with
\[ 
r(t) = 
\begin{cases}
0 & \text{for }  t\leq -2\\
1 & \text{for } t\geq -1.
\end{cases}
\]
We define $\varphi$ to be the map defined by \begin{equation}\label{eqn:phiextension}\varphi(x,t) = (\varphi_{r(t)}(x),t)\end{equation} for $(x,t)\in \partial M\times(-\infty, 0]$ and by the identity outside of $\partial M\times(-\infty, 0]$. Note that $\varphi$ restricts to $\varphi_1$ on $\partial M$ and  to the identity outside of  $B(\Gamma)\times[-2,0]$.

Define $\xi_0:=\varphi_*(\xi)$. Let $\mathfrak{F}_1$ be a foliation of $\partial M$ divided by $\varphi(\Gamma) = \varphi_1(\Gamma)$ which agrees  with  $h(\partial F\times[-1,1])_{\Xi_{\mathcal{A}}}$ on $A(\Gamma)$ and   with $\mathfrak{F}_0 = (\partial M)_{\xi_0}$ outside of $B(\Gamma)$, as illustrated in Figure \ref{fig:foliation2}. We   apply flexibility, as in Proposition \ref{prop:differentfoliations}, with respect to the collar determined by  the contact vector field $\varphi_*(\partial_t)$ for $\xi_0$, to obtain a contact structure $\xi_1=f_{01}(\xi_0)$  with $(\partial M)_{\xi_1} = \mathfrak{F}_1$. In doing so, we can arrange that  $\xi_1=h_*(\Xi_{\mathcal{A}})$ on $A(\Gamma)$, by Remark \ref{rmk:anycontactstructure}. Since $\mathfrak{F}_0$ and $\mathfrak{F}_1$ agree outside of $B(\Gamma)$, we can assume that $\partial_t$ is a contact vector field for $\xi_{01}$ on the complement of $B(\Gamma)\times[0,1]$ in the product $\partial M\times[0,1]$ used to define the map $f_{01}$. It then follows from the construction of $f_{01}$ that $\xi_1=\xi$ outside of $B(\Gamma)\times[-3,0]$, since $\xi_0$ and $\varphi_*(\partial_t)$ agree with $\xi$ and $\partial_t$, respectively, outside of $B(\Gamma)\times[-2,0]$.
We will hereafter denote $\xi_1$ by $\xi_h$ to indicate its dependence on $h$. 
We   may now glue $(F\times[-1,1], \Xi_{\mathcal{A}})$ to $(M,\xi_h)$ via $h$ contact geometrically. We perform this gluing, rounding corners as illustrated in Figure \ref{fig:straighten}, to obtain a \emph{contact preclosure} \[(M',\xi') = (M\cup_h F\times[-1,1], \xi_h\cup \Xi_{\mathcal{A}})\] of $(M,\Gamma,\xi)$. 

\begin{remark}
\label{rmk:divset}
The dividing set for $\xi'$ consists of two parallel nonseparating curves on each component $\partial_{\pm} M'$ of $\partial M'$. The negative region on $\partial_+ M'$ is the annulus bounded by these curves and retracts onto a regular neighborhood of $c\times\{1\}$, where $c$ is the curve in $\mathcal{A}$. Likewise, the positive region on $\partial_- M'$ is an annulus which retracts onto a neighborhood of $c\times\{-1\}$.
\end{remark}

\begin{figure}[ht]

\centering
\includegraphics[width=7.6cm]{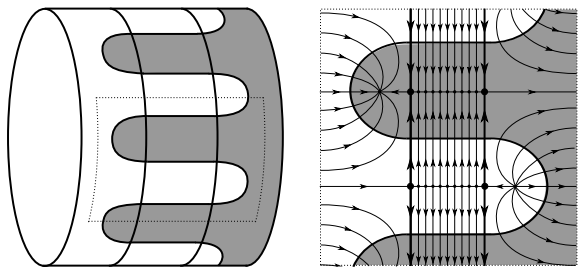}
\caption{Left, a view of $A_i$, $B_i$, and $\varphi(\Gamma_i)=\varphi_1(\Gamma_i)$. Right,   a view of the  foliation $\mathfrak{F}_1$ in the portion of $B_i$ contained within the dotted rectangle. }
\label{fig:foliation2}
\end{figure}

\begin{figure}[ht]
\labellist
\hair 2pt\tiny

\pinlabel $h$ at 108 13
\pinlabel $h$ at 108 155
\endlabellist
\centering
\includegraphics[width=12.4cm]{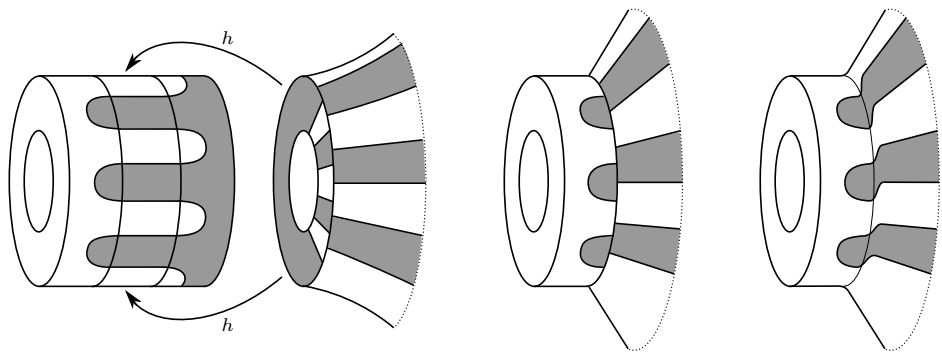}
\caption{Left, gluing $\partial F\times[-1,1]$ to $M$ along $A(\Gamma)$, as viewed near some $B_i\times (-\infty,0]$. Middle, the glued manifold. Right, the contact preclosure with convex boundary after rounding corners.}
\label{fig:straighten}
\end{figure}

The rest of this subsection is devoted to proving  the well-definedness of $(M',\xi')$. Our main result is the following.

\begin{theorem}
\label{thm:preclosurewelldefined} Suppose $(M'_1,\xi'_1)$ and $(M'_2,\xi'_2)$ are  contact preclosures of $(M,\Gamma,\xi)$ defined using   auxiliary surfaces of the same genus. Then, up to flexibility,  $(M_1',\xi_1')$ and $(M_2',\xi_2')$ are contactomorphic by a map  isotopic to one that restricts to the identity on $M\ssm N(\Gamma)$ for some regular neighborhood $N(\Gamma)$ of $\Gamma$.
\end{theorem}

\begin{proof}

First, suppose  all  choices in the constructions of $(M'_1,\xi'_1)$ and $(M'_2,\xi'_2)$ are the same except for that of the vertically invariant collar  of $\partial M$ used to define $\xi_h$. Suppose $\xi_{h,1}$ and $\xi_{h,2}$ are the contact structures on $M$ defined from two  different collars. The connectedness of the space of such  collars  implies that $\xi_{h,1}$ and $\xi_{h,2}$ are isotopic by an isotopy stationary on $\partial M$. It follows that $(M'_1,\xi'_1)$ and $(M'_2,\xi'_2)$ are contactomorphic by a map  isotopic to one that restricts to the identity on $M$, as desired.
It therefore suffices to prove Theorem \ref{thm:preclosurewelldefined} in the case that $(M'_1,\xi'_1)$ and $(M'_2,\xi'_2)$ are defined using the same  collar. We  will assume  below that this is the case. 
We will also continue to think of the contact structure $\Xi_{\mathcal{A}}$ as being completely determined by the arc configuration $\mathcal{A}$. This is fine for the purpose of this proof: since any two such $\Xi_{\mathcal{A}}$ are related by flexibility and isotopy as in Lemma \ref{lem:productuniqueness} and Remark \ref{rmk:isotopykind},  the contact preclosures formed from any two such $\Xi_{\mathcal{A}}$ are related as claimed in the theorem, assuming  all other choices are the same.

Below, we prove Theorem \ref{thm:preclosurewelldefined} in the case that $(M'_1,\xi'_1)$ and $(M'_2,\xi'_2)$ are built using auxiliary surfaces  with \emph{isomorphic} arc configurations.

\begin{definition}
Suppose $\mathcal{A}_1$ and $\mathcal{A}_2$ are  arc configurations on   $F_1$ and $F_2$, and suppose the boundary components of each $F_j$ have been ordered.
We say that $\mathcal{A}_1$ and $\mathcal{A}_2$ are  \emph{isomorphic} if there is a diffeomorphism from  $(F_1,\mathcal{A}_1)$ to $(F_2,\mathcal{A}_2)$  which respects these boundary orderings.
\end{definition}

For $k=1,2$, suppose   $(M'_k,\xi'_k)$ is  defined using the auxiliary surface $F_k$,  the arc configuration $\mathcal{A}_k$,   the  neighborhoods $A(\Gamma)_k = \cup_i A_{i,k}$ and $B(\Gamma)_k = \cup_i B_{i,k}$, and  the diffeomorphism $h_k$. For $i=1,\dots,m$, let  $B_i$ be an annular neighborhood of $\Gamma_i$ containing  $B_{i,1}\cup B_{i,2}$ on which the leaves of $(\partial M)_\xi$ are cocores  with no singularities, and let $B(\Gamma)=\cup_i B_i$. Let us assume  that $\mathcal{A}_1$ and $\mathcal{A}_2$ are isomorphic (where the boundary components of  $F_k$ are ordered according to $h_k$ and the ordering of the components of $\Gamma$, as usual) by an isomorphism  \[g:(F_1,\mathcal{A}_1)\to(F_2,\mathcal{A}_2).\] 
Note that $g$ induces a canonical isotopy class of  contactomorphisms \[\tilde g:(F_1\times[-1,1],\Xi_{\mathcal{A}_1})\to (F_2\times[-1,1],\Xi_{\mathcal{A}_2})\] for which $h_2\circ\tilde g\circ h_1^{-1}$ restricts to a diffeomorphism from  $A_{i,1}$ to $A_{i,2}$ for  $i=1,\dots,m$.
Let  \[\varphi_r:\partial M\rightarrow \partial M,\,r\in[0,1],\] be an isotopy supported in $B(\Gamma)$, such that  $\varphi_0={ id}$ and $\varphi_1$ restricts to the map \[\label{eqn:hgh}h_2\circ\tilde g\circ h_1^{-1}: A(\Gamma)_1\rightarrow A(\Gamma)_2.\] Let \[\varphi:M\to M\] be the diffeomorphism of $M$ extending $\varphi_1$ defined as  in \eqref{eqn:phiextension}. By construction,   the characteristic foliations  of $\xi_{h_1}$ and $(\varphi_*)^{-1}(\xi_{h_2})$ on $\partial M$   agree on   $A(\Gamma)_1$ and outside of   $B(\Gamma)$. Let  $\xi_{h_1}'$ be a contact structure obtained from  $\xi_{h_1}$ by flexibility such that \[(\partial M)_{\xi_{h_1}'} = (\partial M)_{(\varphi_*)^{-1}(\xi_{h_2})}\]  and such that $\xi_{h_1}'$ agrees with $\xi_{h_1}$ on  $A(\Gamma)_1\times[-3, 0]$ and outside  of  $B(\Gamma)\times[-3, 0]$.  The contact preclosure $(M_1',\xi_1'')$ constructed from $\xi_{h_1}'$ is  then related to  $(M_1',\xi_1')$ by flexibility. 

To complete the proof of Theorem \ref{thm:preclosurewelldefined} in this case, it therefore suffices to show that $(M_1',\xi_1'')$ is contactomorphic to $(M_2',\xi_2')$ by a map   isotopic to one that restricts to the identity outside of $B(\Gamma) \times [-3,0]$. 
For this, it suffices to  show that $(M,\xi_{h_1}')$ is contactomorphic to $(M,\xi_{h_2})$ by a map  isotopic to one which restricts to  the identity outside of $B(\Gamma)\times [-3,0]$, through maps which restrict to $h_2\circ\tilde g\circ h_1^{-1}$ on $A(\Gamma)_1$. Indeed, a contactomorphism from $(M,\xi_{h_1}')$ to $(M,\xi_{h_2})$ of this form  extends to the desired contactomorphism from $(M_1',\xi_1'')$  to $(M_2',\xi_2')$ by the map  $\tilde g$.  Since $(\varphi_*)^{-1}(\xi_{h_2})$ is already contactomorphic to $\xi_{h_2}$ by  such a map (namely, $\varphi$), it suffices to show that $\xi_{h_1}'$ and $(\varphi_*)^{-1}(\xi_{h_2})$ are isotopic by an isotopy stationary on $\partial M$. To see this, let $d_i$ be one of the boundary components of $ B_i$. For each $t$, the multicurve $\cup_i \,(d_i\times\{t\})$  divides the characteristic foliations on $\partial M\times\{t\}$ induced by  $\xi_{h_1}'$ and $(\varphi_*)^{-1}(\xi_{h_2})$. Since these two contact structures induce the same characteristic foliations on $\partial M$ and agree outside of $\partial M\times[-3,0]$, Lemma \ref{lem:uniqueness}  implies that $\xi_{h_1}'$ and $(\varphi_*)^{-1}(\xi_{h_2})$ are isotopic by an isotopy stationary on $\partial M$ (and outside of $\partial M\times[-3, 0]$), as desired.

It remains to prove Theorem \ref{thm:preclosurewelldefined} in the case that $(M'_1,\xi'_1)$ and $(M'_2,\xi'_2)$ are defined using  auxiliary surfaces of the same genus  with \emph{nonisomorphic} arc configurations. For this, we will need a way of relating nonisomorphic configurations. 

 \begin{definition}
 \label{def:standard}
Suppose $\mathcal{A} = \{c,a_1,\dots,a_m\}$ is an arc configuration on a surface $F$ with ordered boundary components  $\partial_1 F,\dots,\partial_m F$, and suppose the arcs have been labeled so that      $a_i$ is the (unique)  arc meeting $\partial_i F$. Suppose further that when $c$ is traversed  according to one of its two orientations, the arcs $a_1, \dots, a_m$ appear  ``locally" to the left of $c$ and in that cyclic order, as depicted in Figure \ref{fig:standard}. Such an arc configuration is called \emph{standard}. 
\end{definition}

\begin{figure}[ht]
\labellist
\hair 2pt\tiny
\pinlabel $c$ at -7 52
\pinlabel $a_1$ at 25 78
\pinlabel $a_2$ at 53 72
\pinlabel $a_3$ at 83 70
\pinlabel $a_m$ at 135 77
\endlabellist
\centering
\includegraphics[width=3cm]{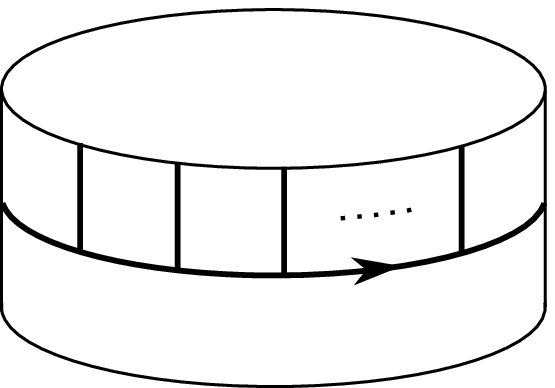}
\caption{A standard arc configuration  in a neighborhood of $c$.}
\label{fig:standard}
\end{figure}

\begin{lemma}
\label{lem:isoarcconfig}
If $\mathcal{A}_1$ and $\mathcal{A}_2$ are standard arc configurations on  surfaces  $F_1$ and $F_2$ of the same genus and with the same number of boundary components, then $\mathcal{A}_1$ and $\mathcal{A}_2$ are isomorphic. \end{lemma}

\begin{proof}
For $k=1,2$, let $F_k'$ be the surface obtained by cutting $F_k$ open along the curve and arcs in $\mathcal{A}_k$. The first condition in Definition \ref{def:standard} implies that   $F_1'$ and $F_2'$ have the same genus and two boundary components. Moreover, one boundary   component of  $F'_k$ is partitioned into segments labeled by  elements of $\mathcal{A}_k$ and  boundary components of $F_k$; the other  is labeled solely  by the curve in $\mathcal{A}_k$. The second condition in Definition \ref{def:standard} implies that there is an orientation-preserving homeomorphism from $F_1'$ to $F_2'$ which preserves this  labeling (with respect to the natural bijections between  elements of $\mathcal{A}_1$ and $\mathcal{A}_2$ and between components of $\partial F_1$ and $\partial F_2$). The lemma follows.
\end{proof}

We now define  two  ``moves" on arc configurations: \emph{addition} is the process of adding  one arc to a configuration while \emph{deletion} is the process of deleting  arcs from a configuration  until there is exactly one arc meeting each boundary component. It is easy to see that one can transform any arc configuration  into a standard  one via a finite sequence of these moves: one first uses deletion to obtain a configuration in which  there is  exactly one arc meeting each boundary component and then alternates additions with deletions to turn this configuration into a standard one, as illustrated in Figure \ref{fig:arcswaps2}. It follows that arbitrary arc configurations on  surfaces  of the same genus and with the same number of boundary components can be made isomorphic  after  finitely many additions and deletions.
Thus, to complete the proof of Theorem \ref{thm:preclosurewelldefined}, it suffices to show that the theorem holds for contact preclosures built from arc configurations related by deletion (for an arc configuration with exactly one arc meeting each boundary component, an addition is the inverse of a deletion).

\begin{figure}[ht]
\labellist
\hair 2pt
\tiny
\pinlabel $1$ at 28 -9
\pinlabel $2$ at 28 140
\pinlabel $3$ at 111 140
\pinlabel $4$ at 111 -9
\endlabellist
\centering
\includegraphics[width=15cm]{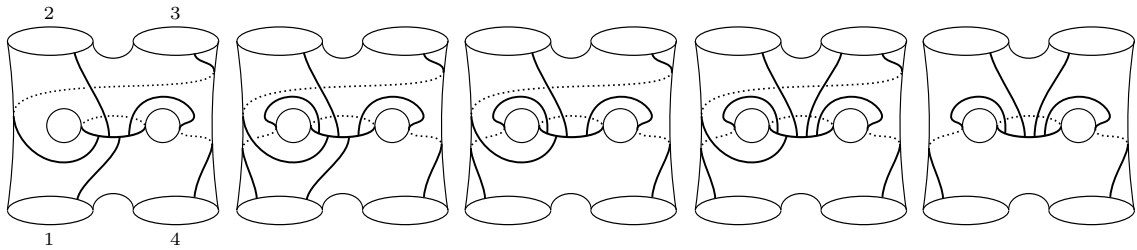}
\caption{Making an arc configuration standard through an alternating sequence of additions and deletions. The four boundary components are labeled as shown on the left.}
\label{fig:arcswaps2}
\end{figure}

Fix an auxiliary surface $F$,  the neighborhoods $A(\Gamma)$ and $B(\Gamma)$, and the diffeomorphism $h$. Let $\mathcal{A}_1$ be an arbitrary arc configuration on $F$ and suppose $\mathcal{A}_2$ is obtained from $\mathcal{A}_1$ by deletion. The dividing set $\Gamma_{\mathcal{A}_2}$  is normally defined in terms of the boundary of a regular neighborhood of $\mathcal{A}_2$, as in \eqref{eqn:divset}. Below, we will instead imagine $\Gamma_{\mathcal{A}_2}$ as coming from the boundary of a regular neighborhood of a certain \emph{graph} $\mathcal{A}_1'
$ on $F$,  \[\Gamma_{\mathcal{A}_2}=-\overline{(\partial N(\mathcal{A}_1')\ssm\partial M)}.\] This graph $\mathcal{A}'_1$ is obtained  by retracting arcs of $\mathcal{A}_1$ a short distance into $F$ until there is exactly one arc meeting each boundary component, as  shown in Figure \ref{fig:retraction}. We retract precisely those arcs that are deleted when forming $\mathcal{A}_2$ from $\mathcal{A}_1$. Although $\mathcal{A}_1'$ is not  an arc configuration in general, a  neighborhood of this graph retracts onto a  neighborhood of $\mathcal{A}_2$, so these two ways of defining $\Gamma_{\mathcal{A}_2}$ result in isotopic dividing sets.

Below, we will use the notation $\xi_{h,k}$ in place of   $\xi_h$ to denote the contact structure on $M$ defined using the map $h$ and the arc configuration $\mathcal{A}_k$. Our goal is to  prove Theorem \ref{thm:preclosurewelldefined} for the contact preclosures 
\begin{align}
\label{eqn:preclosures1}
(M_1',\xi_1') &= (M\cup_h F\times[-1,1], \xi_{h,1}\cup \Xi_{\mathcal{A}_1})\\
\label{eqn:preclosures2}(M_2',\xi_2') &= (M\cup_h F\times[-1,1], \xi_{h,2}\cup \Xi_{\mathcal{A}_2}).
\end{align}  We start by specifying the  data needed to define the contact structures $\xi_{h,j}$. As usual, we will assume that the dividing set of $\Xi_{\mathcal{A}_1}$ on the annulus $\partial_iF\times[-1,1]$ consists of $2k_i$  cocores. Note that the dividing set of $\Xi_{\mathcal{A}_2}$ on each $\partial_iF\times[-1,1]$ consists of just $2$  cocores.

 \begin{figure}[ht]
\labellist
\hair 2pt

\endlabellist
\centering
\includegraphics[width=7.5cm]{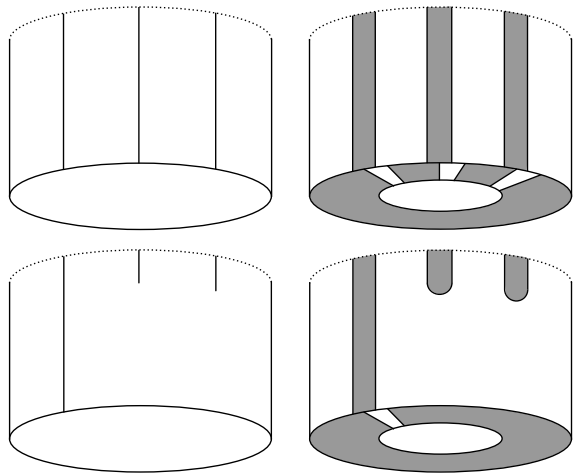}
\caption{Top: left, the arc configuration $\mathcal{A}_1$ near a component $\partial_i F$; right, the corresponding portion of $(F\times[-1,1],\Xi_{\mathcal{A}_1})$. Bottom: left, the graph $\mathcal{A}'_1$ obtained from $\mathcal{A}_1$ by retracting all but one of the arcs meeting each boundary component of $F$; right, the corresponding portion of $(F\times[-1,1],\Xi_{\mathcal{A}_2})$.}
\label{fig:retraction}
\end{figure}

 For $i=1,\dots,m$, let $D_i\subset \inr(A_i)$ be an annulus which intersects $\Gamma_i$ in $2$ cocores, and let $D(\Gamma)=\bigcup_i D_i$. Let $C_{i,1}\subset \inr(D_i)$ be an annulus which intersects $\Gamma_i$ in $2k_i$ cocores. In particular, we require that  one component of $D_i\cap \Gamma_i$ intersects $C_{i,1}$ in  $1$ cocore while the other  intersects $C_{i,1}$ in $2k_i-1$ cocores. Finally, let $C_{i,2}\subset \inr(D_i)$ be an annulus which intersects $\Gamma_i$ in $2$ cocores. See Figure \ref{fig:D} for an illustration of these annuli after  straightening below. The annuli $C_{i,1}$ and $C_{i,2}$ will be used to define the contact structures $\xi_{h,1}$ and $\xi_{h,2}$ in the usual way while the $D_i$ are auxiliary annuli that will be used to relate $(M_1',\xi_1')$ and $(M_2',\xi_2')$. 
 
 \begin{figure}[ht]
\labellist
\hair 2pt\tiny
\pinlabel $\phi_2(D_i)$ at 219 122
\pinlabel $\phi_2(C_{i,2})$ at 218 -6
\pinlabel $\phi_1(D_i)$ at 68 122
\pinlabel $\phi_1(C_{i,1})$ at 67 -6
\pinlabel $\phi_1(\Gamma_i)$ at 15 68
\pinlabel $\phi_2(\Gamma_i)$ at 166 68
\endlabellist
\centering
\includegraphics[width=8.5cm]{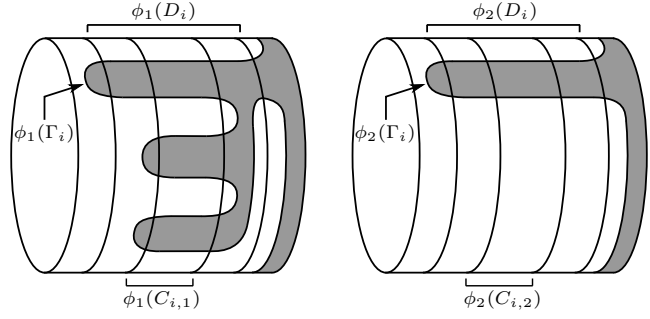}
\caption{A view from within $B_i$ of the annuli $D_i$, $C_{i,1}$,  and $C_{i,2}$ after straightening by $\varphi_1$ on the left and $\varphi_2$ on the right. In particular, $\varphi_{1}(C_{i,1})=\varphi_{2}(C_{i,2})=A_i$ and $\varphi_1(D_i)=\varphi_2(D_i)$.}
\label{fig:D}
\end{figure}

Let $\mathfrak{F}$ be a foliation of $\partial M$ divided by $\Gamma$ which contains the boundary components $d_i^{\pm}$ of $D_i$ as unions of leaves and  agrees with $(\partial M)_{\xi}$ outside of $B(\Gamma)$. Let $\xi'$ be the contact structure with $(\partial M)_{\xi'} = \mathfrak{F}$ obtained from $\xi$ by flexibility. Note that the curves $d_i^{\pm}$ are  Legendrian with respect to $\xi'$. Choose a vertically invariant collar $\partial M\times(-\infty,0]$ of $\partial M$ with respect to $\xi'$. We can arrange (by choosing $\mathfrak{F}$ more carefully) that $\xi'$ is invariant in the $[-\epsilon,\epsilon]$-direction for some small tubular neighborhoods $d_i^{\pm}\times [-\epsilon,\epsilon]\subset \partial M$ of the $d_i^{\pm} = d_i^{\pm}\times\{0\}$.  This implies that the annuli (with corners) given by \[E_i = (d_i^+\times [-4,0]_t)\,\cup \,(D_i\times\{-4\}_t)\,\cup (d_i^-\times[-4,0]_t)\,\subset \,\partial M \times (-\infty,0]_t\] are convex. 

We now construct the isotopies of $\partial M$ which ``straighten out" the annuli $C_{i,1}$ and $C_{i,2}$. 
For $k=1,2$, let \[\varphi_{r,k}:\partial M\rightarrow \partial M,\,r\in[0,1],\] be an isotopy supported in $B(\Gamma)$ such that $\varphi_{0,k}={id}$, $\varphi_{1,k}(C_{i,k}) = A_i$,  and $h$ identifies the positive region on $\partial_i F\times[-1,1]$ with respect to $\Xi_{\mathcal{A}_k}$ with the negative region on $A_i$ determined by $\varphi_{1,k}(\Gamma)$. We  will additionally require that $\varphi_{r,1}=\varphi_{r,2}$ in a neighborhood of the  curves $d_i^{\pm}$ and outside $D(\Gamma)$. Let  \[\varphi_k:M\to M\] be the diffeomorphism of $M$ extending   $\varphi_{1,k}$  defined as in \eqref{eqn:phiextension}), and define $\xi_{0,k} := (\varphi_k)_*(\xi')$. Note that $\xi_{0,1}=\xi_{0,2}$  outside of \begin{equation}\label{eqn:nbhdgamma}\varphi_1(D(\Gamma)\times[-3,0]) = \varphi_2(D(\Gamma)\times[-3,0])\end{equation} and in  neighborhoods of the  annuli $G_i=\varphi_1(E_i)=\varphi_2(E_i).$

For $k=1,2$, let $\mathfrak{F}_{1,k}$ be a foliation of $\partial M$ divided by $\varphi_k(\Gamma)$ which agrees on $A(\Gamma)$ with the image of the characteristic foliation  $(\partial F\times[-1,1])_{\Xi_{\mathcal{A}_k}}$ under $h$ and   with $\mathfrak{F}_{0,k}:=(\partial M)_{\xi_{0,k}}$ outside of  $B(\Gamma)$. We will additionally require that $\mathfrak{F}_{1,1}$ and $\mathfrak{F}_{1,2}$ agree outside of $\varphi_1(D(\Gamma))=\varphi_2(D(\Gamma)).$ Let $\xi_{1,k}$ be the contact structure with  $(\partial M)_{\xi_{1,k}} = \mathfrak{F}_{1,k}$ obtained from $\xi_{0,k}$ by flexibility, using  the collar determined by the contact vector field $(\varphi_k)_*(\partial_t)$ for $\xi_{0,k}$. As usual, we  let $\xi_{h,k}:=\xi_{1,k}$. We can arrange that the annuli $G_i$ are convex for both $\xi_{h,1}$ and $\xi_{h,2}$ and, moreover,  that $\xi_{h,1}$ and $\xi_{h,2}$ agree in neighborhoods of these annuli and outside  the neighborhood of $\Gamma$ in \eqref{eqn:nbhdgamma}. Let $N(\Gamma)$ denote the component of $M\ssm \cup_i G_i$ containing $\Gamma$. Then, in particular, $\xi_{h,1}=\xi_{h,2}$ on $M\ssm N(\Gamma).$ We will record this fact as \begin{equation}\label{eqn:part1}(M\ssm N(\Gamma), \xi_{h,1})=(M\ssm N(\Gamma), \xi_{h,2})\end{equation} for later use.

We now prove Theorem \ref{thm:preclosurewelldefined}  for the contact preclosures $(M_1',\xi_1')$  and $(M_2',\xi_2')$  formed from $\xi_{h,1}$ and $\xi_{h,2}$ as in \eqref{eqn:preclosures1} and \eqref{eqn:preclosures2}.  Since the convex surfaces \[F\times\{t\} \subset (F\times[-1,1], \Xi_{\mathcal{A}_1})\] have collared Legendrian boundary,  there is a  collar $\partial F\times[0,1]\subset F$ such that  $\Xi_{\mathcal{A}_1}$ is invariant in the $[0,1]$-direction on \[\partial F\times[0,1]\times[-1,1]\subset F\times[-1,1].\]  We can arrange that $\Xi_{\mathcal{A}_2}$ is invariant in the same direction on  the smaller  neighborhood  \[\partial F\times[0,1/2]\times[-1,1]\subset F\times[-1,1],\] and that   \begin{equation}\label{eqn:part2}((F\ssm (\partial F\times[0,1]))\times[-1,1],\Xi_{\mathcal{A}_1})=((F\ssm (\partial F\times[0,1]))\times[-1,1],\Xi_{\mathcal{A}_2}).\end{equation} Note that the annuli \[H_i=\partial_iF\times\{1\}\times[-1,1]\,\,\,\,\,\,{\rm and}\,\,\,\,\,\,H_i'=\partial_iF\times\{1/2\}\times[-1,1]\] are convex with collared Legendrian boundary with respect to both $\Xi_{\mathcal{A}_1}$ and $\Xi_{\mathcal{A}_2}$. See Figure \ref{fig:productnbhd} for an illustration of these  annuli.

 \begin{figure}[ht]
\labellist
\hair 2pt

\pinlabel $0$ at -5 20
\pinlabel $\frac{1}{2}$ at -5 67
\pinlabel $1$ at -5 114

\pinlabel $H_i'$ at 193 91
\pinlabel $H_i$ at 193 136
\endlabellist
\centering
\includegraphics[width=9.2cm]{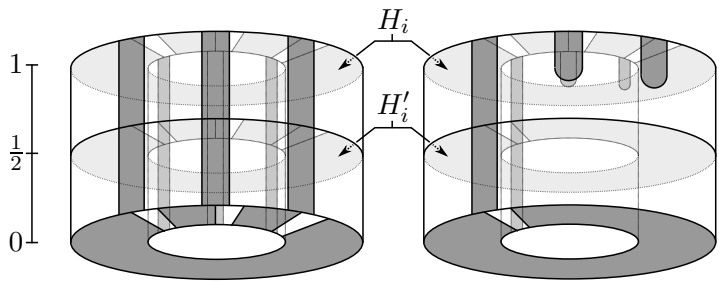}
\caption{Left and right, the regions $\partial F\times[0,1]\times[-1,1]$ in $\Xi_{\mathcal{A}_1}$ and $\Xi_{\mathcal{A}_2}.$}
\label{fig:productnbhd}
\end{figure}

For  $k=1,2$, the  annuli $G_i$ and $H_i$, together with two  annuli in $\partial M_k'$, bound a solid torus $T_{i,k}$ in $M_k'$, as shown in Figure \ref{fig:alternateview}. Moreover, the complement $(M_k'\ssm \cup_i T_{i,k},\xi_k')$ is the union of the pieces in \eqref{eqn:part1} and \eqref{eqn:part2}, which implies that \[(M_1'\ssm \cup_i T_{i,1},\xi_1')=(M_2'\ssm \cup_i T_{i,2},\xi_2').\]   To prove Theorem \ref{thm:preclosurewelldefined}, it therefore suffices to show that, after rounding corners,   $(T_{i,1},\xi_1')$ is contactomorphic to $(T_{i,2},\xi_2')$, up to flexibility. But after rounding corners, $T_{i,1}$ and $T_{i,2}$ are  solid tori with convex boundaries with dividing sets consisting of two parallel curves of slope $-1$, as shown in Figure \ref{fig:Tij} for $T_{i,2}$. As there is a unique tight solid torus with these boundary conditions, up to flexibility and contactomorphism, all that remains is to show that  $(T_{i,1},\xi_1')$ and $(T_{i,2},\xi_2')$ are tight.

 \begin{figure}[ht]
 
\labellist
\hair 2pt
\pinlabel $0$ at -8 403
\pinlabel $-\infty$ at -20 353
\pinlabel $H_i$ at 498 391
\pinlabel $G_i$ at 462 230
\pinlabel $T_{i,1}$ at 460 330
\pinlabel $T_{i,2}$ at 460 77
\endlabellist
\centering
\includegraphics[width=9cm]{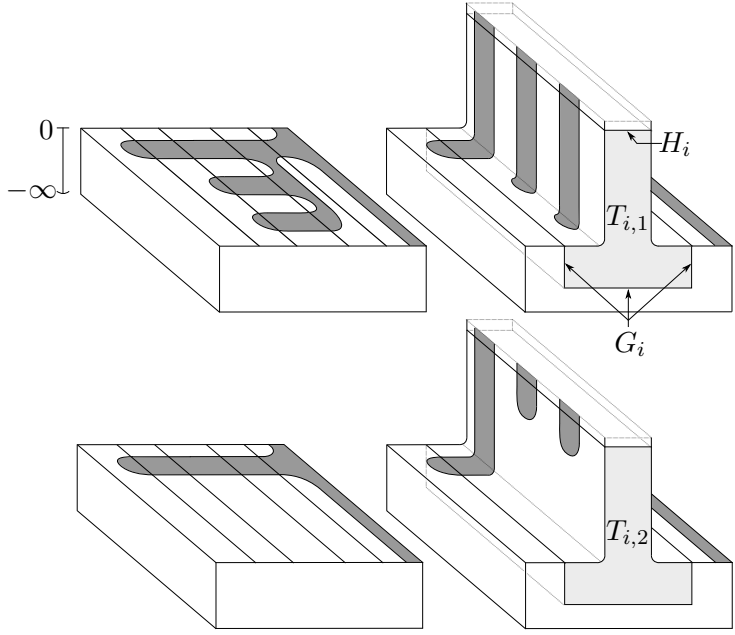}
\caption{Top: left,  $(M,\xi_{h,1})$ on a portion of $B_i\times(-\infty,0]$; right, the corresponding portion of  $(M',\xi_1')$ gotten by attaching $(F\times[-1,1],\Xi_{\mathcal{A}_1})$. Bottom, the analogous pictures for $(M,\xi_{h,2})$ and $(M',\xi_2')$. The lightly shaded regions on the top and bottom are meridional disks for the tori $T_{i,1}$ and $T_{i,2}$.}
\label{fig:alternateview}
\end{figure}

 \begin{figure}

\centering
\includegraphics[width=7.9cm]{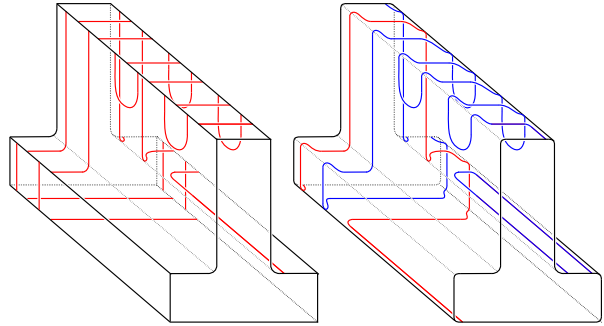}
\caption{Left, the dividing set on $\partial T_{i,2}$. Right, the dividing set after rounding corners; it consists of two curves of slope $-1$ drawn in red and blue.}
\label{fig:Tij}
\end{figure}

\begin{lemma}
For  $k=1,2$, the solid torus $(T_{i,k},\xi_k')$ is tight.
\end{lemma}

\begin{proof}
For $k=1$, this follows from the fact that  $(T_{i,1},\xi_1')$ can be embedded as a contact submanifold of some vertically invariant  neighborhood of $B_i$ in which the characteristic foliation on each copy of $B_i$ agrees with that of $\xi_{h,1}$. To see that such a neighborhood is tight, note that it is related by flexibility to a vertically invariant neighborhood of $B_i$ in which the characteristic foliation on each copy of $B_i$ consists of cocores. The latter is, in some sense, a standard neighborhood of a dividing curve: it embeds as a neighborhood of any dividing curve on any convex surface in any contact manifold, and is therefore tight. The former neighborhood of $B_i$ is thus tight as well since flexibility preserves tightness. To see that $(T_{i,1},\xi_1')$ embeds into such a neighborhood, note that $T_{i,1}$ is contained in a union  \begin{equation}\label{eqn:vinvariant}(B_i\times(-\infty,0]_s\cup_h\partial F\times[0,1]_s\times[-1,1],\xi_{h,1}\cup \Xi_{\mathcal{A}_1}),\end{equation} where  $\xi_{h,1}$ on $B_i\times(-\infty,0]_s$ is invariant in the $\partial_s$-direction. Here, $\partial_s = (\varphi_1)_*(\partial_t),$ where $\partial_t$ is the contact vector field for the vertically invariant collar $\partial M\times(-\infty,0]$ for $\xi'$ discussed earlier. Since $\Xi_{\mathcal{A}_1}$ is invariant in the $[0,1]_s$-direction, the union in \eqref{eqn:vinvariant} can be embedded in a vertically invariant neighborhood $B_i\times(-\infty,1]_s$.

For $k=2$, let $T_{i,2}'$ be the solid torus in $M_2'$ bounded by  $G_i$, $H_i'$,   and two annuli in $\partial M'_2$. By the reasoning above (considering $\Xi_{\mathcal{A}_2}$ rather than $\Xi_{\mathcal{A}_1}$), the solid torus $(T_{i,2}',\xi_2')$ embeds into a vertically invariant neighborhood of $B_i$  and is therefore tight. Note that $(T_{i,2},\xi_2')$ is obtained by gluing  $(\partial_iF\times[1/2,1]\times[-1,1],\Xi_{\mathcal{A}_2})$ to  $(T_{i,2}',\xi_2')$ along $H_i'$. After rounding corners, the first piece  is contactomorphic to the $[1/2,1]$-invariant contact structure specified by the dividing set of $\Xi_{\mathcal{A}_2}$ on $H_i'$. Indeed,   both  are tight solid tori with convex boundaries and dividing sets consisting of two parallel curves of slope $-1$. It follows that $(T_{i,2},\xi_2')$ is contactomorphic to the torus obtained from $(T_{i,2}',\xi_2')$ by attaching a vertically invariant neighborhood of some portion of $\partial T_{i,2}'$. Thus, $(T_{i,2},\xi_2')$ is contactomorphic to  $(T_{i,2}',\xi_2')$, and, hence, tight.
\end{proof}

This concludes the proof of Theorem \ref{thm:preclosurewelldefined}.
\end{proof}

The corollary below follows easily from the proof of Theorem \ref{thm:preclosurewelldefined}.

\begin{corollary}
\label{cor:flexibilitypreclosure}
Suppose $(M,\Gamma,\xi_1)$ and $(M,\Gamma,\xi_2)$ are  related by flexibility, with contact preclosures $(M'_1,\xi'_1)$ and $(M'_2,\xi'_2)$ defined using auxiliary surfaces of the same genus. Then, up to flexibility,  $(M_1',\xi_1')$ and $(M_2',\xi_2')$  are contactomorphic by a map smoothly isotopic to one that restricts to the identity on $M\ssm N(\Gamma)$ for some regular neighborhood $N(\Gamma)$ of $\Gamma$. \qed
\end{corollary}


\subsection{Contact closures and the invariants $\invt^g(M, \Gamma,\xi)$ and $\invt(M, \Gamma,\xi)$}
\label{ssec:contact-closures-invariant}

Suppose  $(M',\xi')$ is a contact preclosure of  $(M,\Gamma,\xi)$. As mentioned in Remark \ref{rmk:divset}, the dividing set for $\xi'$ consists of two parallel nonseparating curves on each  component $\partial_{\pm} M'$ of $\partial M'$. One can therefore glue $\partial_+ M'$ to $\partial_- M'$  (after applying flexibility, of course) by a map which identifies the positive region on $\partial_+ M'$  with the negative region on $\partial_- M'$ to form a closed contact manifold $(Y,\bar \xi)$ with a distinguished convex surface $R:=\partial_+ M'=-\partial_- M'$. We   call a  triple $(Y, R, \bar \xi)$ formed in this way  a \emph{simple contact closure} of $(M,\Gamma,\xi)$. One might  then attempt to define an invariant of $\xi$ in terms of the  contact invariant $\psi(Y,\bar\xi)$. We do essentially this but, for naturality purposes,   need the following slightly more involved  notion of contact closure.

\begin{definition}
\label{def:contactclosure}
A \emph{contact closure} of $(M,\Gamma,\xi)$ consists of a closure $\mathscr{D} = (Y,R,r,m)$ of $(M,\Gamma)$ together with a contact structure $\bar \xi$ on $Y$ such that 
\begin{enumerate}
\item $m$ restricts to a contact embedding of $(M\ssm N(\Gamma),\xi)$ into $(Y,\bar \xi)$ for some regular neighborhood $N(\Gamma)$ of $\Gamma$,
\item this restriction  of $m$ extends to a contactomorphism \[(M',\xi')\to (Y\ssm\inr(\Img(r)),\bar \xi)\] for some contact preclosure $(M',\xi')$ of $(M,\Gamma,\xi)$. 
\item $r^*(\bar\xi)$ is a contact structure on $R\times[-1,1]$ obtained, via flexibility, from one that is invariant in the $[-1,1]$-direction.
\end{enumerate}
\end{definition}

\begin{remark}
Note  that $(M,\Gamma,\xi)$ admits a genus $g$ contact closure for every $g\geq g(M,\Gamma)$.
\end{remark}

\begin{remark}
\label{rmk:simple} For a contact closure $(\mathscr{D},\bar\xi)$ as in Definition \ref{def:contactclosure}, the triple $(Y,r(R\times\{t\}),\bar\xi)$ is a simple contact closure of $(M,\Gamma,\xi)$ as described at the top for any $t\in [0,1]$. In particular, each $r(R\times\{t\})$ is convex with negative region an annulus bounded by   essential curves. 
\end{remark}

\begin{definition}
A \emph{marked} contact closure of $(M,\Gamma,\xi)$ is a marked closure $\data = (Y,R,r,m,\eta)$ of $(M,\Gamma)$ together with a contact structure $\bar\xi$ on $Y$ such that $((Y,R,r,m),\bar\xi)$ is a contact closure as in Definition \ref{def:contactclosure} and $r(\eta\times\{0\})$ is dual to the core of the negative annular region of $r(R\times\{0\})$. 
\end{definition}

Suppose $(\data=(Y,R,r,m,\eta),\bar \xi)$ is a marked contact closure of $(M,\Gamma,\xi)$ of genus $g$. Let $r(R\times\{0\})_{\pm}$ denote the positive and negative regions of $r(R\times\{0\})$. It is a standard result in convex surface theory that 
\[\langle c_1(\spc_{\bar\xi}), r(R\times\{0\}) \rangle_Y = \chi(r(R\times\{0\})_+)-\chi(r(R\times\{0\})_-),\] which is equal to $2-2g$  in this case since $r(R\times\{0\})_-$ is an annulus. It follows that 
\[\langle c_1(\spc_{\bar\xi}), r(-R\times\{0\}) \rangle_{-Y} = 2g-2,\] which implies that \[\HMtoc(-Y,\spc_{\bar\xi};\Gamma_{-\eta})\subset \HMtoc(-Y|{-}R;\Gamma_{-\eta})=\SHMt(-\data),\] where $-\data$ is the corresponding marked closure of $(-M,-\Gamma)$. In particular, \[\psi(Y,\bar\xi)\in\SHMt(-\data).\] This leads to the following definition.

\begin{definition}
\label{def:contactinvariant} Given a  marked contact closure $(\data = (Y,R,r,m,\eta),\bar\xi)$  of $(M,\Gamma,\xi)$ of genus $g\geq g(M,\Gamma)$, we define $ \invt^g(M,\Gamma,\xi)$ to be the element of $\SHMtfun(-M,-\Gamma)$ determined by the equivalence class of  \[\invt(\data,\bar\xi):=\psi(Y,\bar\xi)\in\SHMt(-\data),\] in the sense of Remark \ref{rmk:completesubset2}. 
\end{definition}

In Section \ref{sec:well-definedness}, we will prove that $\invt^g(M,\Gamma,\xi)$ is well-defined for each $g$, per the following theorem.

\begin{theorem}
\label{thm:well-defined} If $(\data,\bar\xi)$ and $(\data',\bar\xi')$ are two  marked contact closures of $(M,\Gamma,\xi)$ of the same genus, then \[\Psit_{-\data,-\data'}(\invt(\data,\bar\xi)) \doteq \invt(\data',\bar\xi').\] 
\end{theorem}

Furthermore, we will show that for $g$ sufficiently large, the  contact elements $\invt^g(M,\Gamma,\xi)$ are equal, per the following theorem.

\begin{theorem}
\label{thm:well-defined2} For every $(M,\Gamma,\xi)$, there is an integer $N(M,\Gamma,\xi)$ such that if $(\data,\bar\xi)$ and $(\data',\bar\xi')$ are marked contact closures of $(M,\Gamma,\xi)$ of genus at least $N(M,\Gamma,\xi)$, then \[\Psit_{-\data,-\data'}(\invt(\data,\bar\xi)) \doteq \invt(\data',\bar\xi').\] 
\end{theorem}

This theorem motivates the following definition.

\begin{definition}
\label{def:contactinvariantuniv} We define \[\invt(M,\Gamma,\xi):=\invt^g(M,\Gamma,\xi)\in\SHMtfun(-M,-\Gamma)\] for any $g\geq N(M,\Gamma,\xi)$.

\end{definition}

We will prove Theorems \ref{thm:well-defined} and \ref{thm:well-defined2} in Section \ref{sec:well-definedness}.

\subsection{Properties}  Below, we assume  Theorems \ref{thm:well-defined} and \ref{thm:well-defined2}  hold in order to state and prove some basic properties about our contact invariants. We state these results for the invariants $\invt^g$. By Theorem \ref{thm:well-defined2}, they also hold for the invariant $\invt$.

\begin{lemma}
\label{lem:contactomorphism} Suppose  $f$ is a contactomorphism from $(M,\Gamma,\xi)$ to $(M',\Gamma',\xi')$. Then the induced map \[\SHMtfun(f):\SHMtfun(-M,-\Gamma)\to\SHMtfun(-M',-\Gamma')\] sends $\invt^g(M,\Gamma,\xi)$ to $\invt^g(M',\Gamma',\xi')$.
\end{lemma}

\begin{proof}
Essentially, a contactomorphism gives rise  to contactomorphic closures. More precisely, suppose $(\data',\bar\xi')$ is a marked contact closure of $(M',\Gamma',\xi')$. Since $f$ is a contactomorphism,  $(\data'_f,\bar\xi')$, as defined in \eqref{eqn:dataf}, is a marked contact closure of $(M,\Gamma,\xi)$. According to the definition of $\SHMtfun(f)$ in Subsection \ref{ssec:shm}, it suffices to show that the identity map \[id_{-\data'_f,-\data'}:\SHMt(-\data'_f)\to\SHMt(-\data')\] sends $\invt(\data'_f,\bar\xi')$ to $\invt(\data',\bar\xi')$. But this is immediate since \[\invt(\data'_f,\bar\xi') = \psi(Y',\bar \xi') = \invt(\data',\bar\xi'). \qedhere\] 
\end{proof}

Since the map $\SHMtfun(f)$  only depends on the isotopy class of $f$, we have the following.

\begin{corollary}
\label{cor:isotopyindependence}
Suppose $(M,\Gamma,\xi)$ and $(M,\Gamma,\xi')$ are sutured contact manifolds such that   $\xi$ and $\xi'$ are isotopic through diffeomorphisms fixing $\Gamma$. Then $\invt^g(M,\Gamma,\xi)=\invt^g(M,\Gamma,\xi')$. \qed
\end{corollary}

The following corollary should be thought of as saying that  $\invt^g(M,\Gamma,\xi)$ is essentially independent of the particular choice of multicurve $\Gamma$ dividing $(\partial M)_\xi$.

\begin{corollary}
\label{cor:independenceofdividingset}
Suppose $(M,\Gamma,\xi)$ and $(M,\Gamma',\xi)$ are sutured contact manifolds with the same underlying contact manifold but different dividing sets. Then there is a canonical isomorphism  \[\Psit_{\xi,\Gamma,\Gamma'}:\SHMtfun(-M,-\Gamma)\to\SHMtfun(-M,-\Gamma')\] sending $\invt^g(M,\Gamma,\xi)$ to $\invt^g(M,\Gamma',\xi)$.
\end{corollary}

\begin{proof}
Since the  set of multicurves dividing  $(\partial M)_\xi$ is connected, there is an isotopy \[\varphi_r:\partial M\rightarrow \partial M,\,r\in[0,1],\] such that $\varphi_0={ id}$, each $\varphi_r$ preserves  $(\partial M)_{\xi}$, and $\varphi_1(\Gamma_i)=\Gamma_i'$. Suppose $\partial M\times(-\infty,0]$ is a vertically invariant collar of $\partial M$, and extend $\varphi_r$ to a diffeomorphism \[\varphi:M\to M\] as in \eqref{eqn:phiextension}. It is easy to see, using Lemma \ref{lem:uniqueness}, that $\xi$ and $\varphi_*(\xi)$ are isotopic by an isotopy stationary on $\partial M$. It then follows from Lemma \ref{lem:contactomorphism} and Corollary \ref{cor:isotopyindependence} that \[\Psit_{\xi,\Gamma,\Gamma'}:=\SHMtfun(\varphi):\SHMtfun(-M,-\Gamma)\to\SHMtfun(-M,-\Gamma')\] sends $\invt^g(M,\Gamma,\xi)$ to $\invt^g(M,\Gamma',\xi)$.
That this isomorphism is ``canonical" amounts to showing that it does not depend on the  choices of $\varphi$ or the collar (that it is well-defined), and that \begin{equation}\label{eqn:transitivitypsi}\Psit_{\xi,\Gamma,\Gamma''} = \Psit_{\xi,\Gamma',\Gamma''}\circ \Psit_{\xi,\Gamma,\Gamma'}\end{equation} for any three multicurves $\Gamma,\Gamma',\Gamma''$ dividing $(\partial M)_\xi$. The connectedness of the space of such collars implies that  $\varphi$ is independent, up to isotopy stationary on $\partial M$, of the collar. Thus,  $\SHMtfun(\varphi)$ is independent of the collar. With that established, let us fix some collar, and suppose $\varphi^1$ and $\varphi^2$ are  diffeomorphisms of $M$ as defined above. The contractibility of the space of multicurves dividing $(\partial M)_\xi$ implies that $\varphi^1$ and $\varphi^2$ are  isotopic. It follows that \[\SHMtfun(\varphi^1)=\SHMtfun(\varphi^2).\] Thus, $\Psit_{\xi,\Gamma,\Gamma'} $ is well-defined. Now, suppose $\varphi$ and $\varphi'$ are diffeomorphisms of $M$ of the sort used to define the maps $\Psit_{\xi,\Gamma,\Gamma'}$ and $\Psit_{\xi,\Gamma',\Gamma''}$. The transitivity in \eqref{eqn:transitivitypsi} follows immediately from the fact that  $\varphi'':=\varphi'\circ\varphi$ is a diffeomorphism of the sort used to define $\Psit_{\xi,\Gamma,\Gamma''}$. \end{proof}

The  corollary below indicates the invariance of $\invt^g$ with respect to flexibility. 

\begin{corollary}
\label{cor:contactinvariantflexibility} Suppose  $(M,\Gamma,\xi)$ and $(M,\Gamma,\xi')$ are related by flexibility. Then $\invt^g(M,\Gamma,\xi)=\invt^g(M,\Gamma,\xi')$.
\end{corollary}

\begin{proof}
Suppose $(\data=(Y,R,r,m,\eta),\bar\xi)$ is a marked contact closure of $(M,\Gamma,\xi)$.   Corollary \ref{cor:flexibilitypreclosure}, together with Theorem \ref{thm:flexibility},  implies that $\bar\xi$ is isotopic to a contact structure $\bar\xi'$ 
for which $(\data,\bar\xi')$ is a marked contact closure of $(M,\Gamma,\xi')$. Then $\invt^g(M,\Gamma,\xi)=\invt^g(M,\Gamma,\xi')$ since $\invt(\data,\bar\xi)=\invt(\data,\bar\xi')$.
\end{proof}

\begin{remark} The results above allow us to largely  ignore, when dealing with the invariants $\invt^g$ and $\invt$, the differences between contact structures related by flexibility or isotopy. Accordingly, we will frequently work on the level of dividing sets rather than characteristic foliations and will often think of dividing sets as \emph{isotopy classes} of multicurves.
\end{remark}

Note if $(M,\Gamma,\xi)$ is overtwisted, then so is any contact closure $(\data,\bar\xi)$ of $(M,\Gamma,\xi)$. This implies that  $\invt(\data,\bar\xi)=0$, by Theorem \ref{thm:hm-ot}. The theorem below follows immediately. 

\begin{theorem}
\label{thm:ot}
If $(M,\Gamma,\xi)$ is overtwisted, then $\invt^g(M,\Gamma,\xi)=0$.\qed
\end{theorem}


Given a closed contact 3-manifold $(Y,\xi)$, let  $Y(p)$ denote the sutured contact manifold obtained from $Y$ by removing  a Darboux ball centered at $p$. It is not hard  to show that there is a canonical isotopy class of contactomorphisms relating any two such manifolds for a given point $p$, justifying our notation. When it is not important to keep track of $p$, we will write $Y(1)$ instead (as  in the introduction), indicating that we have removed one Darboux ball. More generally, $Y(n)$ will refer to the (contactomorphism type of the) sutured contact manifold obtained  by removing $n$ disjoint Darboux balls. 

We will prove the following in Section \ref{sec:well-definedness}.

\begin{proposition}
\label{prop:darboux-complement}
There is a morphism 
\begin{equation*}
\label{eq:stein-one-handle}
F_p:\SHMtfun(-Y(p)) \to \HMtoc(-Y) \otimes_\mathbb{Z}\RR
\end{equation*}
which sends $\invt^g(Y(p))$ to $\psi(Y,\xi) \otimes \mathbf{1}$, where $\mathbf{1}$ refers to the equivalence class of $1\in\RR$. \end{proposition}

Since the monopole Floer invariant $\psi(Y,\xi)$ is nonzero for  strongly symplectically fillable contact structures (see \cite{ht}), we have the following immediate corollary.

\begin{corollary}
\label{cor:nonzerostronglyfillable}
If $(Y,\xi)$ is strongly symplectically fillable, then $\invt^g(Y(p))\neq 0$. \qed

\end{corollary}

\begin{remark}
\label{rmk:analoguehatplus}
The morphism in Proposition \ref{prop:darboux-complement} can be thought of an  analogue of the  natural map in Heegaard Floer homology, \[\hf(-Y) \to \hfp(-Y),\]  which sends ${c}(Y,\xi)$ to $c^+(Y,\xi)$. Indeed, the modules comprising the systems  $\SHMtfun(-Y(p))$ and $\HMtoc(-Y) \otimes_\mathbb{Z} \RR$ are isomorphic to $\hf(-Y)\otimes_\mathbb{Z}\RR$ and $\hfp(-Y)\otimes_\mathbb{Z} \RR$, respectively. \end{remark}

Suppose $K$ is a Legendrian knot in the interior of $(M,\Gamma,\xi)$ and  that $(M',\Gamma',\xi')$ is the result of contact $(+1)$-surgery on $K$. We will prove the following in Section \ref{sec:well-definedness}.

\begin{proposition}
\label{prop:shm-legendrian-surgery}
There is a morphism
\[ F_K:\SHMtfun(-M,-\Gamma) \to \SHMtfun(-M',-\Gamma') \]
which sends  $\invt^g(M,\Gamma,\xi)$  to $\invt^g(M',\Gamma',\xi')$.
\end{proposition}

\subsection{Examples}
\label{ssec:closureexamples}
Below, we  compute the contact invariants  of  the Darboux ball and  \emph{product sutured contact handlebodies} more generally. As above, we state these results in terms of the invariants $\invt^g$, but they also hold for the invariant $\invt$ by Theorem \ref{thm:well-defined2}.

We start by constructing a genus $g$ marked contact closure of the Darboux ball $(B^3,S^1,\xi_{std})$ for each \[g\geq g(B^3,S^1)=2\] (where the dividing set $S^1$ is a single equatorial curve on $\partial B^3$). Consider the $[-1,1]$-invariant contact structure $\xi_{D^2}$ on $D^2\times[-1,1]$ for which each $D^2\times\{t\}$ is convex with collared Legendrian boundary and the dividing set on $D^2\times\{1\}$ consists of a single properly embedded arc, as shown in Figure \ref{fig:darbouxballcorners}. The product sutured contact manifold $(D^2\times[-1,1],\partial D^2\times\{0\},\xi_{D^2})$ is contactomorphic to $(B^3,S^1,\xi_{std})$ after rounding corners. One advantage of thinking of the Darboux ball in this way is that, in doing so, we have, in effect, automatically perturbed $\xi_{std}$ as required for forming contact preclosures. 

\begin{figure}[ht]
\labellist
\small \hair 2pt
\pinlabel $\partial D^2\times\{0\}$ at -34 19
\pinlabel $S^1$ at 334 82

\endlabellist
\centering
\includegraphics[width=7.3cm]{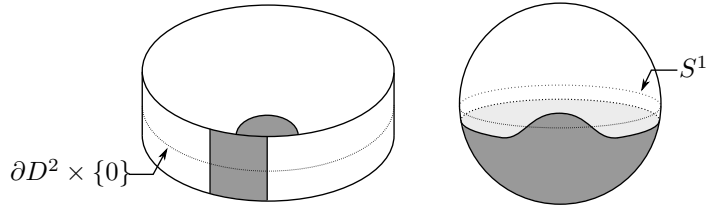}
\caption{Left, $(D^2\times[-1,1],\xi_{D^2})$ with the negative region shaded. Right, the Darboux ball obtained by rounding corners. }
\label{fig:darbouxballcorners}
\end{figure}

Indeed, let $F$ be a genus $g\geq 2$ surface with one boundary component, and let $\mathcal{A} = \{c,a\}$ be an arc configuration on $F$ with a single arc. We may  form a contact preclosure of the Darboux ball by gluing $(F\times[-1,1],\Xi_{\mathcal{A}})$ to $(D^2\times[-1,1],\xi_{D^2})$ according to a map \[h:\partial F\times[-1,1]\to\partial D^2\times[-1,1]\] of the form $f\times{\rm id}$ for some diffeomorphism $f:\partial F\to \partial D^2$, as in Figure \ref{fig:darbouxballclosure}.
The result is a $[-1,1]$-invariant contact structure $\xi'$ on $M'=(D^2\cup F)\times[-1,1]$. Each $(D^2\cup F)\times\{t\}$ is convex with negative region an annular neighborhood $A(c)\times\{t\}$ of the curve $c\times\{t\}$. To form a marked contact closure, we take $R=(D^2\cup F)$ and glue $R\times[-1,1]$, equipped with the $[-1,1]$-invariant contact structure with negative region $A(c)\times\{t\}$ on each $R\times\{t\}$, to $M'$ by the ``identity" maps 
\[
R\times\{\pm 1\}\to (D^2\cup F)\times\{\mp 1\}.
\]
Let $\eta$ be a curve in $R$ dual to the core of $A(c)$. The resulting contact closure is $(\data,\bar\xi)$ with \[\data = ((D^2\cup F)\times S^1, (D^2\cup F),r,m,\eta),\] where $\bar\xi$ is an $S^1$-invariant contact structure for which the negative region on each fiber is a copy of $A(c)$. Here, we are thinking of $S^1$ as the union of two copies of $[-1,1]$, and $r$ and $m$ as the obvious embeddings.

\begin{figure}[ht]
\centering
\includegraphics[width=10.3cm]{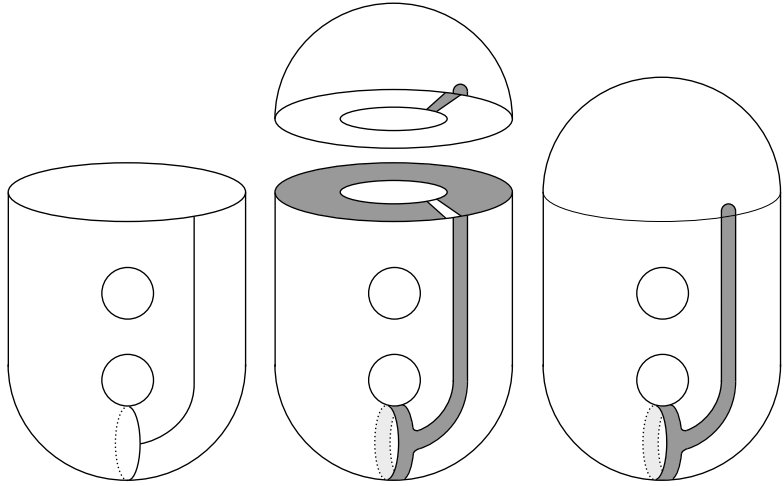}
\caption{Left, the arc configuration $\mathcal{A}$ on $F$ with $g(F)=2$. Middle, gluing $(F\times[-1,1],\Xi_{\mathcal{A}})$ to $(D^2\times[-1,1],\xi_{D^2})$ with the negative region shaded. Right, $(D^2\cup F)$ with the  annulus $A(c)$ shaded. }
\label{fig:darbouxballclosure}
\end{figure}

\begin{proposition}
\label{prop:darboux-ball-invt} The invariant $\invt^g(B^3,S^1,\xi_{std})$ is a unit in $\SHMtfun(-B^3,-S^1)\cong \RR.$
\end{proposition}

\begin{proof}
It follows from work of Niederkr{\"u}ger and Wendl (see \cite[Theorem 5]{wennied}) that the contact manifold $((D^2\cup F)\times S^1,\bar\xi)$ is weakly symplectically fillable by some $(W,\omega)$. According to their construction, we may choose the curve $\eta$ so that $r(\eta\times\{0\})$ is, up to a scalar multiple, the Poincar{\'e} dual of $[\omega|_{(D^2\cup F)\times S^1}]$. Since $\SHMt(-\data)$ is defined with respect to the local system $\Gamma_{-\eta}$, it follows from Theorem \ref{thm:psi-weakly-fillable} that the contact class $\invt(\data,\bar\xi)$ is a primitive element of $\SHMt(-\data)$. The proposition then follows from the fact that $\SHMt(-\data)\cong \RR$, by Proposition \ref{prop:productsutured}.
\end{proof}

\begin{remark}
\label{rem:why-twisted} The proof of Proposition \ref{prop:darboux-ball-invt} highlights the need for working with twisted coefficients: the contact structure $\bar\xi$ above is only weakly fillable and, indeed, Wendl shows in \cite[Corollary 2]{wen3} that its untwisted ECH contact invariant vanishes, from which it follows that $\psi(\data,\bar\xi)\in\SHM(-\data)$ vanishes as well. 
\end{remark}

 Below, we compute the contact invariants of product manifolds built from  general  surfaces. 

Let $S$ be a genus $k$ surface with $l\geq 1$ boundary components. Consider the $[-1,1]$-invariant contact structure $\xi_{S}$ on $S\times[-1,1]$ for which each $S\times\{t\}$ is convex with collared Legendrian boundary and the dividing set on $S\times\{1\}$ consists of $k$ boundary parallel arcs, one for each component of $\partial S$, oriented in the same direction as the boundary. See Figure \ref{fig:productsurfacecorners}. Let $H(S)$ be the \emph{product sutured contact handlebody} of genus $2k+l-1$ obtained from $(S\times[-1,1],\partial S\times\{0\},\xi_{S})$ by rounding corners. 
Note that $H(S)$ is precisely the sort of contact handlebody that appears in the  Heegaard splitting associated to an open book with page $S$. 

\begin{figure}[ht]
\centering
\includegraphics[width=11.7cm]{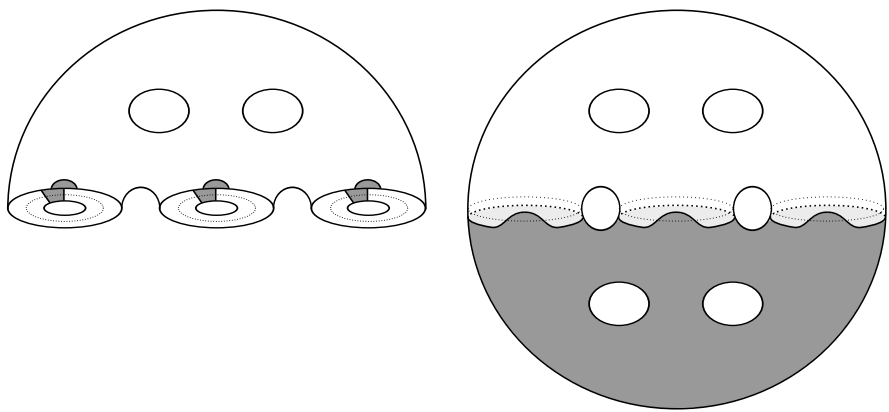}
\caption{Left, $(S\times[-1,1],\xi_{S})$,  with  negative region shaded, for $k=2,l=3$. Right, the convex boundary of the product sutured contact handlebody $H(S)$ obtained by rounding corners.}
\label{fig:productsurfacecorners}
\end{figure}

We have the following generalization of Proposition \ref{prop:darboux-ball-invt}.

\begin{proposition}
\label{prop:handlebody-invt}
The invariant $\invt^g(H(S))$ is a unit in $\SHMtfun(-H(S))\cong \RR$.
\end{proposition}

\begin{proof} This proof is virtually identical to that of Proposition \ref{prop:darboux-ball-invt}. We start by constructing a genus $g$ marked contact closure of $H(S)$ for every \begin{equation}\label{eqn:ghs}g\geq g(H(S))=\max\{2,k+l \} = \max\{2,g(S) + |\partial S|\}.\end{equation} Let $F$ be a surface with $l$ boundary components, and let $\mathcal{A} = \{c,a_1,\dots,a_l\}$ be an arc configuration on $F$ with one arc meeting each boundary component. We form a genus $k+l+g(F)-1$ contact preclosure of $H(S)$ by gluing $(F\times[-1,1],\Xi_{\mathcal{A}})$ to $(S\times[-1,1],\xi_{S})$ and then proceed as in the case of the Darboux ball to construct a marked contact closure of the form $(\data,\bar\xi)$ with \[\data = ((S\cup F)\times S^1, (S\cup F),r,m,\eta),\] where $\bar\xi$ is an $S^1$-invariant contact structure for which the negative region on each fiber is a nonseparating annulus. These are exactly the same $S^1$-invariant contact manifolds as were considered in the proof of Proposition \ref{prop:darboux-ball-invt}. Thus, for an appropriate choice of $\eta$, the contact class $\invt(\data,\bar\xi)$ is a unit in $\SHMt(-\data)\cong \RR$.
\end{proof}

\section{The well-definedness of  $\invt^g(M,\Gamma,\xi)$ and $\invt(M,\Gamma,\xi)$}
\label{sec:well-definedness}
We prove Theorems \ref{thm:well-defined} and \ref{thm:well-defined2} in the next two subsections. On the way to our proof of Theorem \ref{thm:well-defined2}, we define maps on $\SHMtfun$ associated to contact handle attachments and  prove Propositions \ref{prop:shm-legendrian-surgery} and \ref{prop:darboux-complement}.

\subsection{The well-definedness of $\invt^g(M,\Gamma,\xi)$}
\label{ssec:invariant-fixed-g}
We start by describing the isomorphism  $\Psit_{-\data_1,-\data_2}$ for $g(\data_1)= g(\data_2)$, as given in \cite{bs3} but tailored slightly to our setting. We then prove Theorem \ref{thm:well-defined}, which implies that $\invt^g(M,\Gamma,\xi)\in\SHMtfun(-M,-\Gamma)$ is well-defined.

Suppose 
\begin{align*}
(\data_1,\bar\xi_1)&= ((Y_1,R_1,r_1,m_1,\eta_1),\bar\xi_1)\\
(\data_2,\bar\xi_2)&= ((Y_2,R_2,r_2,m_2,\eta_2),\bar\xi_2)
\end{align*}
are two marked contact closures of $(M,\Gamma,\xi)$ of genus $g\geq g(M,\Gamma)$. To define $\Psit_{-\data_1,-\data_2}$, we first choose a contactomorphism \[C:(Y_1\ssm\inr(\Img(r_1)),\bar\xi_1)\to (Y_2\ssm\inr(\Img(r_2)),\bar\xi_2)\] which restricts to $m_2\circ m_1^{-1}$ on $m_1(M\ssm N(\Gamma))$ for some neighborhood $N(\Gamma)$ of $\Gamma$. A contactomorphism of this form exists by Theorem \ref{thm:preclosurewelldefined}. (Technically, Theorem \ref{thm:preclosurewelldefined} says that there is  a contactomorphism of this form after applying flexibility to one of the complements above. However, we will ignore this point, as  we can achieve the same effect by modifying one of the  $\bar\xi_i$ via an arbitrarily small isotopy supported away from $m_i(M)$.) Let $\varphi_{\pm}$ and $\varphi$ be the diffeomorphisms defined by
\begin{align*}
\varphi_{\pm}&= (r_2^{\pm})^{-1}\circ C\circ r_1^{\pm}: R_1\to R_2\\
\varphi&=(\varphi_+)^{-1}\circ \varphi_-:R_1\to R_1,
\end{align*}
where $r_i^\pm$ is the composition
\[R_i\xrightarrow{id\times\{\pm 1\}}R_i\times\{\pm 1\}\xrightarrow{r_i} Y_i.\] 
Remark \ref{rmk:divset} implies that the negative region of $\bar\xi_i$ on each  $r_i(R_i\times\{t\})$ is of the form $r_i(A_i\times\{t\})$ for some annulus $A_i\subset R_i$. Since $C$ is a contactomorphism,  $\varphi_{\pm}$ sends $A_1$ to $A_2$, which implies that $\varphi$ sends $A_1$ to itself. Let 
\[\psi:R_1\to R_1\] be any diffeomorphism such that $\psi$ sends $A_1$ to itself and \[(\varphi_-\circ\psi)(\eta_1) = \eta_2.\]

\begin{remark} 
\label{rmk:cuttinggluing}The diffeomorphisms above  are defined so that the triple $(Y_2,r_2(R_2\times\{0\}),\eta_2)$ is diffeomorphic to that obtained from  $(Y_1,r_1(R_1\times\{0\}),\eta_1)$ by cutting the latter open along the surfaces $r_1(R_1\times\{t\})$ and $r_1(R_1\times\{t'\})$ for some $t<0<t'$ and regluing by the maps $r_1\circ \psi^{-1}\circ r_1^{-1}$ and $r_1\circ(\varphi\circ\psi)\circ r_1^{-1}.$ By expressing these maps as compositions  of Dehn twists, we can realize this cutting and regluing operation via surgery. The isomorphism $\Psit_{-\data_1,-\data_2}$ is then defined in terms of 2-handle cobordism maps associated to such surgeries, as  below.
\end{remark}

Since  $\varphi\circ\psi$ and $\psi^{-1}$ fix the annulus $A_1$, these diffeomorphisms are isotopic to compositions of Dehn twists about nonseparating curves $a_1,\dots,a_m\subset R_1\ssm \partial A_1$,
\begin{align*}
\varphi\circ\psi&\sim D^{e_1}_{a_1}\circ\cdots\circ D^{e_n}_{a_n}\\
\psi^{-1}&\sim D^{e_{n+1}}_{a_{n+1}}\circ\cdots\circ D^{e_m}_{a_m}.
\end{align*}
Here $D_{a_i}$ is a positive Dehn twist about $a_i$, and  $e_i \in\{ \pm 1\}$. 

We next choose real numbers \[-3/4<t_m<\dots<t_{n+1}<-1/4<1/4<t_n<\dots<t_1<3/4,\] and pick some $t_i'$ between $t_i$ and the next greatest number in this list for every $i$ such that $e_i=+1.$ Let $(Y_1)_-$ be the 3-manifold obtained from $Y_1$ by performing $(-1)$-surgeries on the curves $r_1(a_i\times\{t_i\})$ for which $e_i=+1$,  with respect to the framings induced by the surfaces $r_1(R_1\times\{t_i\})$. Let $X_-$ be the 4-manifold obtained by attaching $(+1)$-framed 2-handles to $(Y_1)_-\times[0,1]$ along the curves $r_1(a_i\times\{t_i'\})\times\{1\}$ for which $e_i=+1$. One boundary component of $X_-$ is $-(Y_1)_-$. The other is canonically (up to isotopy) diffeomorphic to $Y_1$ since  the $(+1)$-surgery on $r_1(a_i\times\{t_i'\})$ cancels  the $(-1)$-surgery on  $r_1(a_i\times\{t_i\})$. We may therefore view $X_-$ as a cobordism from $(Y_1)_-$ to $Y_1$. This cobordism gives rise to a map
\[\HMtoc({-}X_-|{-}R_1;\Gamma_{-\nu}):\HMtoc(-(Y_1)_-|{-}R_1;\Gamma_{-\eta_1})\to\HMtoc(-Y_1|{-}R_1;\Gamma_{-\eta_1}),\] where $\nu$ is the cylinder $\nu= r_1(\eta_1\times\{0\})\times[0,1]\subset X_-$. 

Similarly,  let $X_+$  be the 
4-manifold obtained from $(Y_1)_-\times[0,1]$ by attaching $(+1)$-framed 2-handles along the curves $r_1(a_i\times\{t_i\})\times\{1\}$ for which $e_i =-1$. The boundary of $X_+$ is the union of $-(Y_1)_-$ with the 3-manifold $(Y_1)_+$ obtained from $(Y_1)_-$ by performing $(+1)$-surgeries on the curves $r_1(a_i\times\{t_i\})$ for which $e_i = -1$. Thus, $X_+$ gives rise to a map
\[\HMtoc({-}X_+|{-}R_1;\Gamma_{-\nu}):\HMtoc(-(Y_1)_-|{-}R_1;\Gamma_{-\eta_1})\to\HMtoc(-(Y_1)_+|{-}R_1;\Gamma_{-\eta_1}),\] where  $\nu= r_1(\eta_1\times\{0\})\times[0,1]\subset X_+$ in this case. 
This map and the one above are shown to be isomorphisms in \cite{bs3}.

As suggested in Remark \ref{rmk:cuttinggluing}, there is a unique isotopy class of diffeomorphisms \[\bar C:(Y_1)_+\to Y_2\] which restricts to $C$ on  $Y_1\ssm \inr(\Img(r_1))\subset (Y_1)_+$. Let  \[\Theta^{\bar C}:\HMtoc({-}(Y_1)_+|{-}R_1;\Gamma_{-\eta_1})\to \HMtoc({-}Y_2|{-}R_2;\Gamma_{-\eta_2})\] be the  isomorphism on monopole Floer homology induced by $\bar C$. The map \[\Psit_{-\data_1,-\data_2}:\HMtoc({-}Y_1|{-}R_1;\Gamma_{-\eta_1})\to \HMtoc({-}Y_2|{-}R_2;\Gamma_{-\eta_2})\] is defined to be the composition \[\Psit_{-\data_1,-\data_2}=\Theta^{\bar C} \circ\HMtoc({-}X_+|{-}R_1;\Gamma_{-\nu})\circ \HMtoc({-}X_-|{-}R_1;\Gamma_{-\nu})^{-1}.\] In \cite{bs3}, we proved that this map is independent of the choices made in its construction, up to multiplication by a unit in $\RR$. Having defined $\Psit_{-\data_1,-\data_2}$, we may now prove Theorem \ref{thm:well-defined}.

\begin{proof}[Proof of Theorem \ref{thm:well-defined}]

It suffices to show that \[\Psit_{-\data_1,-\data_2}(\invt(\data_1,\bar\xi_1))\doteq \invt(\data_2,\bar\xi_2)\] for the marked contact closures $(\data_1,\bar\xi_1)$ and $(\data_2,\bar\xi_2)$ above.  Note that the curves   $r_1(a_i\times\{t_i\})$ and $r_1(a_i\times\{t_i'\})$ are \emph{nonisolating} in  $r_1(R_1\times\{t_i\})$ and $r_1(R_1\times\{t_i'\})$ since each component of $R_1\ssm a_i$ intersects $\partial A_1$. We can  therefore make these curves   Legendrian for all $i=1,\dots,m$  by isotoping $\bar\xi_1$ slightly, according to  the \emph{Legendrian Realization Principle} \cite{kanda, honda2}. In a slight abuse of notation, let us simply  assume that these curves are Legendrian with respect to $\bar\xi_1$. The surface framings on these Legendrian curves agree with their contact framings since the curves are disjoint from the dividing sets on their respective surfaces (each $a_i$ is disjoint from $\partial A_1$). In particular, we can arrange that the $(\pm 1)$-surgeries performed in defining  $\Psit_{-\data_1,-\data_2}$ are actually  contact $(\pm 1)$-surgeries. Let $(\bar\xi_1)_{\pm}$ be the   contact structures on $(Y_1)_{\pm}$ induced by these surgeries. 
Since  $(+1)$-surgery on $r_1(a_i\times\{t_i'\})$ cancels  $(-1)$-surgery on  $r_1(a_i\times\{t_i\})$ contact geometrically as well as topologically, we have that
\[ \HMtoc({-}X_-|{-}R_1;\Gamma_{-\nu})(\psi((Y_1)_{-},(\bar\xi_1)_-))\doteq \invt(\data_1,\bar\xi_1)\] by Corollary \ref{cor:contactplusone}. The same corollary  tells us that 
\[\HMtoc({-}X_+|{-}R_1;\Gamma_{-\nu})(\psi((Y_1)_{-},(\bar\xi_1)_-))\doteq \psi((Y_1)_{+},(\bar\xi_1)_+).\] Finally, we can arrange that the diffeomorphism $\bar{C}$  is a contactomorphism, which implies \[\Theta^{\bar C}(\psi((Y_1)_{+},(\bar\xi_1)_+))\doteq \invt(\data_2,\bar\xi_2).\] Putting these pieces together, we have that \[\Psit_{-\data_1,-\data_2}(\invt(\data_1,\bar\xi_1))\doteq \invt(\data_2,\bar\xi_2).\] completing the proof. \end{proof}

\subsection{The well-definedness of $\invt(M,\Gamma,\xi)$}
\label{ssec:indepgenus}
We start by describing  the isomorphism   $\Psit_{-\data_1,-\data_2}$ in the case that $g(\data_1)\neq g(\data_2)$. The exposition here is tailored  to the setting of contact closures  and therefore
differs slightly from that in  \cite{bs3}. We then define contact handle attachment maps and prove Theorem \ref{thm:well-defined2}, which implies that $\psi(M,\Gamma,\xi)\in\SHMtfun(-M,-\Gamma)$ is well-defined.

Suppose $(\data_1,\bar\xi_1)$ and $(\data_2,\bar\xi_2)$ are marked contact closures of $(M,\Gamma,\xi)$. Let us first consider the case  in which \[g(\data_2)=g(\data_1)+1=g+1.\]
To define $\Psit_{-\data_1,-\data_2}$, we  first construct  two  additional marked contact closures  as follows. Let $F$ be an auxiliary surface  such that the closed surface formed by gluing $F$ to $R_+(\Gamma)$ has genus $g+1$. Then $F$ has genus at least two since $\data_1$ is formed from an auxiliary surface of genus at least one and $g(\data_1)=g$. It follows that there is an embedded subsurface $\Sigma \subset \inr(F)$ of genus one with two boundary components $c_1, c_2\subset \inr(F)$ such that $F\ssm \Sigma$ is connected. Let $\mathcal{A} = \{c, a_1,\dots,a_m\}$ be an arc configuration on $F$  contained in $F\ssm \Sigma$. Let $(M',\xi')$ be the contact preclosure formed from $F$, $\mathcal{A}$, and some choices of $A(\Gamma)$ and \[h:\partial F\times[-1,1]\to A(\Gamma).\] Note that the boundary components $\partial_{\pm} M'$ have genus $g+1$. As usual, the negative region on $\partial_+M'$ is an annular neighborhood $A(c)$ of $c$, by Remark \ref{rmk:divset}. Let $R$ be a copy of $\partial_+ M'$ and let $(Y,\bar\xi)$ be the closed contact manifold obtained by gluing $R\times[-1,1]$, equipped with the $[-1,1]$-invariant contact structure with negative region $A(c)\times\{t\}$ on each $R\times\{t\}$, to $M'$ by diffeomorphisms \begin{equation}\label{eqn:Rgluing}R\times\{\pm 1\}\to \partial M'_{\mp}\end{equation} which restrict to the ``identity" maps from $\Sigma\times\{\pm 1\}\subset R\times\{\pm 1\}$ to $\Sigma\times\{\mp 1\}\subset \partial M'_{\mp}$. Let $\eta$ be a curve in $R$ dual to the core of $A(c)$ which restricts to a properly embedded arc on $\Sigma$. Then  \[(\data = (Y, R, r,m,\eta),\bar\xi)\] is a genus $g+1$ marked contact closure of $(M,\Gamma,\xi)$, where $r$ and $m$ are the obvious embeddings of $R\times[-1,1]$ and $M$ into $Y$. 

Let $F'$ be the surface obtained from $\overline{F\ssm \Sigma}$ by gluing $c_1$ to $c_2$ via an orientation-reversing diffeomorphism \[f:c_1\to c_2\] which sends $c_1\cap \eta$ to $c_2\cap \eta$, as shown in Figure \ref{fig:Sigma}. Since $\mathcal{A}$ is disjoint from $\Sigma$, it descends to an arc configuration on $F'$. Let $(M'',\xi'')$ be the contact preclosure formed from $F'$, $\mathcal{A}$, $A(\Gamma)$, and $h$. Note that the boundary components $\partial M''_{\pm}$ are obtained from  $\partial M'_{\pm}$ by removing $\Sigma\times\{\pm 1\}$ and gluing  $c_1\times\{\pm 1\}$ to $c_2\times\{\pm 1\}$ by $f$. Let $R'$ be a copy of $\partial M''_{\pm}$ and let $(Y',\bar\xi')$ be the closed contact manifold obtained by gluing $R'\times[-1,1]$, equipped with the $[-1,1]$-invariant contact structure with negative region $A(c)\times\{t\}$ on each $R\times\{t\}$, to $M''$ by the diffeomorphisms \[R'\times\{\pm 1\}\to \partial M''_{\mp}\] induced by those in \eqref{eqn:Rgluing}. Then \[(\data' = (Y', R', r',m',\eta'),\bar\xi')\] is a genus $g$ marked contact closure of $(M,\Gamma,\xi)$, where $r'$ and $m'$ are the embeddings naturally induced by $r$ and $m$ and $\eta'\subset R'$ is the curve induced by $\eta$. 

\begin{figure}[ht]
\labellist
\tiny

\pinlabel $\eta$ at 91 97
\pinlabel $\eta'$ at 247 97
\pinlabel $c_1$ at 17 62
\pinlabel $c_2$ at 106 62
\pinlabel $c$ at 62 54
\pinlabel $c$ at 215 54
\endlabellist
\centering
\includegraphics[width=6.6cm]{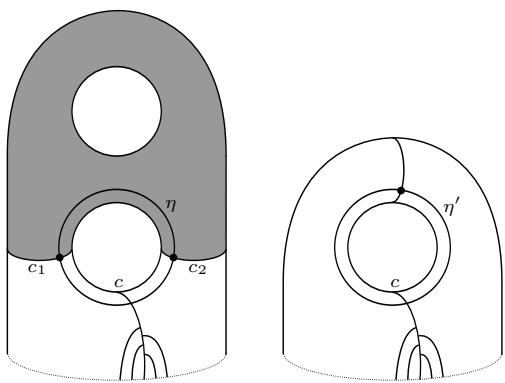}
\caption{Left, a portion of the surface $F$ with the region $\Sigma$ shaded. Right, the surface $F'$. The arc configuration $\mathcal{A}$ is shown in red, intersecting $\eta$ and $\eta'$ along the curve $c$.}
\label{fig:Sigma}
\end{figure}

The isomorphism $\Psit_{-\data',-\data}$ is defined in terms of a \emph{splicing} cobordism $W$, described below. Once we have defined this isomorphism,  we define $\Psit_{-\data_1,-\data_2}$  as in \eqref{eqn:data1data2}.
In defining $W$, the important observation is that the union of annuli \[(c_i\times[-1,1]\subset F\times[-1,1]\subset M')\,\,\cup\,\, (c_i\times[-1,1]\subset R\times[-1,1])\]  is an embedded torus $T_i = c_i\times S^1\subset Y$ for $i=1,2$. Together, these  tori cut $Y$ into    pieces $Y_M$ and $Y_\Sigma$ with \[-\partial Y_M = \partial Y_\Sigma = T_1\cup T_2.\]  In particular, $Y_M$ is the piece that contains $M$ and $Y_\Sigma$ is  the mapping torus of some diffeomorphism of $\Sigma$.  Note that the closed manifold obtained from $Y_M$ by gluing $T_1$ to $T_2$  by $f\times id$ is precisely $Y'$, while the manifold obtained from $Y_\Sigma$ in this way is a mapping torus $M(\Sigma')$ of some diffeomorphism of $\Sigma'$, where $\Sigma'$ is the closed genus two surface obtained from $\Sigma$ by gluing $c_1$ to $c_2$ by $f$. Let $\eta_{\Sigma'}\subset \Sigma'$ be the curve induced by $\eta$. The splicing cobordism \[W: Y' \sqcup M(\Sigma') \to Y\] is then defined by gluing the products $Y_M\times[0,1]$ and $Y_\Sigma \times [0,1]$ to $T_1\times S$, where $S$ is the saddle cobordism depicted in Figure \ref{fig:splicing}. We glue these pieces along the ``horizontal" portions of their boundaries according to the schematic in that figure, making use of the map $f\times id$. This cobordism induces a map 
\begin{align*}
\HMtoc({-}W|{-}R';\Gamma_{{-}\nu})&:\HMtoc({-}Y'|{-}R';\Gamma_{{-}\eta'})\otimes_{\RR}\HMtoc(-M(\Sigma')|{-}\Sigma';\Gamma_{{-}\eta_{\Sigma'}})\\
&\to \HMtoc({-}Y|{-}R;\Gamma_{{-}\eta}),\end{align*}
where $\nu\subset W$ is a pair-of-pants cobordism from $\eta'\sqcup \eta_{\Sigma'}$ to $\eta$. We  define \[\Psit_{-\data',-\data}(-) = \HMtoc({-}W|{-}R';\Gamma_{{-}\nu})(-\otimes 1),\] where $1$ is a generator of $\HMtoc(-M(\Sigma')|{-}\Sigma';\Gamma_{{-}\eta_{\Sigma'}})\cong \RR$. We   then define  \begin{equation}\label{eqn:data1data2}\Psit_{-\data_1,-\data_2} = \Psit_{-\data,-\data_2}\circ\Psit_{-\data',-\data}\circ\Psit_{-\data_1,-\data'}.\end{equation} Here, $\Psit_{-\data,-\data_2}$ and $\Psit_{-\data_1,-\data'}$ are the maps defined in the previous subsection for closures of the same genus.

\begin{figure}[ht]
\labellist
\small


\pinlabel $S$ at 142 130
\pinlabel $T_1\times S$ at 490 130
\pinlabel $Y_M\times [0,1]$ at 490 239
\pinlabel $Y_{\Sigma}\times[0,1]$ at 490 18
\pinlabel $Y'$ at 346 207
\pinlabel $Y$ at 620 130
\pinlabel $M(\Sigma')$ at 332 53
\endlabellist
\centering
\includegraphics[width=10cm]{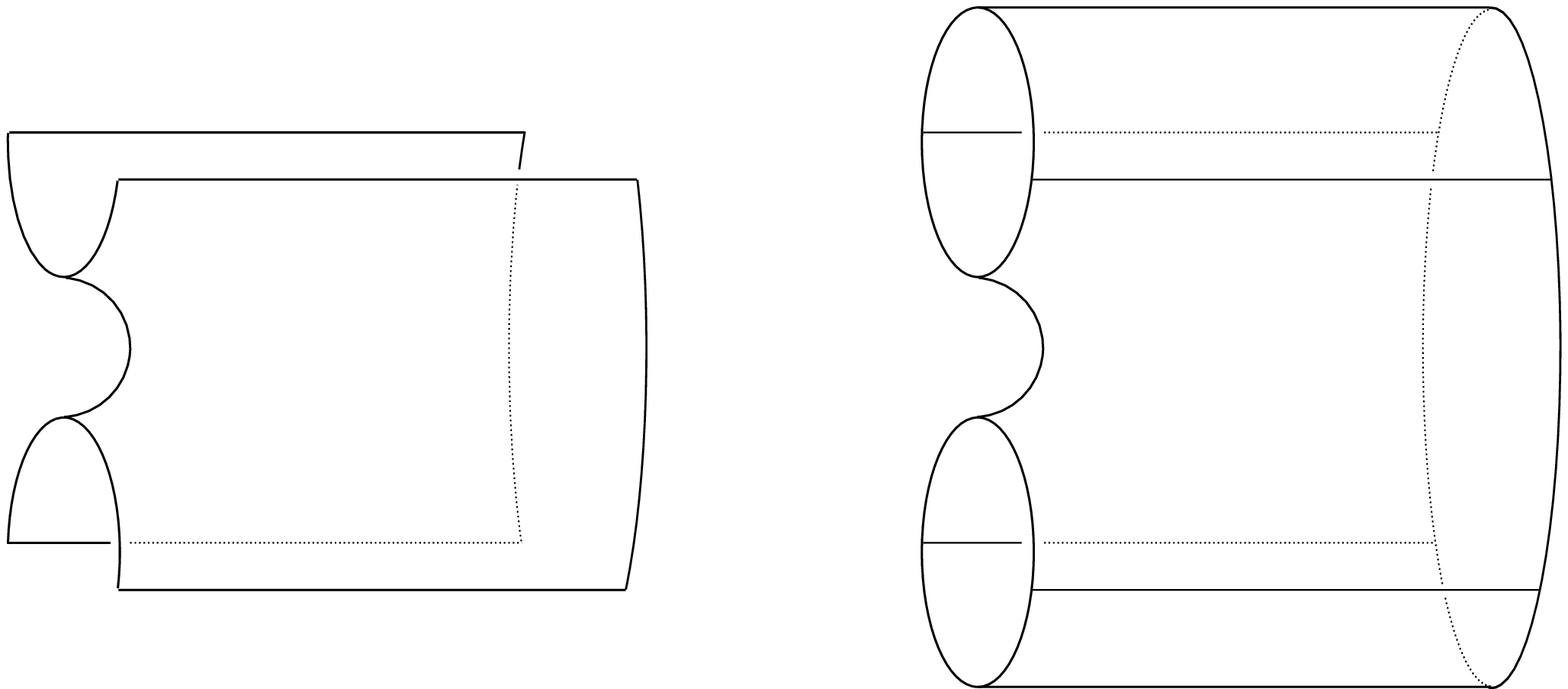}
\caption{Left, the saddle $S$. Right, a schematic of the splicing cobordism $W$.}
\label{fig:splicing}
\end{figure}

In the case that $g(\data_2)=g(\data_1)-1$, we define \[\Psit_{-\data_1,-\data_2}=\Psit_{-\data_2,-\data_1}^{-1}.\]
For the general case, we choose a sequence  $(\data^1,\bar\xi^1),\dots,(\data^n,\bar\xi^n)$  of marked contact closures of $(M,\Gamma,\xi)$ such that $(\data^1,\bar\xi^1)=(\data_1,\bar\xi_1)$, $(\data^n,\bar\xi^n)=(\data_2,\bar\xi_2)$, and \[|g(\data^{i+1})-g(\data^i)|\leq 1\] for all $i=1,\dots,n-1$. Then we define \[\Psit_{-\data_1,-\data_2} = \Psit_{-\data^{n-1},-\data^n}\circ\dots\circ \Psit_{-\data^1,-\data^2}.\] We proved in \cite{bs3} that this map is independent of the choices made in its construction, up to multiplication by a unit in $\RR$.


Below, we define maps on $\SHMtfun$ associated to contact $0$-, $1$-, $2$-, and $3$-handle attachments. We will use the $2$-handle  attachment maps at the end to prove Theorem \ref{thm:well-defined2}, which implies  that the elements $\invt^g(M,\Gamma,\xi)$ are equal for large $g$ and, hence, that $\invt(M,\Gamma,\xi)$ is well-defined.

\subsubsection{$0$-handle attachments}
\label{sssec:0handles}
Attaching a contact $0$-handle to $(M,\Gamma,\xi)$ is equivalent to taking the disjoint union of $(M,\Gamma,\xi)$ with the  Darboux ball $(B^3,S^1, \xi_{std})$. Let $(M_0,\Gamma_0,\xi_0)$ be this disjoint union. We claim that every marked contact closure of $(M_0,\Gamma_0,\xi_0)$ is a marked contact closure of $(M,\Gamma,\xi)$. 
To see this,  think of the  Darboux ball as the product  sutured contact manifold $(D^2\times[-1,1],\partial D^2\times\{0\},\xi_{D^2})$, as in Subsection \ref{ssec:closureexamples}. Let $F_0$ be an auxiliary surface for $(M_0,\Gamma_0,\xi_0)$ with arc configuration $\mathcal{A}_0 = \{c,a_1,\dots,a_m\}$ so that $a_1$ is the unique arc meeting the boundary component $\partial_1F_0$. We form a contact preclosure $(M'_0,\xi'_0)$ of $(M_0,\Gamma_0,\xi_0)$ by attaching $F_0\times[-1,1]$ to $M_0$ such that $\partial_1F_0\times[-1,1]$ is glued to $\partial D^2\times[-1,1]$ by a map \[h:\partial_1 F_0\times[-1,1]\to\partial D^2\times[-1,1]\] of the form $f\times{\rm id}$ for some diffeomorphism $f:\partial_1 F_0\to \partial D^2$. Let $A(c)$ denote the negative annular region on $\partial_+ M'_0$. Let $R$ be a copy of $\partial_+ M'_0$. Let $(Y_0,\bar\xi_0)$ be the closed contact manifold obtained by gluing $R\times[-1,1]$, equipped with the $[-1,1]$-invariant contact structure with negative region $A(c)\times\{t\}$ on each $R\times\{t\}$, to $M'_0$ by  diffeomorphisms \[R\times\{\pm 1\}\to \partial_{\mp} M'_0\] which identify  dividing sets. Let $\eta$ be a curve in $R$ dual to the core of $A(c)$. Then  \[(\data_0 = (Y_0, R, r,m_0,\eta),\bar\xi_0)\] is a marked contact closure of $(M_0,\Gamma_0,\xi_0)$, where $r$ and $m_0$ are the obvious embeddings of $R\times[-1,1]$ and $M_0$ into $Y$. 

Note that $(M'_0,\xi'_0)$ is   a contact preclosure of $(M,\Gamma,\xi)$ as well, formed from  the auxiliary surface $F=F_0\cup_f D^2$ and the arc configuration $\mathcal{A}= \{c,a_2,\dots,a_m\}$. Thus, \[(\data = (Y_0, R, r,m,\eta),\bar\xi = \bar\xi_0)\] is a marked contact closure of $(M,\Gamma,\xi)$, where $m$ is the restriction of $m_0$ to $M$. In particular, $\SHMt(-\data) = \SHMt(-\data_0)$. We define the $0$-handle attachment map \[\mathscr{H}_0:\SHMtfun(-M,-\Gamma)\to \SHMtfun(-M_0,-\Gamma_0)\] to be the morphism determined by the equivalence class of the identity map  from $\SHMt(-\data)$ to $\SHMt(-\data_0)$, which we will denote in this case by \[id_{-\data,-\data_0}:\SHMt(-\data)\to\SHMt(-\data_0).\] 
To prove that $\mathscr{H}_0$ is well-defined, we need only show that if $(\data_0,\bar\xi_0)$ and $(\data_0',\bar\xi'_0)$ are marked contact closures of $(M_0,\Gamma_0,\xi_0)$ constructed as above, and  $(\data,\bar\xi)$ and $(\data',\bar\xi')$ are the corresponding marked contact closures of $(M,\Gamma,\xi)$, then the  diagram   \[ \xymatrix@C=55pt@R=30pt{
\SHMt(-\data) \ar[r]^-{id_{-\data,-\data_0}} \ar[d]_{\Psit_{-\data,-\data'}} & \SHMt(-\data_0) \ar@<-1.5ex>[d]^{\Psit_{-\data_0,-\data_0'}} \\
\SHMt(-\data') \ar[r]_-{id_{-\data',-\data_0'}} & \SHMt(-\data_0')
} \] commutes, up to multiplication by a unit in $\RR$. But this is clear: $\Psit_{-\data_0,-\data_0'}$ is a composition of  maps associated to 2-handle   and splicing cobordisms, and $\Psit_{-\data,-\data'}$ can be defined via the exact same composition. 
Note that \[id_{-\data,-\data_0}(\invt(\data,\bar\xi))\doteq \invt(\data_0,\bar\xi_0)\]  since  $\invt(\data,\bar\xi) = \invt(\data_0,\bar\xi_0)$ in  $\SHMt(-\data) = \SHMt(-\data_0)$. We therefore have the following.

\begin{proposition}
\label{prop:H0}
$\mathscr{H}_0(\invt^g(M,\Gamma,\xi))=\invt^g(M_0,\Gamma_0,\xi_0)$ for each $g\geq g(M,\Gamma)$. \qed
  \end{proposition}

\subsubsection{$1$-handle attachments}
\label{sssec:1handles}
Suppose $D_-$ and $D_+$ are disjoint embedded disks in $\partial M$ which each intersect $\Gamma$ in a single properly embedded arc. To attach a contact $1$-handle  to $(M,\Gamma,\xi)$ along these disks, we glue $(D^2\times[-1,1],\xi_{D^2})$ to $(M,\Gamma,\xi)$ by diffeomorphisms \[D^2\times\{-1\}\to D_- \ \ {\rm and} \ \  D^2\times\{+1\}\to D_+,\] which preserve and reverse orientations, respectively, and  identify dividing sets, and then we round corners, as illustrated in Figure \ref{fig:onehandle}. Let $(M_1,\Gamma_1,\xi_1)$ be the resulting sutured contact manifold. As in the $0$-handle case, we claim that every marked contact closure of $(M_1,\Gamma_1,\xi_1)$ is also a marked contact closure  of $(M,\Gamma,\xi)$. 

\begin{figure}[ht]

\centering
\includegraphics[width=14cm]{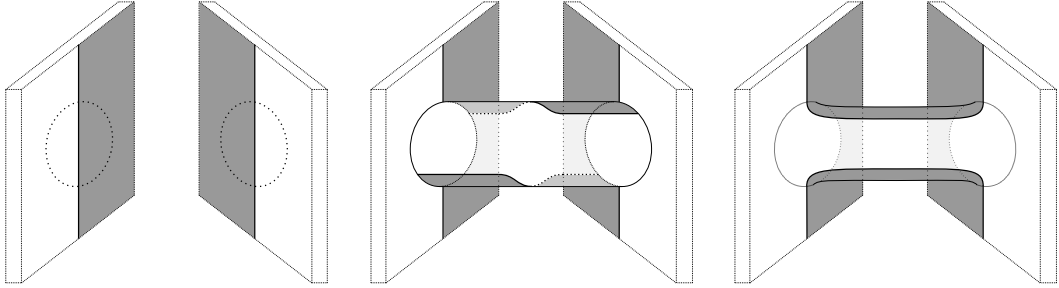}
\caption{Left, a portion of a vertical invariant neighborhood of $\partial M$ near the disks $D_-,D_+\subset \partial M$, whose boundaries are dotted. Middle, attaching the contact $1$-handle.  Right, the $1$-handle attachment after rounding corners.}
\label{fig:onehandle}
\end{figure}

The rough idea is that if $F_1$ is an auxiliary surface for $(M_1,\Gamma_1)$, then, in the corresponding preclosure $M_1'$, the union of $F_1\times[-1,1]$ with the contact $1$-handle is a product $F\times[-1,1]$, where $F$ is an auxiliary surface for $(M,\Gamma)$, so that $M_1'$ is also a preclosure of $(M,\Gamma)$. We make this  precise as follows. Choose an auxiliary surface $F_1$ for $(M_1,\Gamma_1)$, a neighborhood $A(\Gamma_1)$, and  a diffeomorphism \[h_1:\partial F_1\times[-1,1]\to A(\Gamma_1).\]  We perturb $\xi_1$ as usual, so that, near the contact $1$-handle, the resulting dividing set intersects $A(\Gamma_1)$  as shown in the upper right of Figure \ref{fig:onehandleaux}.  We choose an arc configuration $\mathcal{A}_1$  on $F_1$ so that the negative region  of $\Gamma_{\mathcal{A}_1}$ on a portion of the convex surface $F_1\times\{1\}\subset (F_1\times[-1,1],\Xi_{\mathcal{A}_1})$  glued near the  $1$-handle consists of neighborhoods of four arcs, as shown in the lower right of Figure \ref{fig:onehandleaux}. Let $(M_1',\xi_1')$ denote the resulting contact preclosure of $(M_1,\Gamma_1,\xi_1)$.

Let $\gamma_+$ and $\gamma_-$ be the  arcs of $\partial F_1$ such that $\gamma_{\pm}\times\{1\}\subset \partial F_1\times\{1\}$ are mapped to the  contact $1$-handle by $h_1$. Let $F$ be the surface obtained  by attaching a two-dimensional $1$-handle  to $F_1$ with feet at $\gamma_{\pm}$. Note that $F$ is an auxiliary surface for $(M,\Gamma)$. We use it to construct a contact preclosure of $(M,\Gamma,\xi)$ as follows. Let $A(\Gamma)\subset \partial M$ be a neighborhood of $\Gamma$ which agrees with $A(\Gamma_1)$ outside of the attaching disks $D_{\pm}$. We perturb $\xi$ so that the resulting dividing set agrees with that of the perturbed $\xi_1$ outside of $D_{\pm}$ and is disjoint from $A(\Gamma)$ inside these disks, as shown in the upper left of Figure \ref{fig:onehandleaux}. By choosing $h_1$ more carefully to begin with, we can assume that it extends to a map \[h:\partial F\times[-1,1]\to A(\Gamma)\] which agrees with $h_1$ away from $\gamma_{\pm}\times[-1,1]$. Let $\mathcal{A}$ be the arc configuration on $F$ induced by $\mathcal{A}_1$, and let $(M',\xi')$ be the corresponding contact preclosure of $(M,\Gamma,\xi)$.


\begin{figure}[ht]

\centering
\includegraphics[width=9.6cm]{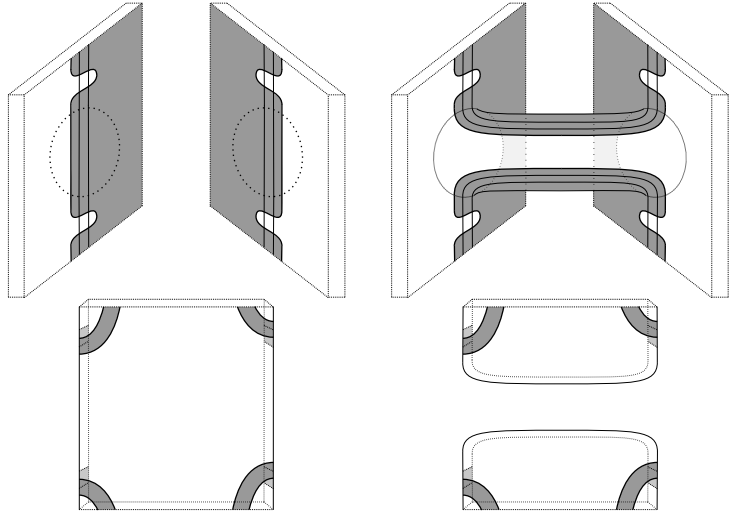}
\caption{Upper left, a portion of a vertically invariant neighborhood of $\partial M$ near  $D_+,D_-$ showing  $A(\Gamma)$ and the perturbed dividing set. Upper right, the corresponding portions of $M_1$ and $A(\Gamma_1)$. Lower left, a portion of $F\times[-1,1]$ with the negative regions on $F\times\{1\}$ and $\partial F\times[-1,1]$ shaded. Lower right, the corresponding portion of $F_1\times[-1,1]$. The 3-balls $N\subset M'$ and $N_1\subset M_1'$ are obtained by gluing the portions of $F\times[-1,1]$ and $F_1\times[-1,1]$ shown here to the portions of $M$ and $M_1$ shown here via the maps $h$ and $h_1$.}
\label{fig:onehandleaux}
\end{figure}

Let $N\subset M'$ be the 3-ball obtained by gluing the portion of $F\times[-1,1]$ shown in Figure \ref{fig:onehandleaux} to the portion of $M$ shown there, and define $N_1\subset M_1'$ analogously. Note that  \begin{equation}\label{eqn:NM}(M'\ssm N,\xi')=(M_1'\ssm N_1,\xi_1').\end{equation}  We claim that, after rounding corners,  $N$ and $N_1$ are Darboux balls. This will imply that the identity map on the manifold in \eqref{eqn:NM} extends (uniquely, up to isotopy) to a contactomorphism \[(M',\xi')\to(M_1',\xi_1').\] For the  claim, it is enough to show that $N$ and $N_1$ are tight (as there is  a unique tight ball). But since $N$ and $N_1$ only depend on  $\xi$ on a vertically invariant neighborhood of $\partial M$, it suffices to find some $(M,\Gamma,\xi)$ such that all contact preclosures of $(M,\Gamma,\xi)$ and $(M_1,\Gamma_1,\xi_1)$ are tight. We can take $(M,\Gamma,\xi)$ to be the  Darboux ball $(B^3,S^1,\xi_{std})$, in which case  $(M_1,\Gamma_1,\xi_1)$ is the product sutured contact handlebody $H(S)$ for a surface $S$ with genus $0$ and $2$ boundary components. It follows from the results in Subsections \ref{ssec:closureexamples} and \ref{ssec:invariant-fixed-g} that   contact preclosures of both the Darboux ball and this handlebody are always tight, settling the claim.

Now, let $R$ be a copy of $\partial_+ M'_1$, and let $(Y_1,\bar\xi_1)$ be the closed contact manifold obtained by gluing $R\times[-1,1]$, equipped with the appropriate $[-1,1]$-invariant contact structure, to $M'_1$ in the usual way. Consider the marked contact closure  \[(\data_1 = (Y_1, R, r,m_1,\eta),\bar\xi_1)\] of $(M_1,\Gamma_1,\xi_1)$, where $\eta$ is an appropriately chosen curve on $R$, and $r$ and $m_1$ are the obvious embeddings of $R\times[-1,1]$ and $M_1$ into $Y_1$. Since  $(M',\xi')$ and $(M_1',\xi_1')$ are  contactomorphic by the identity map outside of balls, we have that \[(\data = (Y_1, R, r,m,\eta),\bar\xi=\bar\xi_1)\] is a marked contact closure of $(M,\xi)$, where  $m$ is the restriction of $m_1$ to $M$. In particular, $\SHMt(-\data) = \SHMt(-\data_1)$. We define the $1$-handle attachment map \[\mathscr{H}_1:\SHMtfun(-M,-\Gamma)\to \SHMtfun(-M_1,-\Gamma_1)\] to be the morphism determined by the equivalence class of the identity map  from $\SHMt(-\data)$ to $\SHMt(-\data_1)$. The proof that $\mathscr{H}_1$ is well-defined is then exactly as in the $0$-handle case. 
Since the identity map $id_{-\data,-\data_1}$ sends $\invt(\data,\bar\xi)$ to $\invt(\data_1,\bar\xi_1)$, we have the following.

\begin{proposition}
\label{prop:H1}
$\mathscr{H}_1(\invt^g(M,\Gamma,\xi))=\invt^g(M_1,\Gamma_1,\xi_1)$ for each $g\geq g(M,\Gamma)$. \qed
\end{proposition}

\subsubsection{$2$-handle attachments}
\label{sssec:2handles}
Suppose $\gamma$ is an embedded curve in $\partial M$ which intersects $\Gamma$ in two points. Let $A(\gamma)$ be an annular neighborhood of $\gamma$ intersecting $\Gamma$ in two cocores. To attach a contact $2$-handle  to $(M,\Gamma,\xi)$ along $A(\gamma)$, we glue $(D^2\times[-1,1],\xi_{D^2})$ to $(M,\Gamma,\xi)$ by an orientation-reversing diffeomorphism \[\partial D^2\times [-1,1] \to A(\gamma)\] which identifies positive regions with negative regions, and then  round corners, as illustrated in Figure \ref{fig:twohandle2}. Let $(M_2,\Gamma_2,\xi_2)$ be the resulting sutured contact manifold.

\begin{figure}[ht]

\centering
\includegraphics[width=11.3cm]{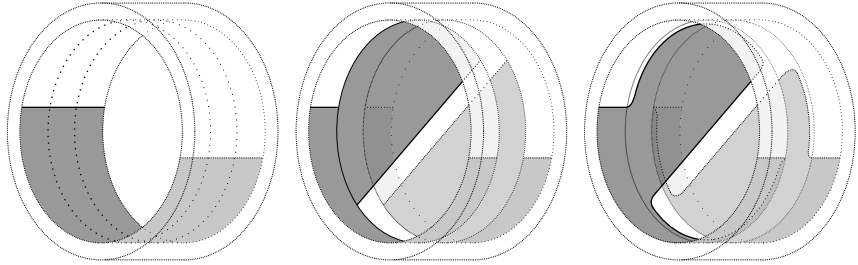}
\caption{Left, a portion of a vertically invariant neighborhood of $\partial M$ near $A(\gamma)\subset\partial M$, whose boundary is dotted. This portion is a neighborhood $N$ of the Legendrian curve $\gamma'$. Middle,  attaching the contact $2$-handle. Right, the $2$-handle attachment after rounding corners.}
\label{fig:twohandle2}
\end{figure}

Now, consider the sutured contact manifold $(M_{1},\Gamma_{1},\xi_{1})$ obtained from $(M_2,\Gamma_2,\xi_2)$ by attaching a contact $1$-handle along disks in the interiors of the $D^2\times\{\pm 1\}$ boundary components of the contact $2$-handle, as indicated in Figure \ref{fig:twohandle}, and let \[\mathscr{H}_1:\SHMtfun(-M_2,-\Gamma_2)\to \SHMtfun(-M_1,-\Gamma_1)\] be the corresponding $1$-handle attachment map, as defined in Subsubsection \ref{sssec:1handles}. It is easy to see that $(M_{1},\Gamma_{1})$ is diffeomorphic to the sutured manifold obtained from $(M,\Gamma)$ by performing $\partial M$-framed surgery on a copy $\gamma'$ of $\gamma$ in the interior of $M$. By the Legendrian Realization Principle, we can assume that $\gamma'$ is Legendrian in $(M,\Gamma,\xi)$ since $\gamma$ is nonisolating in $\partial M$. Moreover, since $\gamma$ intersects $\Gamma$ in exactly two points, the $\partial M$-framing on $\gamma'$ is one more than its contact framing. Below, we argue that $(M_{1},\Gamma_{1},\xi_{1})$ is in fact contactomorphic (by a canonical isotopy class of contactomorphisms) to the result of contact $(+1)$-surgery on  $\gamma'$. 

\begin{figure}[ht]
\centering
\includegraphics[height=7.8cm]{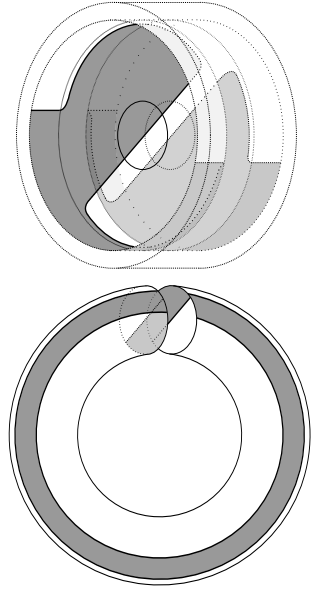}
\caption{Attaching a contact $1$-handle  to form $M_{1}$. The circles on $D^2\times\{\pm 1\}$ indicate   where the feet of this handle are to be attached. The union of the $1$-handle  below with the portion of $M_1$ shown above  is the solid torus $N_{1}$.}
\label{fig:twohandle}
\end{figure}

To see this, let $N\subset M$ be the solid torus on the left in Figure \ref{fig:twohandle2} and let $N_{1}\subset M_{1}$ be the solid torus obtained from $N$ by attaching the $1$- and $2$-handles  as indicated in Figures \ref{fig:twohandle2} and \ref{fig:twohandle}.  After slight modification, $N$ and $N_{1}$ can be made to have convex boundaries.  Note that  \begin{equation}\label{eqn:NM'}(M\ssm N,\Gamma,\xi)=(M_{1}\ssm N_{1},\Gamma_{1},\xi_{1}).\end{equation} Furthermore, the identity map, restricted to $\partial M\ssm N  = \partial M_1\ssm N_1$, extends uniquely (up to isotopy) to a diffeomorphism   \[(\partial M, \Gamma)\to(\partial M_{1}, \Gamma_{1}).\] It follows that the identity map on the manifold in \eqref{eqn:NM'} extends uniquely (up to isotopy) to a contactomorphism \[(M\ssm N',\Gamma,\xi)\to(M_{1}\ssm N_{1}',\Gamma_{1},\xi_{1}),\] where $N'\subset \inr(N)$ is the  solid torus with convex boundary obtained by removing a vertically invariant collar $\partial N\times[-\infty,0)$ from $N$, and $N'_{1}$ is defined from $N_{1}$ analogously. In other words, there is a canonical contactomorphism (up to isotopy), \[f:(M',\Gamma',\xi')\to(M_{1},\Gamma_{1},\xi_{1}),\] where $(M',\Gamma',\xi')$ is the contact manifold obtained from $(M,\Gamma,\xi)$ by removing the solid torus neighborhood $N'$ of $\gamma'$ and gluing back in a  solid torus contactomorphic to $N_{1}'$ according to the contact framing on $\gamma'$ plus one. To show that this operation is actually a contact $(+1)$-surgery, all that remains is to show that $N$ and $N_{1}$ are tight (as there is a unique tight solid torus with the given dividing set on its boundary).

Since $N$ and $N_{1}$ only depend on  $\xi$ on a vertically invariant neighborhood of $\partial M$, it suffices to find some $(M,\Gamma,\xi)$ such that  both $(M,\Gamma,\xi)$ and some sutured contact manifold obtained from $(M,\Gamma,\xi)$ by attaching contact $2$- and $1$-handles as above are tight. We can take $(M,\Gamma,\xi)$ to be the tight solid torus $H(S)$ for a surface  $S$ with genus 0 and 2 boundary components. Etg{\"u} and {\"O}zba{\u{g}}c{\i} show in \cite[Example 3]{etguozbagci} that one can obtain the  Darboux ball  by  attaching a contact $2$-handle to this solid torus. We proved in Proposition \ref{prop:darboux-ball-invt} that any marked contact closure $(\data,\bar\xi)$ of $(B^3,S^1,\xi_{std})$ has nonzero invariant $\invt(\data,\bar\xi)$. This then implies, by the earlier results in this section, that any marked contact closure of the sutured contact manifold obtained from the Darboux ball by attaching a contact $1$-handle also has nonzero invariant. The manifold resulting from this $1$-handle attachment is therefore tight as well. 

Thus, $(M',\Gamma',\xi')$ is obtained from $(M,\Gamma,\xi)$ via contact $(+1)$-surgery along the Legendrian curve $\gamma'$. 

In order to define the $2$-handle map $\mathscr{H}_2$, we first define the morphism $F_K$ 
 in Proposition \ref{prop:shm-legendrian-surgery}. Suppose \[(\data = (Y,R,r,m,\eta),\bar\xi)\] is a marked contact closure of $(M,\Gamma,\xi)$. Let $(Y',\bar\xi')$ be the contact 3-manifold obtained from $Y$ by performing contact $(+1)$-surgery on  $m(K)$ for some Legendrian knot $K\subset M$. Then  \[(\data' = (Y',R,r',m',\eta),\bar\xi')\] is a contact closure of $(M',\Gamma',\xi')$, where  $r'$ is the map induced by $r$ and $m'$ is the embedding of $M'$ into $Y'$ induced by $m$. Let $W$ be the $2$-handle cobordism from $Y$ to $Y'$ corresponding to the above surgery. We define 
 \[F_{K}: \SHMtfun(-M,-\Gamma)\to\SHMtfun(-M',-\Gamma')\] 
  to be the morphism induced by the map 
\[\HMtoc({-}W|{-}R;\Gamma_{-\nu}):\SHMt(-\data)\to\SHMt(-\data'),\] where $\nu\subset W$ is the natural cylindrical cobordism from $r(\eta\times\{0\})\subset Y$ to $r'(\eta\times\{0\})\subset Y'$. 
To prove that $F_{K}$ is well-defined, we must show that the  diagram 
\[ \xymatrix@C=41pt@R=32pt{
\SHMt(-\data_1) \ar[rr]^-{\HMtoc({-}W_1|{-}R_1;\Gamma_{-\nu_1})} \ar[d]_{\Psit_{-\data_{1},-\data_{2}}} && \SHMt(-\data_1') \ar[d]^{\Psit_{-\data_1',-\data_2'}} \\
\SHMt(-\data_{2}) \ar[rr]_-{\HMtoc({-}W_2|{-}R_2;\Gamma_{-\nu_2})} &&  \SHMt(-\data_2')
} \]
commutes, up to multiplication by a unit in $\RR$,  for any two marked contact closures 
$(\data_1 ,\bar\xi_1)$ and $(\data_2,\bar\xi_2)$  of $(M,\Gamma,\xi)$, where $(\data_i' ,\bar\xi_i')$ is the marked contact closure  of $(M',\Gamma',\xi')$ induced by $(\data_i,\bar\xi_i)$ and $W_i$ is the $2$-handle cobordism from $Y_i$ to $Y_i'$.
But this  follows from the  commutativity of the cobordisms used to define these maps:  $W_1$ and $W_2$ are built by attaching $2$-handles along curves in the regions $m_1(M)$ and $m_2(M)$, while the vertical isomorphisms are defined from cobordisms built by attaching 2-handles or splicing along tori outside of these regions. Since $\HMtoc({-}W|{-}R;\Gamma_{-\nu})$ sends $\invt(\data,\bar\xi)$ to $\invt(\data',\bar\xi')$, by Corollary \ref{cor:contactplusone}, we have the following, which proves Proposition \ref{prop:shm-legendrian-surgery}.

\begin{proposition}
\label{prop:Fgamma}
$F_{K}(\invt^g(M,\Gamma,\xi)) = \invt^g(M',\Gamma',\xi')$ for each $g\geq g(M,\Gamma)$. \qed
\end{proposition}

We now define the $2$-handle attachment map \[\mathscr{H}_2:\SHMtfun(-M,-\Gamma)\to \SHMtfun(-M_2,-\Gamma_2),\] to be the composition of morphisms \[\mathscr{H}_2 = \mathscr{H}_1^{-1}\circ \SHMtfun(f)\circ F_{\gamma'}.\] That $\mathscr{H}_2$ is independent of $\gamma'$  follows from the fact that any two such Legendrian realizations of $\gamma$ are related by an ambient isotopy of $M$ supported in $N$.  
Unpacking the composition above, we see that $\mathscr{H}_2$ may  also be formulated  as follows. Suppose $(\data,\bar\xi)$ is a marked contact closure of $(M,\Gamma,\xi)$ and let $(\data',\bar\xi')$ be the induced marked contact closure of the surgered manifold $(M',\Gamma',\xi')$. Then \[(\data_{2}=(Y',R',r',m_{2},\eta'),\bar\xi_2=\bar\xi')\] is a marked contact closure of $(M_{2},\Gamma_{2},\xi_{2})$, where $m_{2}$ is the restriction of $m'\circ f^{-1}$ to $M_2\subset M_{1}$. Let \[id_{-\data',-\data_{2}}: \SHMt(-\data')\to\SHMt(-\data_{2})\] be the identity map on $\SHMt(-\data')=\SHMt(-\data_{2})$. Then $\mathscr{H}_2$ is the morphism induced by the map 
\[id_{-\data',-\data_{2}}\circ \HMtoc({-}W|{-}R;\Gamma_{-\nu}):\SHMt(-\data)\to \SHMt(-\data_2).\]
Note that Propositions \ref{prop:Fgamma} and \ref{prop:H1} imply the following.

\begin{proposition}
\label{prop:H2}
$\mathscr{H}_2(\invt^g(M,\Gamma,\xi))=\invt^g(M_2,\Gamma_2,\xi_2)$ for each $g\geq g(M,\Gamma)$. \qed
\end{proposition}

\subsubsection{$3$-handle attachments}
\label{sssec:3handles}
Attaching a contact 3-handle to $(M,\Gamma,\xi)$ amounts to gluing the Darboux ball  to $(M,\Gamma,\xi)$ along an $S^2$ boundary component of $M$ with one dividing curve. Let $(M_3,\Gamma_3,\xi_3)$ be the result of this gluing. We will first assume that $\partial M$ is disconnected, so that $M_3$ has boundary. Let $p$ be a point in $M_3$ in the interior of this Darboux ball. Then there is a canonical isotopy class of contactomorphisms \[f:(M,\Gamma,\xi)\to (M',\Gamma',\xi'),\] where $(M',\Gamma',\xi')$ is the sutured contact manifold obtained by taking the contact connected sum of $(M_3,\Gamma_3,\xi_3)$ with $(B^3,S^1,\xi_{std})$ at the point $p$. Let $(M_{0},\Gamma_{0},\xi_{0})$ be the disjoint union of $(M_3,\Gamma_3,\xi_3)$ with $(B^3,S^1,\xi_{std}),$ and let \[\mathscr{H}_0: \SHMtfun(-M_3, -\Gamma_3)\to \SHMtfun(-M_{0},-\Gamma_{0})\] be the corresponding  $0$-handle attachment map, as defined in Subsubsection \ref{sssec:0handles}. Suppose \[(\data_{0} = (Y_{0},R,r,m,\eta), \bar\xi_{0})\] is a marked contact closure of $(M_{0},\Gamma_{0},\xi_{0})$. Then \[(\data' = (Y',R,r,m',\eta), \bar\xi')\] is a marked contact closure of $(M',\Gamma',\xi')$, where $(Y',\bar\xi')$ is the self contact connected sum obtained from $(Y_{0},\bar\xi_{0})$ by removing Darboux balls around $m(p)$ and some point in $m(B^3)\subset Y_{0}$ and gluing in $S^2\times I$, equipped with a tight, $I$-invariant contact structure, and $m'$ is the  embedding of $M'$ into $Y'$ induced by $m$. In particular, $(Y',\bar\xi')$ is a contact connected sum of $(Y_{0},\bar\xi_{0})$ with the tight $S^1\times S^2$.
Now, there is a natural Stein $1$-handle cobordism \[(W,\omega):(Y_{0},\bar\xi_{0})\to(Y',\bar\xi').\] Let $\nu\subset W$ be a  cylindrical cobordism from $r(\eta)\subset Y_{0}$ to $r(\eta)\subset Y'$. Then the map  \[\HMtoc(W|{-}R;\Gamma_{-\nu}): \SHMt(-\data')\to\SHMt(-\data_{0})\] sends $\invt(\data',\xi')$ to $\invt(\data_{0},\xi_{0})$, up to multiplication by a unit in $\RR$, by Theorem \ref{thm:ht-functoriality}. We define \[F_{\#}:\SHMtfun(-M',-\Gamma')\to \SHMtfun(-M_0,-\Gamma_0)\] to be the morphism determined by the equivalence class of this map. That $F_{\#}$ is well-defined follows from similar considerations as before; namely, these Stein 1-handle cobordisms are attached along balls in the interiors of $Y'$ and $Y_0$ and therefore commute with the 2-handle and splicing cobordisms used to define the isomorphisms in the systems $\SHMtfun(-M',-\Gamma')$ and $\SHMtfun(-M_0,-\Gamma_0)$. 
We  define the $3$-handle attachment map \[\mathscr{H}_3:\SHMtfun(-M,-\Gamma)\to\SHMtfun(-M_3,-\Gamma_3)\] to be the composition \[\mathscr{H}_3 = \mathscr{H}_0^{-1}\circ F_{\#}\circ\SHMtfun(f).\] By Proposition \ref{prop:H0}, we have the following. 

\begin{proposition}
\label{prop:H3}
$\mathscr{H}_3(\invt^g(M,\Gamma,\xi))=\invt^g(M_3,\Gamma_3,\xi_3)$ for each $g\geq g(M,\Gamma)$. \qed
\end{proposition}



 Suppose now that $(Y,\xi)$ is a closed contact manifold, and let $Y(p)$ be the sutured contact manifold obtained from $(Y,\xi)$ by removing a Darboux ball around $p$. Below, we use similar ideas to construct the morphism \[F_p:\SHMtfun(-Y(p)) \to \HMtoc(-Y) \otimes_\mathbb{Z} \RR\] in Proposition \ref{prop:darboux-complement}. Note that $(Y,\xi)$ is obtained from $Y(p)$ by a contact $3$-handle attachment. Hence, there is a canonical isotopy class of contactomorphisms \[f:Y(p) \to M,\]  where $M$ is the sutured contact manifold obtained as  the contact connected sum of $(Y,\xi)$ with $(B^3,S^1,\xi_{std})$ at the point $p$. Suppose \[(\data = (Y_{B^3},R,r,m,\eta),\bar\xi)\] is a marked contact closure of  $(B^3,S^1,\xi_{std})$. This naturally gives rise to a marked contact closure \[(\data_M= (Y\#Y_{B^3},R,r,m_M,\eta),\xi\#\bar\xi)\] of $M$, where $m_M$ is the obvious extension of $m$. Let \[(W,\omega):(Y,\xi)\sqcup (Y_{B^3},\bar\xi)\to(Y\#Y_{B^3},\xi\#\bar\xi)\] be the natural Stein $1$-handle cobordism, and let $\nu\subset W$ be a  cylindrical cobordism from $r(\eta)\subset Y_{B^3}$ to $r(\eta)\subset Y\# Y_{B^3}$.  The map  \begin{equation}\label{eqn:mapW}\HMtoc(W|{-}R;\Gamma_{-\nu}):\SHMt(-\data_M)\to\HMtoc(-Y)\otimes_{\mathbb{Z}}\SHMt(-\data)\end{equation} sends $\invt(\data_M,\xi\#\bar\xi)$ to $\psi(Y,\xi)\otimes \invt(\data,\bar\xi)$, up to multiplication by a unit in $\RR$, by Theorem \ref{thm:ht-functoriality}. Let \[F_{\#}:\SHMtfun(-M)\to\HMtoc(-Y)\otimes_{\mathbb{Z}}\SHMtfun(-B^3,-S^1)\] be the morphism determined by the equivalence class of this map. Since $\SHMtfun(-B^3,-S^1)\cong \RR$, we may define $F_p$ to be the composition \[F_p =F_{\#}\circ \SHMtfun(f).\] We then have the following, which proves Proposition \ref{prop:darboux-complement}.

\begin{proposition}
\label{prop:Fp}
$F_p(\invt^g(Y(p))) = \psi(Y,\xi)\otimes \mathbf{1}$ for each $g\geq g(Y(p))=2$. \qed
\end{proposition}

\begin{remark}
For the map in \eqref{eqn:mapW}, we are viewing $W$ as a cobordism with one incoming and two outgoing boundary components. Reducible monopoles  make defining  maps associated to cobordisms with multiple incoming or outgoing boundary components difficult. This difficulty is typically overcome by restricting to nontorsion $\Sc$ structures on the boundary. In \eqref{eqn:mapW}, however, we are not restricting the $\Sc$ structures on $Y$.  Fortunately for us, Bloom has recently worked out the combinatorics needed to define maps on $\HMtoc$ associated to cobordisms with a single incoming boundary component and multiple outgoing boundary components \cite{bloom4}. 
\end{remark}

\subsubsection{The relative Giroux correspondence and $\invt(M,\Gamma,\xi)$}
\label{sssec:indepgenus}

Below, we use the ``existence" part of the relative Giroux correspondence between partial open books and sutured contact manifolds, together with our contact $2$-handle attachment maps, to prove Theorem \ref{thm:well-defined2}. Our discussion of  this correspondence   differs slightly in style but not in substance from the discussions in \cite{etguozbagcirelative,hkm4}.

\begin{definition} A \emph{partial open book} is a quadruple $(S,P,h,\mathbf{c})$, where: 
\begin{enumerate}
\item $S$ is a surface with nonempty boundary, 
\item $P$ is a subsurface of $S$,  
\item $h:P\to S$ is an embedding which restricts to the identity on $\partial P\cap \partial S$, 
\item $\mathbf{c}=\{c_1,\dots,c_n\}$ is a set of disjoint, properly embedded arcs in $P$ such that $S\ssm \mathbf{c}$ deformation retracts onto $S\ssm P$.
\end{enumerate}
\end{definition}

\begin{remark}
The collection $\mathbf{c}$ of \emph{basis} arcs for $P$ is not typically recorded in the data of a partial open book. Usually, it is just required that $S$ be obtained from $\overline{S\ssm P}$ by successive  $1$-handle attachments. The basis arcs specify a $1$-handle decomposition of $P$.
\end{remark}

Let $H(S)$ be the  product sutured contact handlebody obtained from $(S\times[-1,1],\partial S\times\{0\},\xi_{S})$ by rounding corners, as defined in  Subsection \ref{ssec:closureexamples}. 
Let $\gamma_i$ be the curve on $\partial H(S)$ corresponding to  \begin{equation}\label{eqn:basishandle}(c_i\times\{1\})\cup (\partial c_i\times [-1,1])\cup (h(c_i)\times\{-1\})\,\subset\, \partial (S\times[-1,1]).\end{equation}  Let $M(S,P,h,\mathbf{c})$ be the sutured contact manifold obtained from $H(S)$ by attaching contact $2$-handles along the curves in \begin{equation}\label{eqn:gammac}\boldsymbol{\gamma}(h,\mathbf{c}):=\{\gamma_1,\dots,\gamma_n\}.\end{equation}


\begin{definition}
A \emph{partial open book decomposition} of $(M,\Gamma,\xi)$ is a partial open book $(S,P,h,\mathbf{c})$ together with a contactomorphism \[f:M(S,P,h,\mathbf{c})\to (M,\Gamma,\xi).\]
\end{definition}

The theorem below represents the ``existence" part of the relative Giroux correspondence between partial open books and sutured contact manifolds, proven by Honda, Kazez, and Mati{\'c} in \cite{hkm4}.

\begin{theorem}
\label{thm:relativegiroux1}
Every sutured contact manifold admits a partial open book decomposition.
\end{theorem}

\begin{definition} We define $N(M,\Gamma,\xi)$ to be the minimum of \[\{g(H(S)) = \max\{2,g(S)+|\partial S|\}\}\] over all partial open book decompositions $(S,P,h,\mathbf{c},f)$ of $(M,\Gamma,\xi)$. This is the constant  in Theorem \ref{thm:well-defined2}.
\end{definition}

\begin{proof}[Proof of Theorem \ref{thm:well-defined2}]
Let $(\data,\bar\xi)$ and $(\data',\bar\xi')$ be marked contact closures of $(M,\Gamma,\xi)$ with genus at least $N(M,\Gamma,\xi)$. It suffices to show that  \begin{equation}\label{eqn:fundamentaleq}\Psit_{-\data,-\data'}(\invt(\data,\bar\xi)) \doteq \invt(\data',\bar\xi').\end{equation} Suppose $(S,P,h,\mathbf{c},f)$ is a partial open book for $(M,\Gamma,\xi)$ with $g(H(S)) = N(M,\Gamma,\xi)$. Let $(\data_f,\bar\xi)$ and $(\data'_f,\bar\xi')$ be the induced marked contact closures of $M(S,P,h,\mathbf{c})$. Let $(\data_S,\bar\xi_S)$ and $(\data'_S,\bar\xi'_S)$ be marked contact closures of $H(S)$ with \begin{equation}\label{eqn:gg'}g(\data_S) = g(\data)=g(\data_f)\,\,\,\,{\rm and }\,\,\,\, g(\data'_S) = g(\data')=g(\data'_f).\end{equation} Since $M(S,P,h,\mathbf{c})$ is obtained from $H(S)$ by attaching contact $2$-handles, there is a morphism \[\mathscr{H}:\SHMtfun(-H(S))\to \SHMtfun(-M(S,P,h,\mathbf{c}))\] obtained by composing the corresponding $2$-handle morphisms  defined in Subsubsection \ref{sssec:2handles}.  Let 
\begin{align*}
\mathscr{H}_{-\data_S,-\data_f}&:\SHMt(-\data_S)\to\SHMt(-\data_f)\\
\mathscr{H}_{-\data_S',-\data_f'}&:\SHMt(-\data_S')\to\SHMt(-\data_f')
\end{align*}
be the induced (equivalence classes of) maps. Then the diagram 
\[ \xymatrix@C=30pt@R=32pt{
\SHMt(-\data_S) \ar[rr]^-{\mathscr{H}_{-\data_S,-\data_f}} \ar[d]^{\Psit_{-\data_S,-\data_S'}} && \SHMt(-\data_f) \ar[d]^{\Psit_{-\data_f,-\data_f'}}\ar[rr]^-{id_{-\data_f,-\data}} &&\SHMt(-\data)\ar[d]^{\Psit_{-\data,-\data'}} \\
\SHMt(-\data_S') \ar[rr]_-{\mathscr{H}_{-\data_S',-\data_f'}} &&  \SHMt(-\data_f')\ar[rr]_{id_{-\data_f',-\data'}} &&\SHMt(-\data')
}\] commutes, up to multiplication by a unit in $\RR$. Furthermore, the genus equalities in \eqref{eqn:gg'}, combined with Proposition \ref{prop:H2}, imply that 
\begin{align*}
\mathscr{H}_{-\data_S,-\data_f}(\invt(\data_S,\bar\xi_S))&\doteq \invt(\data_f,\bar\xi)\\
\mathscr{H}_{-\data_S',-\data_f'}(\invt(\data_S',\bar\xi_S'))&\doteq \invt(\data_f',\bar\xi').
\end{align*}
We know that \[\Psit_{-\data_S,-\data_S'}(\invt(\data_S,\bar\xi_S))\doteq \invt(\data_S',\bar\xi_S')\] since these two contact classes generate \[\SHMt(-\data_S)\cong\SHMt(-\data_S')\cong\RR,\] by Proposition \ref{prop:handlebody-invt}. The commutativity of the leftmost square in the diagram above then implies that  \[\Psit_{-\data_f,-\data_f'}(\invt(\data_f,\bar\xi))\doteq \invt(\data_f',\bar\xi').\] But this fact, combined with  the commutativity of the rightmost square and the obvious equalities 
\begin{align*}
id_{-\data_f,-\data}(\invt(\data_f,\bar\xi))&= \invt(\data,\bar\xi)\\
id_{-\data_f',-\data'}(\invt(\data_f',\bar\xi'))&= \invt(\data',\bar\xi'),
\end{align*}
 implies \eqref{eqn:fundamentaleq}.
\end{proof}

In particular, Theorem \ref{thm:well-defined2} implies that the elements $\invt^g(M,\Gamma,\xi)$ are equal for all $g\geq N(M,\Gamma,\xi)$. As in Definition \ref{def:contactinvariantuniv}, we denote this common element by \[\invt(M,\Gamma,\xi)\in\SHMtfun(-M,-\Gamma).\] The following are then immediate corollaries of Propositions \ref{prop:H0}, \ref{prop:H1}, \ref{prop:Fgamma}, \ref{prop:H2}, \ref{prop:H3}, and \ref{prop:Fp}. We refer to those propositions for the notation.

\begin{corollary}
\label{cor:handles} For $i=0,\dots,3$, the morphism  \[\mathscr{H}_i:\SHMtfun(-M,-\Gamma)\to\SHMtfun(-M_i,-\Gamma_i)\]   sends $\invt(M,\Gamma,\xi)$ to $\invt(M_i,\Gamma_i,\xi_i)$. \qed
\end{corollary}

\begin{corollary}
\label{cor:Fgamma'} The morphism \[F_{K}:\SHMtfun(-M,-\Gamma)\to\SHMtfun(-M',-\Gamma')\] sends 
$\invt(M,\Gamma,\xi)$ to  $\invt(M',\Gamma',\xi')$. \qed
\end{corollary} 

\begin{corollary}
\label{cor:Fp'} The morphism
\[F_p:\SHMtfun(-Y(p)) \to \HMtoc(-Y) \otimes_\mathbb{Z} \RR\]  sends $\invt(Y(p))$ to $\psi(\xi)\otimes \mathbf{1}$. \qed
\end{corollary}

The following  corollary provides the inspiration for our construction in \cite{bsSHI} of a contact invariant in sutured instanton homology.

\begin{corollary}
\label{cor:alternatebasis}
Suppose $(S,P,h,\mathbf{c},f)$ is a partial open book decomposition of $(M,\Gamma,\xi)$. Let \[\mathscr{H}:\SHMtfun(-H(S))\to \SHMtfun(-M(S,P,h,\mathbf{c}))\] be the composition of contact 2-handle morphisms associated to $\mathbf{c}$. Then \[\invt(M,\Gamma,\xi) = \SHMtfun(f)(\mathscr{H}(\mathbf{1}))\in \SHMtfun(-M,-\Gamma),\]
where $\mathbf{1}$ is the generator of $\SHMtfun(-H(S))\cong \RR$. \qed
\end{corollary}

Suppose $(M,\Gamma)$ is  a \emph{sutured submanifold} of $(M',\Gamma')$, as defined in \cite{hkm5}. Let $\xi$ be a contact structure on $M'\ssm \inr(M)$ with convex boundary and dividing set $\Gamma$ on $\partial M$ and $\Gamma'$ on $\partial M'$. The sutured contact manifold $(M'\ssm \inr(M),\Gamma \cup \Gamma',\xi')$ can be obtained from a vertically invariant contact structure on $\partial M\times I$ by attaching contact handles. Given a contact handle decomposition $H$ of this sort, we  define  \[\Phi_{\xi,H}:\SHMtfun(-M,-\Gamma)\to\SHMtfun(-M',-\Gamma')\] to be the  corresponding composition of  contact handle attachment maps, as in the introduction. Note that if $\xi_M$ is a contact structure on $M$ which agrees with $\xi$ near $\partial M$, then \[\Phi_{\xi,H}(\invt(M,\Gamma,\xi_M)) = \invt(M',\Gamma',\xi_M\cup \xi)\] by Corollary \ref{cor:handles}. 

\begin{corollary}
\label{cor:contactsubmfld} If $(M,\Gamma,\xi)$ embeds into $(M',\Gamma',\xi')$ as a sutured contact submanifold     and $\invt(M,\Gamma,\xi) = 0$, then $\invt(M',\Gamma',\xi')=0$. \qed
\end{corollary}

We can use Corollary \ref{cor:contactsubmfld} to prove the following slightly weaker version of Theorem \ref{thm:ot}  without relying on the fact  that the monopole Floer invariant vanishes for overtwisted contact structures on closed 3-manifolds (Theorem \ref{thm:hm-ot}).

\begin{lemma}
\label{lem:ot2}
If $(M,\Gamma,\xi)$ is overtwisted, then $\invt(M,\Gamma,\xi)=0$.
\end{lemma}

\begin{proof}
By Corollary \ref{cor:contactsubmfld}, it is enough to show that a standard neighborhood $(M,\Gamma,\xi)$ of an overtwisted disk  has vanishing invariant. In \cite{hkm4}, Honda, Kazez, and Mati{\'c} describe a partial open book for $(S,P,h,\mathbf{c})$ for $(M,\Gamma,\xi)$ in which $S$ is an annulus, $\mathbf{c}$ consists of a single boundary parallel arc $c$, and $h(c)$ is another boundary parallel arc such that $c\cup h(c)$ is homotopic to a core curve $\alpha$ of the annulus $S$. As usual, we let $\gamma$ be the curve on $\partial H(S)$ corresponding to  \[(c\times\{1\})\cup (\partial c\times [-1,1])\cup (h(c)\times\{-1\}) \subset S\times[-1,1].\] Then $M(S,P,h,\mathbf{c})$ is obtained from $H(S)$ by attaching a contact 2-handle along $\gamma$. Let \[\mathscr{H}:\SHMtfun(-H(S))\to \SHMtfun(-M(S,P,h,\mathbf{c}))\] be the corresponding map. By Corollary \ref{cor:alternatebasis}, it suffices to show that $\mathscr{H}\equiv 0$. In fact, we will show that $\SHMtfun(-M(S,P,h,\mathbf{c}))=0$. 

To see this, let $\data = (Y,R,r,m,\eta)$ be any marked closure of $H(S)$. Let $\gamma'$ be a parallel copy of $\gamma$ in the interior of $Y$ and let $Y'$ be the result of $0$-surgery on $m(\gamma')$ with respect to the framing induced by $\partial H(S)$. By the construction of the contact 2-handle  map in the previous section, we know that there is an embedding \[m':M(S,P,h,\mathbf{c})\to Y'\] such that $\data' = (Y',R,r,m',\eta')$ is a marked closure of $M(S,P,h,\mathbf{c})$. Note that $\gamma$ is isotopic to the curve $\alpha'\subset \partial H(S)$ corresponding to $\alpha\times\{1\} \subset S\times[-1,1]$, by an isotopy which sends the $\partial H(S)$-framing on $\gamma$ to that of $\alpha'$, as depicted in Figure \ref{fig:OTdisk}. The image $m(\alpha')$ is isotopic to $r(a\times\{t\})$ for some embedded curve $a\subset R$ and any $t\in[0,1]$, by an isotopy which sends the $\partial H(S)$-framing on $m(\alpha')$ to the $r(R\times\{t\})$-framing on $r(a\times\{t\})$. We can therefore view $Y'$ as obtained from $Y$ by $0$-surgery on $r(a\times\{t\})$. Since $r(a\times\{t\})$ compresses $r(R\times\{t\})$, the surface $r(-R\times\{0\})\subset -Y'$ is homologous to a surface of genus $g(R)-1$. By the adjunction inequality in monopole Floer homology \cite{kmbook}, this implies that $\HMtoc(-Y',\spc;\Gamma_{\eta})=0$ whenever \[|\langle c_1(\spc), [r(-R\times\{0\})]\rangle| = 2g(R)-2.\] In particular, $\SHMt(-\data')=0$, which implies that $\SHMtfun(-M(S,P,h,\mathbf{c}))=0$.

\begin{figure}[ht]
\labellist
\hair 2pt\tiny
\pinlabel $c$ at 14 118
\pinlabel $h(c)$ at 156 118
\pinlabel $\alpha'$ at 569 100
\pinlabel $\gamma$ at 418 60
\endlabellist
\centering
\includegraphics[width=13cm]{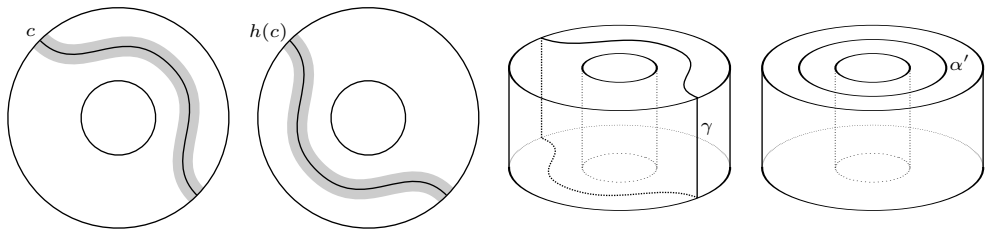}
\caption{A partial open book  for a standard neighborhood of an overtwisted disk.  The shaded regions represent $P$ and $h(P)$. The two rightmost diagrams show the curves $\gamma$ and $\alpha'$ in $\partial H(S)$, drawn as $(c\times\{1\})\cup (\partial c\times [-1,1])\cup (h(c)\times\{-1\})$ and $\alpha\times\{1\}$ in $S\times[-1,1]$.}
\label{fig:OTdisk}
\end{figure}

\end{proof}

\begin{remark}
With a bit of work, one should similarly be able to use Corollary \ref{cor:contactsubmfld} to prove that if $(M,\Gamma,\xi)$ has positive Giroux torsion, then $\invt(M,\Gamma,\xi)=0$ in analogy with \cite{ghihonvh}.
\end{remark}

\begin{remark} In \cite{hkm5}, Honda, Kazez, and Mati{\'c} define a map  similar to $\Phi_{\xi,H}$ which depends only on $\xi$. We expect that our map $\Phi_{\xi,H}$ is likewise independent of $H$, as in Conjecture \ref{conj:PhiHwd}.
\end{remark}

\section{The bypass exact triangle}
\label{sec:bypasstri}
In this section, we  work over the Novikov field $\RR/2\RR := \RR\otimes_{\mathbb{Z}} \Z/2\Z$ in order to use the surgery exact triangle in monopole Floer homology (see Remark \ref{rmk:novfield}). The results of the previous sections, including the construction and invariance of $\invt(M,\Gamma,\xi)$ and the definition of the contact handle attachment maps, hold over $\RR/2\RR$ without modification.

Suppose $(M,\Gamma)$ is a sutured manifold and  $\alpha\subset \partial M$ is an arc which intersects $\Gamma$ in three points, including both endpoints of $ \alpha$. A \emph{bypass move} along $\alpha$ replaces $\Gamma$ with a new set of sutures $\Gamma'$ which differ from $\Gamma$ in a neighborhood of $\alpha$, as shown in Figure \ref{fig:bypass-move}. 

\begin{figure}[ht]
\labellist
\small \hair 2pt
\pinlabel $\alpha$ at 22 28
\endlabellist
\centering
\includegraphics[width=5.6cm]{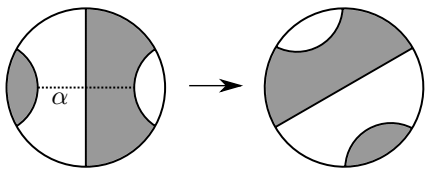}
\caption{A bypass move along the arc $\alpha$.}
\label{fig:bypass-move}
\end{figure}

If $\Gamma$ is the dividing set of a contact structure $\xi$ on $M$, then a bypass move is achieved by attaching an actual \emph{bypass} along $\alpha$, as defined by Honda in \cite{honda2}.  In \cite{ozbagci}, {\"O}zba{\u{g}}c{\i} observed  that attaching a bypass along $\alpha$ is equivalent to first attaching a contact $1$-handle  along disks in $\partial M$ centered at the endpoints of $ \alpha$ and then attaching a contact $2$-handle along the union $\beta$ of $\alpha$ with an arc  on the boundary of this $1$-handle, as shown in Figure \ref{fig:bypass-handles}. There is a canonical isotopy class of diffeomorphisms between the resulting manifold and $M$ which restrict to the identity outside a neighborhood of these  handles. A bypass move along $\alpha$ thus gives rise to a morphism \[\mathscr{H}_\alpha:\SHMtfun(-M,-\Gamma)\to \SHMtfun(-M,-\Gamma')\] which is the composition of the corresponding contact $1$- and $2$-handle maps with the map induced by this isotopy class of diffeomorphisms.  Corollary \ref{cor:handles} implies the following.

\begin{figure}[ht]
\labellist
\small \hair 2pt
\pinlabel $\alpha$ at 51 57
\pinlabel $\beta$ at 242 39
\endlabellist

\centering
\includegraphics[width=10cm]{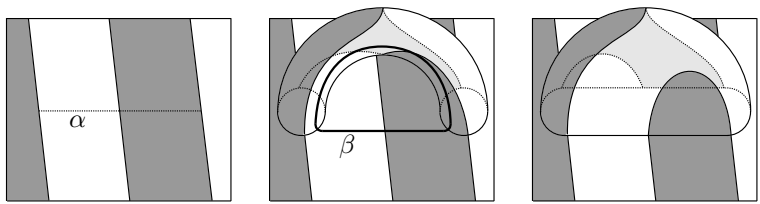}
\caption{Performing a bypass move by attaching a contact $1$-handle at the endpoints of $\alpha$ and a contact $2$-handle along $\beta$.}
\label{fig:bypass-handles}
\end{figure}

\begin{proposition}
\label{prop:bypass}
Suppose $(M,\Gamma',\xi')$ is obtained from $(M,\Gamma,\xi)$ by attaching a bypass along $\alpha$ and pulling back the resulting contact structure to $M$ by the canonical isotopy class of diffeomorphisms. Then the induced  map $\mathscr{H}_\alpha$ sends $\invt(M,\Gamma,\xi)$ to $\invt(M,\Gamma',\xi')$.
\end{proposition}

Figure \ref{fig:bypass-triangle} shows a   sequence of bypass moves, performed in some fixed neighborhood in $\partial M$, resulting in a 3-periodic sequence of sutures on $M$. Such a sequence of bypass moves is what Honda calls a \emph{bypass triangle}. Work-in-progress of Honda shows that a bypass triangle gives rise to a \emph{bypass exact triangle} in sutured (Heegaard) Floer homology. The main result of this section is the analogous result in the monopole Floer setting, per the theorem below. 

\begin{figure}[ht]
\labellist
\small \hair 2pt
\pinlabel $\alpha_1$  at 24 118
\pinlabel $\alpha_2$  at 148 117
\pinlabel $\alpha_3$  at 96 38
\pinlabel $\Gamma_1$  at -3 145
\pinlabel $\Gamma_2$  at 173 144
\pinlabel $\Gamma_3$  at 85 -8
\endlabellist
\centering
\includegraphics[width=5.6cm]{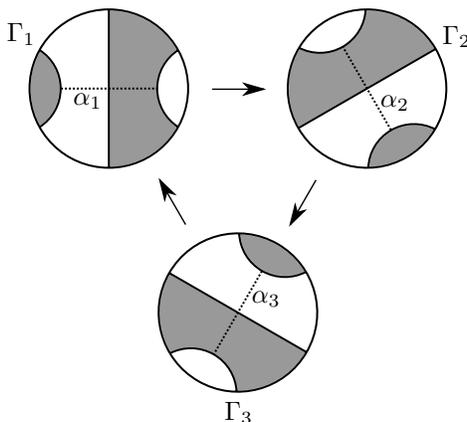}
\caption{The bypass triangle.  Each picture shows the attaching arc used to achieve the next set of sutures in the triangle.}
\label{fig:bypass-triangle}
\end{figure}

\begin{theorem}
\label{thm:bypass}
Suppose $\Gamma_1,\Gamma_2,\Gamma_3\subset \partial M$ is the 3-periodic sequence of sutures resulting from successive bypass moves along arcs $\alpha_1,\alpha_2,\alpha_3$ as   in Figure \ref{fig:bypass-triangle}. Then the maps $\mathscr{H}_{\alpha_1},\mathscr{H}_{\alpha_2},\mathscr{H}_{\alpha_3}$ fit into an exact triangle
\[ \xymatrix@C=-25pt@R=35pt{
\SHMtfun(-M,-\Gamma_1) \ar[rr]^{\mathscr{H}_{\alpha_1}} & & \SHMtfun(-M,-\Gamma_2) \ar[dl]^{\mathscr{H}_{\alpha_2}} \\
& \SHMtfun(-M,-\Gamma_3). \ar[ul]^{\mathscr{H}_{\alpha_3}} & \\
} \]
\end{theorem}

\begin{proof}
We will prove Theorem \ref{thm:bypass} by realizing the bypass exact triangle  as the usual surgery exact triangle in monopole Floer homology. 

Note that by enlarging our local picture slightly, we can think of the arcs $\alpha_1,\alpha_2,\alpha_3$ as being arranged as in Figure \ref{fig:bypass-arcs} with respect to $\Gamma_1$. We may therefore view \[(M,\Gamma_2)\,\,\,{\rm and} \,\,\,(M,\Gamma_3)\,\,\, {\rm and}\,\,\,(M,\Gamma_1)\] as being obtained from $(M,\Gamma_1)$ by attaching bypasses along the arcs \[\alpha_1\,\,\,{\rm and}\,\,\,\alpha_1,\alpha_2\,\,\, {\rm and}\,\,\, \alpha_1,\alpha_2, \alpha_3,\] respectively. As described above, attaching a bypass along $\alpha_i$ amounts to attaching a contact $1$-handle $H_i$ along disks centered at the endpoints of $\alpha_i$ and then attaching a contact $2$-handle along a curve $\beta_i$ which extends $\alpha_i$ over the handle. Let $(Z_1,\gamma_1)$ be the sutured manifold obtained by attaching all three $H_1,H_2,H_3$ to $(M,\Gamma_1)$, as  in Figure \ref{fig:bypass-setup}. We will  view $\beta_1,\beta_2,\beta_3$ as curves in $\partial Z_1$, as shown in the figure. For $i=1,2,3$, let $(Z_{i+1},\gamma_{i+1})$ be the result of attaching a contact $2$-handle to $(Z_i,\gamma_i)$ along $\beta_i$. We thus have the following canonical (up to isotopy) identifications:
\begin{align*} 
(Z_1,\gamma_1)&\cong (M,\Gamma_1)\cup H_1\cup H_2\cup H_3\\
(Z_2,\gamma_2)&\cong (M,\Gamma_2)\cup H_2\cup H_3\\
(Z_3,\gamma_3)&\cong (M,\Gamma_3)\cup H_3\\
(Z_4,\gamma_4)&\cong (M,\Gamma_1).\\
\end{align*}

\begin{figure}[ht]
\labellist
\small \hair 2pt
\pinlabel $\alpha_1$ at 17 35
\pinlabel $\alpha_2$ at 32 51
\pinlabel $\alpha_3$ at 48 67
\endlabellist
\centering
\includegraphics[width=4cm]{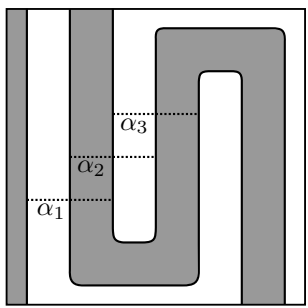}
\caption{Another view of the arcs of attachment for the bypasses in the triangle. The suture drawn here is $\Gamma_1$.}
\label{fig:bypass-arcs}
\end{figure}

\begin{figure}[ht]
\labellist
\small \hair 2pt
\pinlabel $\beta_1$ at 416 63
\pinlabel $\beta_2$ at 461 165
\pinlabel $\beta_3$ at 507 265
\pinlabel $a$ [B] at 974 95
\pinlabel $b$ [B] at 915 97
\endlabellist
\centering
\includegraphics[width=15cm]{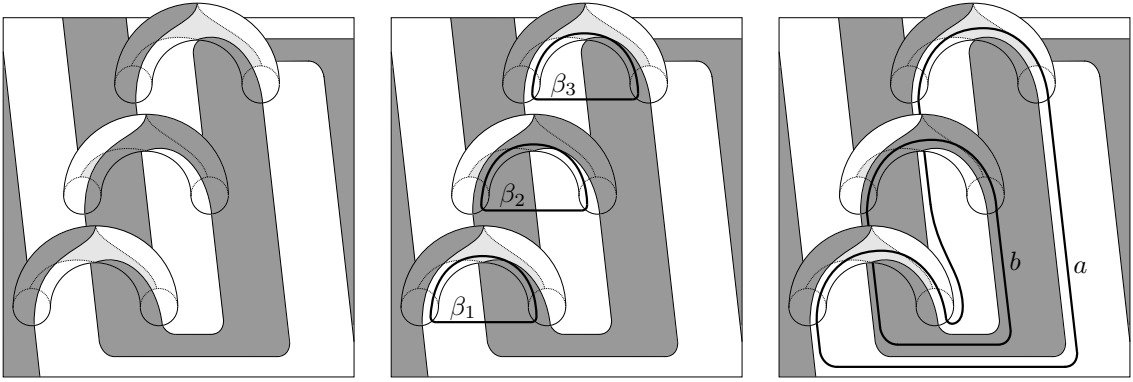}
\caption{A view of $(Z_1,\gamma_1)$, obtained by attaching the contact 1-handles $H_1,H_2,H_3$ to $(M,\Gamma_1).$ Middle, the attaching curves $\beta_1,\beta_2,\beta_3$ for the contact $2$-handles. Right, the curves $a$ and $b$.}
\label{fig:bypass-setup}
\end{figure}

Recall  that contact $1$-handle attachment has little effect on the level of closures. Specifically, if $\data=(Y,R,r,m,\eta)$ is a marked closure of sutured manifold after the $1$-handle attachment, then there is a marked closure of the sutured manifold before the $1$-handle attachment of the form $\data'=(Y,R,r,m',\eta)$. The corresponding $1$-handle attachment morphism is induced by the identity map from $\SHMt(-\data')$ to $\SHMt(-\data)$. 
We thus have canonical isomorphisms
\[\SHMtfun(-Z_i,-\gamma_i)\cong \SHMtfun(-M,-\Gamma_i),\] for $i=1,2,3,4$, where the subscript of $\Gamma_i$ is taken mod $3$. In particular, $\SHMtfun(-Z_4,-\gamma_4)$ is canonically identified with $\SHMtfun(-Z_1,-\gamma_1)$.
Therefore, to prove Theorem \ref{thm:bypass}, it  suffices to prove that there is an exact triangle 
\[ \xymatrix@C=-25pt@R=35pt{
\SHMtfun(-Z_1,-\gamma_1) \ar[rr]^{\mathscr{H}_{\beta_1}} & & \SHMtfun(-Z_2,-\gamma_2) \ar[dl]^{\mathscr{H}_{\beta_2}} \\
& \SHMtfun(-Z_3,-\gamma_3), \ar[ul]^{\mathscr{H}_{\beta_3}} & \\
} \] where $\mathscr{H}_{\beta_i}$ is the morphism associated to contact $2$-handle attachment along $\beta_i$.

Recall that on the level of closures,  contact $2$-handle attachment corresponds to surgery. Specifically,  if $\data_i = (Y_i,R,r_i,m_i,\eta)$ is a marked closure of $(Z_i,\gamma_i)$, then there is a marked closure of  $(Z_{i+1},\gamma_{i+1})$ of the form $\data_{i+1} = (Y_{i+1},R,r_{i+1},m_{i+1},\eta)$, where $Y_{i+1}$ is the result  of $0$-surgery on $m_i(\beta_i')$ with respect to the $(\partial Z_i)$-framing, where $\beta_i'$ is a pushoff of $\beta_i$ into the interior of $Z_i$. The morphism $\mathscr{H}_{\beta_i}$ is induced by the $2$-handle cobordism map
\[F_i:=\HMtoc({-}W_i|{-}R;\Gamma_{-\nu}):\SHMt(-\data_i)\to\SHMt(-\data_{i+1})\] corresponding to this surgery. So, to prove the exact triangle above, and, therefore, Theorem \ref{thm:bypass}, it  suffices to find a closure $\data_1$ of $(Z_1,\gamma_1)$ such that the surgeries relating the $-Y_i$ are exactly those that one encounters in the usual surgery exact triangle in monopole Floer homology. More precisely, it suffices to arrange that:
\begin{itemize}
 \item $F_1$ is the map associated to $0$-surgery on some  $K\subset-Y_1$,
 \item $F_2$ is the map associated to $(-1)$-surgery on a meridian $\mu_1\subset -Y_2$ of $K$,
 \item $F_3$ is the map associated to $(-1)$-surgery on a meridian $\mu_2\subset -Y_3$ of $\mu_1$.
  \end{itemize}
  
Let $\data = (Y,R,r,m,\eta)$ be a marked closure of $(Z_1,\gamma_1)$. Let $a$ and $b$   be embedded curves in the positive and negative regions of $\partial Z_1$ as shown in Figure \ref{fig:bypass-setup}. Since neither curve intersects $\gamma_1$, we can assume that $a$ and $b$ are contained in $r(R\times\{-1\})$ and $r(R\times\{+1\})$, respectively. Let $Y_1$ be the manifold obtained from $Y$ by performing $(+1)$-surgeries on pushoffs $a'$ and $b'$ of $a$ and $b$ into the interior of $r(R\times[-1,1])$, with respect to their $\partial Z_1$-framings. Then $\data_1 = (Y_1,R,r_1,m_1,\eta)$ is a marked closure of $(Z_1,\gamma_1)$, where $m_1$ is the embedding induced by $m$ and $r_1$ is the canonical (up to isotopy) embedding induced by $r$. For ease of notation, we will think of the $m_i$ as being inclusions, and simply write $x$ for $m_i(x)$ for points $x\in Z_i$. In particular, $W_i$ is the 2-handle cobordism corresponding to $0$-surgery on $\beta_i'\subset Y_i$. 

\begin{figure}[ht]
\labellist
\small \hair 2pt
\pinlabel $\beta_2''$ at 457 515
\pinlabel $\beta_3''$ at 517 230
\endlabellist
\centering
\includegraphics[width=14cm]{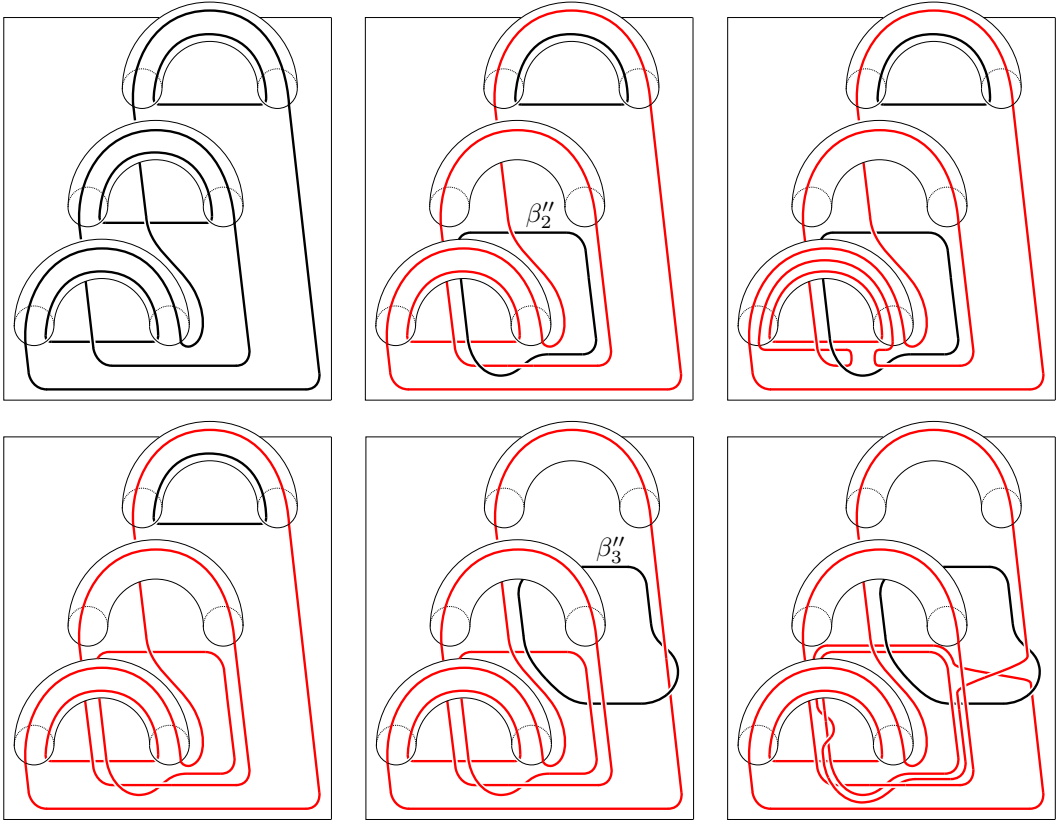}
\caption{Top-left, the curves $\beta_1',\beta_2',\beta_3',a',b'$ in a neighborhood of $\partial Z_1\subset Y$. Red indicates   curves that have been surgered along. Top-middle, sliding $\beta'_2\subset Y_2$ over $b'$ to produce $\beta_2''$. Top-right, showing that $\beta_2''$ bounds a meridional disk of $\beta_1'$ disjoint from the other surgery curves. Bottom-left, $\beta_3'\subset Y_3$. Bottom-middle, sliding $\beta'_3$ over $a'$ and $\beta_1'$ to produce $\beta_3''$. Bottom-right, showing that $\beta_3''$ bounds a meridional disk of $\beta_2''$ disjoint from the other surgery curves.}
\label{fig:handleslide1}
\end{figure}

Recall that $Y_2$ is the result of $0$-surgery on $\beta_1'\subset Y_1$. Let $\beta_2''$ be the curve in $Y_2$ obtained by handlesliding $\beta_2'$ across the surgered  curve $b'$, as shown in the top-middle of Figure \ref{fig:handleslide1}. Note that the $0$-framing on $\beta_2'$ corresponds to the $(+1)$-framing on $\beta_2''$ under this isotopy. We may therefore think of $Y_3$ as the result of $(+1)$-surgery on $\beta_2''$ and $W_2$ as the corresponding 2-handle cobordism. We claim that $\beta_2''$ is a meridian of the surgered curve $\beta_1'$. This is apparent once we handleslide the surgered curve $b'$ over the surgered curve $\beta_1'$, as shown in the top-right of Figure \ref{fig:handleslide1}.

Let $\beta_3''$ be the curve in $Y_3$ obtained by handlesliding $\beta_3'$ over the surgered curve $a'$ and then over the surgered curve $\beta_1'$, as shown in the bottom-middle of Figure \ref{fig:handleslide1}. The $0$-framing on $\beta_3'$ corresponds to the $(+1)$-framing on $\beta_3''$ under this isotopy, so we may therefore think of $Y_4\cong Y_1$ as the result of $(+1)$-surgery on $\beta_3''$ and $W_3$ as the corresponding 2-handle cobordism. We claim that $\beta_3''$ is a meridian of the surgered curve $\beta_2''$. This is  apparent once we handleslide the surgered curve $a'$ over the surgered curve $\beta_2''$, as shown in the bottom-right of Figure \ref{fig:handleslide1}, noting that we are free to isotop $\beta_2''$ through the 1-handle $H_2$. It follows from these considerations that the maps $F_1,F_2,F_3$ are of the form described above for $K = \beta_1'$, $\mu_1 = \beta_2''$,  $\mu_2 = \beta_3''$.
This completes the proof of Theorem \ref{thm:bypass}.
\end{proof}

\bibliographystyle{hplain}
\bibliography{References}

\end{document}